\documentclass[reqno,english]{amsart}
\usepackage{amsfonts,amsmath,latexsym,verbatim,amscd,mathrsfs,color,array}
\usepackage[colorlinks=true]{hyperref}
\usepackage{amsmath,amssymb,amsthm,amsfonts,graphicx,color}
\usepackage{amssymb}
\usepackage{pdfsync}
\usepackage{epstopdf}
\usepackage{cite}

\usepackage{graphicx}

\newcommand{\R} {\mathbb R}

\newcommand{\cuad}{{\sqcap\kern-.68em\sqcup}}

\newcommand{\be}{\begin{equation}}
\newcommand{\ee}{\end{equation}}

\renewcommand{\Im}{\mathop{\rm Im}}

\newtheorem{lemma}{Lemma}[section]
\newtheorem{prop}{Proposition}[section]
\newtheorem{theorem}{Theorem}

\newtheorem{remark}{Remark}[section]
\newcommand{\bremark}{\begin{remark} \em}
\newcommand{\eremark}{\end{remark} }

\long\def\comment#1{\marginpar{\vtop{\raggedright\small$\bullet$\ #1}}}
\long\def\hide#1{}

\long\def\anot#1{\ \\{\bf \crr ANNOTATION.} {#1}}
\long\def\noanot#1{}
\definecolor{redd}{RGB}{200,0,0}

\def\crr{\color{red}}
\long\def\elim#1{{\color{red} ELIMINAR\\ #1}}

\def\crr{}
\long\def\comment#1{}
\long\def\anot#1{}
\long\def\elim#1{}

\long\def\comment#1{\marginpar{\raggedright\small$\bullet$\ #1}}
\numberwithin{equation}{section}
\title[Infinite time bubbling for the $SU(2)$ Yang-Mills heat flow on $\mathbb{R}^4$]{Infinite time bubbling for the $SU(2)$ Yang-Mills heat flow on $\mathbb{R}^4$}
\author[Y. Sire]{Yannick Sire}
\address{\noindent Department of Mathematics, Johns Hopkins University, 404 Krieger Hall, 3400 N. Charles Street, Baltimore, MD 21218, USA}
\email{ysire1@jhu.edu}
\author[J. Wei]{Juncheng Wei}
\address{\noindent Department of Mathematics, Chinese University of Hong Kong, Shatin, NT, Hong Kong}
\email{wei@math.cuhk.edu.hk}
\author[Y. Zheng]{Youquan Zheng}
\address{\noindent School of Mathematics, Tianjin University, Tianjin 300072, P. R. China}
\email{zhengyq@tju.edu.cn}
\begin{document}
\begin{abstract}
We investigate the long time behaviour of the Yang-Mills heat flow on the bundle $\mathbb{R}^4\times SU(2)$. Waldron \cite{Waldron2019}  proved global existence and smoothness of the flow on closed $4-$manifolds, leaving open the issue of the behaviour in infinite time. We exhibit two types of long-time bubbling: first we construct an initial data and a globally defined solution which {\sl blows-up} in infinite time at a given point in $\mathbb R^4$. Second, we prove the existence of {\sl bubble-tower} solutions, also in infinite time. This answers the basic dynamical properties of the heat flow of Yang-Mills connection in the critical dimension $4$ and shows in particular that in general one cannot expect that this gradient flow converges to a Yang-Mills connection. We emphasize that we do not assume for the first result any symmetry assumption; whereas the second result on the existence of the bubble-tower is in the $SO(4)$-equivariant class, but nevertheless new.
\end{abstract}
\maketitle
\tableofcontents
\section{Introduction}

It is a classical topic in differential geometry to relate and understand the interplay between the geometry of submanifolds and the theory of vector bundles. For example, in the seminal paper \cite{Tiangang2000}, Gang Tian exhibited a link between Yang-Mills connections, which are critical points of the square of the $L^2$-norm of the connection form on  a vector bundle and  {\sl calibrated minimal submanifolds}. In the present paper, motivated by recent global well-posedness results due to Waldron \cite{Waldron2019}, we investigate the long time behaviour of the Yang-Mills heat flow.

Let $E\to M$ be a vector bundle over a four dimensional Riemannian manifold without boundary, with compact Lie group $G$ as its structure group. Let $T^*M$ be the cotangent bundle  over $M$ and for $1 \leq p \leq 4 $, let $\Omega^p(M)$ be the bundle of $p-$forms on $M$ with $T^*M=\Omega^1(M)$.
 A connection $A$ on $E$ can be given by specifying a covariant derivative $D_A$ from $C^\infty(E)$ into $C^\infty (E\otimes \Omega^1(M))$. In a local trivialization of the vector bundle $E$, the covariant derivative $D_A$ writes
$$
D:=D_A=d+A_\alpha
$$
where $A=(A_\alpha)_{\alpha}$ is a section of $T^*M \otimes \mathfrak{g}$ where $\mathfrak{g}$ is the Lie algebra of $G$ embedded in a large unitary group, i.e. the  connection $A$ is a $\mathfrak{g}-$valued $1-$form.  The curvature $F_A$ of the connection $A$ is given by the tensor $D_A^2: \Omega^0(M)\to \Omega^2(M)$, which can be formally  written
$$
F_A=F:=dA+A\wedge A.
$$

For a connection $A$, the Yang-Mills functional is
$$
YM(A) = \frac{1}{2}\int_{M}|F_A|^2 dx.
$$
It is well known that the Euler-Lagrange equation of $YM$ is then
$$
D^*_AF_A=0
$$
where $D^*_A$ denotes the adjoint operator of $D_A$ with respect to the Killing form of $G$ and the metric on $M$. By the second Bianchi identity, it holds that
$$
D_AF_A=0.
$$
 A connection $A$ is Yang-Mills if and only if it is a critical point of $YM$, which then is equivalent to the equation
$$
D^*_AF_A = 0.
$$
In order to obtain Yang-Mills connections on any given bundle $E$, a natural approach is to deform a given connection along the negative gradient flow of $YM$ which is  given by the following evolution equation
\begin{equation}\label{e:Yangmillsheatflow}
\frac{\partial A}{\partial t} = -D_A^*F_A,
\end{equation}
starting from any initial connection $A_0$. This equation plays a fundamental role in Donaldson's work (see e.g. \cite{DonaldsonKronheimer}).

A gauge transformation is a (sufficiently smooth) map $S$ from $M$ into $G$. The gauge group acts on connections as
$$
S(A):=S\cdot  A\cdot S^{-1}-dS \cdot S^{-1}
$$
YM is gauge-invariant in the sense that $YM(S^*(D)) = YM(D)$ for any gauge transformation and any connection $D = d + A$. The Yang-Mills equation is therefore not elliptic, as the kernel of the linearized operator is infinite dimensionnal. Similarly, the evolution problem (\ref{e:Yangmillsheatflow}) is not parabolic, and the methods developed for parabolic equations cannot be directly applied to prove existence and uniqueness of solutions to (\ref{e:Yangmillsheatflow}). For further background material on Yang-Mills equations, we refer the interested readers to, for instance, \cite{DonaldsonKronheimer}, \cite{FreedUhlenbeck1991}, \cite{Jost1991},  \cite{Lawson1985}.

In the seminal work of Taubes \cite{Taubesjdg1982}, the Morse theory for Yang-Mills functional was established. In \cite{wanghongyujdg}, nonminimal solutions to the Yang-Mills equation with group $SU(2)$ on $\mathbb{S}^2\times\mathbb{S}^2$ and $\mathbb{S}^1\times \mathbb{S}^3$ are constructed.
In \cite{Tiangang2000}, Gang Tian was interested in a compactification of the moduli space of Yang-Mills connections, pursuing the search of geometric invariants. To do so, one needs to consider singular Yang-Mills connections, i.e. singular solutions of the PDE $D_AF_A=0$ on $M$. In four dimensions, it is known since the important work of Uhlenbeck \cite{Uh1,Uh2} that Yang-Mills connections are smooth up to a discrete set of points on $M$ and that those connections can be extended to the whole manifold, with a smaller $L^2$ norm of the curvature form. In higher dimensions, the picture is more complicated and Tian  \cite{Tiangang2000} proved that the blow-up set of Yang-Mills connections is closed and $H^{n-4}$ rectifiable where $H^m$ is the $m-$Hausdorff measure. Thanks to the monotonicity of the rescaled energy, one has the following bubbling phenomenon: given any sequence $A_i$ of Yang-Mills connections, $A_i$ converges up to a subsequence and modulo a gauge transformation to a Yang-Mills connection $A_\infty$ in the smooth topology outside of a closed set of codimension at least $4$. Furthermore, the energy concentrates in the sense of measures:
$$
|F_{A_k0}|^2dvol\rightharpoonup |F_{A_\infty}|^2dvol + \Theta dH^{n-4}|_{S}.
$$
The limiting connection $A_\infty$ is smooth on $M\setminus S$, $\Theta \geq 0$ is called the multiplicity and the set $S$ is the blow-up locus of $A_i$. The achievement of Tian is a deep understanding of the blow-up locus and hence of the natural compactification of the Yang-Mills connections in higher dimensions. He showed that the blow-up locus is $(n-4)$-rectifiable and if it arises as a special subclass of connections, then it is a closed calibrated integral minimizing current, namely the generalized mean curvature of $S$ is equal to 0, see also \cite{TaoTian2004}. We refers also the reader to the more recent work by Naber and Valtorta \cite{NaberValtorta2019}.

\vspace{0.5cm}

{\bf Long time behavior of the Yang-Mills heat flow}.
As far as the flow \eqref{e:Yangmillsheatflow} is concerned, the theory is much less developed than its elliptic version. In \cite{RadeJRAM1992}, the global existence and uniqueness of Yang-Mills flow over 2 or 3 dimensional manifolds were proved. In spatial dimensions greater than 4, finite time blow-up solutions were constructed in \cite{Naito1994}. The behaviour of the Yang-Mills flow on Riemannian manifolds of dimension four was not very well understood until recently. The foundational work of Struwe \cite{Struwe1994} gives a global weak solution with finitely many point singularities, in analogy with the harmonic map flow in dimension two. In \cite{SchlatterAGAG}, Schlatter gave the exact formulation, the proofs of the blow-up analysis and the long-time behaviour
of the Yang-Mills flow in Theorem 2.4 of \cite{Struwe1994}.  In \cite{SchlatterJRAM}, Schlatter also  proved the global existence of four dimensional Yang-Mills heat flow for small data. Recently, the global well-posedness for any initial data was established by Alex Waldron in \cite{Waldron2019} (see also \cite{Waldron2014, Waldron2016}). The asymptotic behaviour and the structure of the singular set for the Yang-Mills heat flow in dimensions $\geq 4$ was analyzed in \cite{HongTianMathAnn}.

It was already pointed out in \cite{GrotowskiShatah2007} that the Yang-Mills heat flow on $4-$manifolds behaves similarly as the degree 2 harmonic map heat flow. In \cite{SchlatterStruwe1998}, Schlatter and Struwe showed that the Yang-Mills heat flow of $SO(4)$-equivariant connections on a $SU(2)$-bundle over a ball in $\mathbb{R}^4$ admits a smooth solution for all times using the super/sub solution method for two dimensional harmonic map flow developed in \cite{ChangDingYe1992}. The infinite time bubbling under radially symmetric assumptions was proved in Chapter 4 of \cite{Waldron2014}.
\vspace{0.3cm}

In this paper, we prove the existence of infinite time blow-up solutions on the trivial bundle $\mathbb{R}^4\times SU(2)$ {\sl without symmetry assumptions}. Even more, we also prove the existence of bubble-tower solutions as $t\to +\infty$. To state our result, let us recall the well known BPST/ADHM instantons.
We identify the field of quaternions $\mathbb H$
$$
x = x_1 + x_2i + x_3 j + x_4k\in \mathbb H
$$
with elements of $\mathbb{R}^4$. Then the following algebraic properties hold:
$$
i^2 = j^2 = k^2 = -1,\quad ij = k = -ji,\quad jk = i = -kj,\quad ki = j = -ik
$$
and
$$
\bar{x} = x_1 -x_2 i - x_3 j - x_4 k, \quad |x|^2 = x_1^2 + x_2^2 + x_3^2 + x_4^2 = x\cdot\bar{x}, \quad Im x = x_2i + x_3 j + x_4k.
$$
It is well known that there is an isomorphism between the Lie algebra $su(2)$ of the structure group $SU(2)$ and $Im\, \mathbb H$. Considering $B(x) = Im(f(x, \bar{x})d\bar{x})$, it was pointed out by Polyakov that, when $f(x, \bar{x}) = \frac{x}{1+|x|^2}$, then $B$ is nontrivial self-dual instanton (a solution of the Yang-Mills equations) on the bundle $E = \mathbb{R}^4\times SU(2)$. In this case, one has
$$
B(x) = Im\left(\frac{x}{1+|x|^2}d\bar{x}\right), \quad F_B = \frac{dx\wedge d\bar{x}}{(1+|x|^2)^2}.
$$
See the references  \cite{AHDM} and \cite{BPST}.

Our first result is the infinite time bubbling at one point for the flow (\ref{e:Yangmillsheatflow}):
\begin{theorem}\label{t:main}
Let $q$ be a point in $\mathbb{R}^4$. There exist an initial datum $A_0(x)$, $A_0\in H^1(\mathbb{R}^4)\cap C(\mathbb{R}^4)$ and smooth functions $\xi(t)\to q$, $0<\mu(t)\to 0$, as $t\to +\infty$, such that the solution $A(x, t)$ to (\ref{e:Yangmillsheatflow}) has the following form modulo a gauge transformation,
\begin{equation}\label{solutionofmainproblem1}
A(x, t) = Im\left(\frac{x-q}{1+|x-q|^2}d\bar{x}\right)+Im\left(\eta_0\left(\frac{|x-q|}{\sqrt{\mu(t)}}\right)\frac{x-\xi(t)}{\mu(t)^2+|x-\xi(t)|^2}d\bar{x}\right) +\varphi(x, t),
\end{equation}
where $\eta_0(s)$ is a smooth cut-off function satisfying $\eta_0(s) = 1$ for $s < 1$ and $\eta_0(s) = 0$ for $s > 2$. As $t\to +\infty$, the differential 1-forms $\varphi(x, t)\to 0$ uniformly away from the blow-up point $q$. Moreover, the parameter $\mu(t)$ decays to $0$ exponentially.
\end{theorem}

In his thesis \cite{Waldron2014},  Waldron proved the existence of infinite time blowing-up solutions for (\ref{e:Yangmillsheatflow}) in the $SO(4)$-equivariant case. The method of \cite{Waldron2014} is based on the scheme of Raphael and Schweyer \cite{RaphaelSchweyercpam}. The proof of Theorem \ref{t:main} (and Theorem \ref{t:main-bubble-tower}  below) is based on the inner-outer parabolic gluing method developed in \cite{delPinoMussoJEMS} and \cite{DDW2020}; we do not need the $SO(4)$-equivariant assumption in Theorem \ref{t:main} and this is a main achievement of our paper . Furthermore, we would like to emphasize that it was believed that the Yang-Mills flow would be generally converging at $+\infty$ towards a Yang-Mills connection. These constructions show that this is generally not the case. Theorem \ref{t:main} is also valid in the multiple bubble case after minor modifications of the proof.

The heat flow is not the only relevant time-dependent equations for Yang-Mills connections. In a series of important papers, Oh and Tataru considered the energy-critical hyperbolic Yang-Mills flow where the heat operator is replaced by the wave one. They provide a complete picture of global-wellposedness {\sl vs} finite-time blow-up (the so-called Threshold conjecture). See \cite{OhTataruBAMS}, \cite{OhTatarupartI}, \cite{OhTatarupartII}, \cite{OhTataruCMP}, \cite{OhTatarupartIV} and references therein. Note that the solution in (\ref{solutionofmainproblem1}) has the same form as the one in Theorem 6.1 of \cite{OhTatarupartI}, Theorem 1.3 of \cite{SchlatterJRAM} and  Theorem 1.2 of \cite{SchlatterAGAG}.

More precisely, there exist sequences $R_k\searrow 0$, $x_k\to q$, $t_k\nearrow + \infty$, such that the solutions have the following asymptotic form
$$
A_k(x) = d + R_kA(x_k+R_kx, t_k)\to A_\infty, \quad k\to \infty,
$$
modulo a gauge transformation in $H^{1, 2}_{loc}$, where $A_\infty$ is a Yang-Mills connection on $\mathbb{R}^4$.

\vspace{0.5cm}

{\bf The bubble tower solutions of Yang-Mills heat flow.}
We also construct a completely new solution in large times for the flow (\ref{e:Yangmillsheatflow}). Now we restrict ourselves to $SO(4)$-equivariant solutions of (\ref{e:Yangmillsheatflow}), which means that we assume that the connection $A$ takes the form
$$A(x, t) = Im(\frac{x}{2r^2}\psi(r, t)d\bar{x})$$ $r = |x|$ ( see e.g. \cite{GrotowskiShatah2007} and \cite{SchlatterStruwe1998}).
In this case the equation (\ref{e:Yangmillsheatflow}) reduces to
\begin{equation}\label{bubbletowerIntro123}
\frac{\partial }{\partial t}\psi = \psi_{rr} + \frac{1}{r}\psi_r -\frac{2}{r^2}(\psi-1)(\psi-2)\psi.
\end{equation}
Then we prove the following
\begin{theorem}\label{t:main-bubble-tower}
(1) There exists a solution of (\ref{bubbletowerIntro123}) having the following form $$
\psi(r, t) = \eta_0\left(\frac{r}{\sqrt{\mu_2(t)}}\right)\frac{2r^2}{\mu_2(t)^2+r^2} + \psi_1(r, t)+\varphi_2(r, t).
$$
Here $\psi_1(r, t)$ is the  (one-)bubble solution of (\ref{bubbletowerIntro123}) constructed in Theorem \ref{t:main} with form
$$ \psi_1(r, t) = \frac{2r^2}{1+r^2}+\eta_0\left(\frac{r}{\sqrt{\mu_1(t)}}\right)\frac{2r^2}{\mu_1(t)^2+r^2}+\varphi_1(r, t).$$
Moreover, we have the following estimates as $t\to +\infty$:
$$\mu_1(t)\sim e^{-c_1t}$$
and
$$\mu_2(t) \sim e^{-\frac{c_*e^{2c_1t}}{2c_1}}$$
for some constants $c_1> 0$ and $c_* > 0$.  Furthermore one has  $\varphi_1(r, t)\to 0$ and $\varphi_2(r, t)\to 0$ as $t\to +\infty$, uniformly away from the point $r = 0$.

\vspace{0.2cm}

Equivalently, we have

(2) There exists a solution $A(x, t)$ to (\ref{e:Yangmillsheatflow}) of the form
$$
A(x, t) = Im\left(\eta_0\left(\frac{|x|}{\sqrt{\mu_2(t)}}\right)\frac{x}{\mu_2(t)^2+|x|^2}d\bar{x}\right)+A_1(x, t)++\tilde\varphi_2(x, t)
$$
with $\tilde\varphi_2(x, t) = Im\left(\frac{x}{2r^2}\varphi_2(r, t)d\bar{x}\right)$
and $A_1(x, t)$ is the one bubble solution of (\ref{e:Yangmillsheatflow}) constructed in Theorem \ref{t:main}; $A_1(x, t)$ has the following form
$$A_1(x, t) = Im\left(\frac{x}{1+|x|^2}d\bar{x}\right)+Im\left(\eta_0\left(\frac{|x|}{\sqrt{\mu_1(t)}}\right)\frac{x}{\mu_1(t)^2+|x|^2}d\bar{x}\right)+\tilde\varphi_1(x, t),$$
Moreover, the parameters satisfy $\mu_1(t)\sim e^{-c_1t}$ and $\mu_2(t) \sim e^{-\frac{c_*e^{2c_1t}}{2c_1}}$ as $t\to +\infty$, $c_1> 0$ for some constants $c_1,c_* > 0$.  The 1-forms $\tilde\varphi_1(x, t)\to 0$ and $\tilde\varphi_2(x, t)\to 0$ as $t\to +\infty$, uniformly away from the point $x = 0$.
\end{theorem}

Theorem \ref{t:main-bubble-tower} is new even in the $SO(4)$-equivariant case. If we use the transformation $\bar\psi = r^{-2}\psi$, then (\ref{bubbletowerIntro123}) becomes the following heat equation
\begin{equation}\label{bubbletower111intro}
\frac{\partial }{\partial t}\bar\psi = \bar\psi_{rr} + \frac{5}{r}\bar\psi_r +(6-2r^2\bar\psi)\bar\psi^2
\end{equation}
with steady solution $\bar\psi_0(r) = \frac{2}{r^2+\lambda^2}$. (\ref{bubbletower111intro}) is an evolution ODE, which enjoys very similar properties as the six-dimensional energy critical heat equation.  However, despite this analogy, the constructions in \cite{DelpinoMussoWei} and \cite{SunWeiZhang} for the nonlinear heat equation are designed to handle space dimensions $\geq 7$. The estimates for the outer problem (which is the main difficulty in the construction of bubble-tower solutions) are very hard to apply to the six-dimensional case since the blow-up dynamics are exponential. Our new idea to adjust the bubble with respect to the {\sl exact} solution constructed in Theorem \ref{t:main}. This gives us a uniform scaling parameter for the outer problem. Using this idea, we write the first approximation of $\bar{\psi}(r, t)$ as
$$
\bar U(r, t) = U_*(r, t)+\eta_0\left(\frac{r}{\sqrt{\mu_2(t)}}\right)\frac{1}{\mu_2(t)^2}U\left(\frac{r}{\mu_2(t)}\right)
$$
with
$$
U(r) = \frac{2}{r^2+1}
$$
and $U_*(r, t)$ is the one bubble solution of (\ref{bubbletower111intro}) constructed in Theorem \ref{t:main}. Then we use the inner-outer gluing scheme which gives us a solution of (\ref{bubbletower111intro}) with form
$$
\bar \psi(r, t) = U_*(r, t)+\eta_0\left(\frac{r}{\sqrt{\mu_2(t)}}\right)\frac{1}{\mu_2(t)^2}U\left(\frac{r}{\mu_2(t)}\right)+\varphi_2(r, t).
$$

\vspace{0.5cm}

{\bf The main difficulty and the  Donaldson-De Turck trick for Yang-Mills heat flow}. As mentioned above, the gauge group  consists of all smooth maps from $M$ into $G\subset SO(4)$. The Yang-Mills equations are gauge-invariant, making them non-elliptic. Another feature is that for any connection $A$, there exists a gauge transformation $S$ such that $S(A)$ is a  Coulomb gauge, e.g. $\sum_i \partial_i A_i = 0$, see \cite{Tiangang2002}. A way to fix the gauge is to consider the Coulomb gauge which turns the equations into a strongly elliptic system of the form
$$
\overline{\Delta} A_\alpha+ \text{terms involving only lower-order derivatives of } A_i = 0,
$$
where $\overline{\Delta}_A$ is the Bochner laplacian. In the case of the heat flow, one needs to write the gradient flow of the functional $YM$ in a gauge-invariant fashion.

We describe now the argument in Section 4 of \cite{Struwe1994} (see also \cite{DeTurck1983}, \cite{DonaldsonKronheimer} and \cite{Feehan}) which relies on a version of De Turck's trick for Ricci flow . For $T\in (0, +\infty]$,
let $A_1$ be a smooth connection and suppose that $A = A_1+\varphi$ is a smooth solution of the following Cauchy problem
\begin{equation}\label{e:Yang-Millsmodified}
\left\{
\begin{aligned}
\frac{\partial A}{\partial t} & + D_A^*F_A + D_AD_A^*\varphi = 0\text{ on } M\times (0, T),\\
& A(0) = A_0.
\end{aligned}
\right.
\end{equation}
Through the identification
\begin{equation*}
s = S^{-1}\circ \frac{d}{dt}S = -D_A^*\varphi,
\end{equation*}
the solution $\varphi = \varphi(t)$ generates a family of gauge transformations $S$ that can be readily  recovered by solving the initial value problem
\begin{equation*}
\frac{d}{dt}S = S\circ s, \quad S(0) = id.
\end{equation*}
Define $\tilde A: = (S^{-1})^*A$, then the connection $\tilde A$ is a smooth solution of the Yang-Mills gradient flow
\begin{equation}\label{e:Yang-Millslinearized}
\left\{
\begin{aligned}
\frac{\partial \tilde A}{\partial t} & + D_{\tilde A}^*F_{\tilde A} = 0\text{ on }M\times (0, T),\\
& \tilde A(0) = A_0.
\end{aligned}
\right.
\end{equation}
On the other hand, if $\tilde A$ is a solution of the Yang-Mills gradient flow (\ref{e:Yang-Millslinearized}), then the connection defined by $A:=S^*\tilde A$ is a solution of the Cauchy problem (\ref{e:Yang-Millsmodified}).

Furthermore, the connection $A-A_1$ belongs to the space $C(M\times [0, T), T^*M \otimes \mathfrak g)\cap C^\infty(M\times (0, T), T^*M \otimes \mathfrak g)$ if and only if the same is true for the connection $\tilde A - A_1$. If the initial connection $A_0$ is of class $C^\infty$ , then the connection $A-A_1$ belongs to the space $C^\infty(M\times [0, T), T^*M \otimes \mathfrak g)$ if and only if the same is true for the connection $\tilde A - A_1$. See Lemma 20.3 in \cite{Feehan} for more regularity results.

\section{The Approximation}
\subsection{BPST/ADHM instantons.}
Recall the following notations. We use the quaternions
$$
x = x_1 + x_2i + x_3 j + x_4k\in \mathbb H
$$
for elements of $\mathbb{R}^4$.  One can then construct an instanton (see \cite{AHDM}, \cite{BPST}) by considering
\begin{equation}\label{BPST-instant}
B(x) = Im\left(\frac{x}{1+|x|^2}d\bar{x}\right), \quad F_B = \frac{dx\wedge d\bar{x}}{(1+|x|^2)^2}.
\end{equation}
Let us write
$$
B = \sum_{i = 1}^4 B_i dx_i,\quad F = \sum_{i < j}F_{ij}dx_i\wedge dx_j,
$$
then we have
$$
B_1(x) = Im\left(\frac{x}{1+|x|^2}\right) = \frac{x_2 i + x_3 j + x_4 k}{1+|x|^2},
$$
$$
B_2(x) = Im\left(\frac{-x i}{1+|x|^2}\right) = \frac{-x_1 i + x_3 k - x_4 j}{1+|x|^2},
$$
$$
B_3(x) = Im\left(\frac{-x j}{1+|x|^2}\right) = \frac{-x_1 j - x_2 k + x_4 i}{1+|x|^2},
$$
$$
B_4(x) = Im\left(\frac{-x k}{1+|x|^2}\right) = \frac{-x_1 k + x_2 j - x_3 i}{1+|x|^2}.
$$
See \cite{Atiyah1979} and \cite{FreedUhlenbeck1991} for more details and some background. Based on this solution, 't Hooft constructed the following $5-$parameter family of solutions of the Yang-Mills equation:
$$
B_{\mu, \xi}(x) = Im\left(\frac{x-\xi}{\mu^2+|x-\xi|^2}d\bar{x}\right), \quad\mu\in \mathbb{R},\quad \xi\in \mathbb H,
$$
with curvature
$$
F_{B_{\mu, \xi}} = \frac{\mu^2dx\wedge d\bar{x}}{(\mu^2+|x-\xi|^2)^2}.
$$
It was proved by Atiyah-Hitchin-Singer \cite{AtiyahHitchinSinger1977,AtiyahHitchinSinger1978} that these are all the self-dual solutions of Yang-Mills with instanton number 1. More precisely, the equivalence class of potentials $B$ on $\mathbb{R}^4$ with $YM(B) = 8\pi^2$ is the moduli space $\mathcal M$ of self-dual connection on the Hopf bundle $\mathbb{S}^3\to \mathbb{S}^7\to \mathbb{S}^4$, and the points in the moduli space $\mathcal M$ are in one-to-one correspondence with the set of all such pairs $(\mu, \xi)$, with $\mu >0$ and $\xi\in \mathbb{H}$, see Section 6.5 of \cite{NaberTGG}. 

\subsection{The linearized operator.}
The full form of the stationary Yang-Mills equation is
\begin{equation}\label{e:Yangmillsequationfullform}
\Delta B_j -\sum_{i = 1}^4\partial_i\partial_jB_i + \sum_{i = 1}^4[\partial_i B_i, B_j] + \sum_{i = 1}^4[B_i, \partial_i B_j] + \sum_{i = 1}^4[B_i, \partial_i B_j - \partial_j B_i + [B_i, B_j]]  = 0
\end{equation}
for $j = 1,\cdots, 4$.

We denote the linearized equation at the BPST/ADHM instanton as $L_B$. This is not an elliptic operator because of the term $\sum_{i = 1}^4\partial_i\partial_j\phi_i$. By the Donaldson-De Turck trick explained in the introduction, we consider the following modified linearized operator
\begin{equation}
\mathcal L [\phi]: = L_B \phi + D_BD^*_B \phi = \nabla_B^*\nabla_B\phi -2*[*F_B, \phi].
\end{equation}
The operator $L_B \phi + D_BD^*_B \phi$ is now (strongly) elliptic. If we write $L_B\phi + D_BD^*_B \phi = \sum_{j=1}^4 \mathcal L_j [\phi]dx_j$, then the {\sl  elliptic } linearized operator at the BPST/ADHM instanton is
\begin{equation}\label{e:Yangmillsequationlinearized}
\begin{aligned}
\mathcal L_j [\phi]
& =\Delta \phi_j  + \sum_{i = 1}^4[\partial_i B_i, \phi_j]\\
&\quad + \sum_{i = 1}^4[\phi_i, \partial_i B_j] + \sum_{i = 1}^4[B_i, \partial_i\phi_j]  + \sum_{i = 1}^4[\phi_i, \partial_i B_j - \partial_j B_i + [B_i, B_j]]\\
&\quad + \sum_{i = 1}^4[B_i, \partial_i\phi_j - \partial_j\phi_i + [\phi_i, B_j]+[B_i, \phi_j]]+\sum_{i = 1}^4[B_j, [B_i, \phi_i]]\\
&\quad + \sum_{i = 1}^4 \partial_j[B_i, \phi_i]
\end{aligned}
\end{equation}
for $j = 1,\cdots, 4$.

The elements in the kernel of this operator are (see \cite{Brendle2003})
$$
Z^0_1 = 2\frac{x_2 i + x_3 j + x_4 k}{(1+|x|^2)^2},\quad Z^0_2 = 2\frac{-x_1 i + x_3 k - x_4 j}{(1+|x|^2)^2},
$$
$$
Z^0_3 = 2\frac{-x_1 j - x_2 k + x_4 i}{(1+|x|^2)^2},\quad Z^0_4 = 2\frac{-x_1 k + x_2 j - x_3 i}{(1+|x|^2)^2}
$$
and
$$
Z^1_1 = 0, \quad Z^1_2 = \frac{2i}{(1+|x|^2)^2}, \quad Z^1_3 = \frac{2j}{(1+|x|^2)^2}, \quad Z^1_4 = \frac{2k}{(1+|x|^2)^2},
$$
$$
Z^2_1 = \frac{-2i}{(1+|x|^2)^2}, \quad Z^2_2 = 0, \quad Z^2_3 = \frac{2k}{(1+|x|^2)^2}, \quad Z^2_4 = \frac{-2j}{(1+|x|^2)^2},
$$
$$
Z^3_1 = \frac{-2j}{(1+|x|^2)^2}, \quad Z^3_2 = \frac{-2k}{(1+|x|^2)^2}, \quad Z^3_3 = 0, \quad Z^3_4 = \frac{2i}{(1+|x|^2)^2},
$$
$$
Z^4_1 = \frac{-2k}{(1+|x|^2)^2},\quad Z^4_2 = \frac{2j}{(1+|x|^2)^2},\quad  Z^4_3 = \frac{-2i}{(1+|x|^2)^2},\quad Z^4_4 = 0,
$$
$$
Z^5_1 = x_1F_{12}+x_3F_{14}-x_4F_{13}, \quad Z^5_2 = -x_2F_{21}+x_3F_{24}-x_4F_{23},
$$
$$
Z^5_3 = x_1F_{32}-x_2F_{31}+x_3F_{34}, \quad Z^5_4 = x_1F_{42}-x_2F_{41}-x_4F_{43},
$$
$$
Z^6_1 = x_1F_{13} + x_2F_{14}-x_4F_{12}, \quad Z^6_2 = x_1F_{23}-x_3F_{21}+x_2F_{24},
$$
$$
Z^6_3 = -x_3F_{31} + x_2F_{34}-x_4F_{32}, \quad Z^6_4 =  x_1F_{43}-x_3F_{41}-x_4F_{42},
$$
$$
Z^7_1 = x_1F_{14} + x_2F_{13}-x_3F_{12}, \quad Z^7_2 = x_1F_{24}-x_4F_{21}+x_2F_{23},
$$
$$
Z^7_3 = x_1F_{34}-x_4F_{31}-x_3F_{32}, \quad Z^7_4 = -x_4F_{41} + x_2F_{43}-x_3F_{42}.
$$
Here we have used the notation $Z^i = \sum_{j=1}^4Z^i_jdx_j$.
Note that $|Z^0| \sim \frac{1}{|x|^3}$ and $|Z^i| \sim \frac{1}{|x|^4}$ as $|x|\to +\infty$, $i = 1, 2, 3, 4$. The three kernel $Z^i$, $i = 5, 6, 7$, with decays like $\frac{1}{|x|^3}$ at infinity, can be written as
$$
Z^i=F_B\left(\theta^i_{\rho\sigma}x_\sigma\frac{\partial}{\partial x_\rho},\cdot\right),\quad i = 5,6,7,
$$
with $\theta^i$ being the three basis of $\Lambda^2_+\mathbb{R}^4$, see Proposition 2.7 in \cite{Brendle2003}. $Z^i$ ($i = 0, 1, 2, 3, 4$) are due to scaling and translation invariance of the Yang-Mills connections, that is to say, $Z^0 = \left(\frac{d}{d\mu}B_{\mu, \xi}\right)|_{\mu=1, \xi=0}$, $Z^i = \left(\frac{d}{d\xi_i}B_{\mu, \xi}\right)|_{\mu=1, \xi=0}$, $i = 1, 2, 3, 4$. While $Z^i$, $i = 5,6,7$ are generated from the global $SU(2)=\mathbb{S}^3$ gauge equivalence for the Yang-Mills connections, see page 87 of \cite{Yangyisong}. We summarize the above arguments as follows.
\begin{lemma}\label{nondegeneracy}
For the instanton defined in (\ref{BPST-instant}), a smooth 1-form $\phi$ satisfies $L_B\phi + D_BD^*_B \phi = 0$ if and only if $\phi$ is linear combination of $Z^i$, $i = 0, 1,\cdots, 7$.
\end{lemma}

As we explained in the elliptic case, we consider the following modified  {\sl parabolic} linearized operator at the BPST/ADHM instanton,
\begin{equation}\label{e:Yangmillsflowlinearized}
\begin{aligned}
\partial_t \phi_j & = \Delta \phi_j + \sum_{i = 1}^4[\phi_i, \partial_i B_j] + \sum_{i = 1}^4[B_i, \partial_i\phi_j]  + \sum_{i = 1}^4[\phi_i, \partial_i B_j - \partial_j B_i + [B_i, B_j]]\\
&\quad + \sum_{i = 1}^4[B_i, \partial_i\phi_j - \partial_j\phi_i + [\phi_i, B_j]+[B_i, \phi_j]]\\
&\quad +\partial_j\sum_{i = 1}^4[B_i, \phi_i]  + \sum_{i = 1}^4[B_j, [B_i, \phi_i]]\\
&: = \Delta \phi_j + \tilde{\mathcal L}_j [\phi]
\end{aligned}
\end{equation}
with
\begin{equation*}\label{e:linearizedoperatorhighorder}
\begin{aligned}
\tilde{\mathcal L}_j [\phi]&: = \sum_{i = 1}^4[\phi_i, \partial_i B_j] + \sum_{i = 1}^4[B_i, \partial_i\phi_j]  + \sum_{i = 1}^4[\phi_i, \partial_i B_j - \partial_j B_i + [B_i, B_j]]\\
&\quad + \sum_{i = 1}^4[B_i, \partial_i\phi_j - \partial_j\phi_i + [\phi_i, B_j]+[B_i, \phi_j]]\\
&\quad +\partial_j\sum_{i = 1}^4[B_i, \phi_i]  + \sum_{i = 1}^4[B_j, [B_i, \phi_i]]
\end{aligned}
\end{equation*}
for $j = 1,\cdots, 4$. Also we define the nonlinear term as
\begin{equation*}\label{e:nonlinearizedoperator}
\begin{aligned}
N_j[\phi]:& = \sum_{i = 1}^4\left([\phi_i, \partial_i\phi_j]+[\partial_i\phi_i, \phi_j]\right) + \sum_{i = 1}^4[\phi_i, \partial_i\phi_j - \partial_j\phi_i + [\phi_i, B_j]]\\
&\quad + \sum_{i = 1}^4[\phi_i, \partial_i\phi_j - \partial_j\phi_i + [B_i, \phi_j]] + \sum_{i = 1}^4[B_i, [\phi_i, \phi_j]] + \sum_{i = 1}^4[\phi_i, [\phi_i, \phi_j]]
\end{aligned}
\end{equation*}
and $N[\phi] = \sum_{j=1}^4N_j[\phi]dx_j$.

\subsection{The ansatz.}\label{section2.3}
For $\mu(t)\in \mathbb{R}^+$, $\xi(t)\in \mathbb{R}^4$, we define the approximate solution as follows
\begin{equation}\label{e:ansatz}
\begin{aligned}
A_{\mu, \xi, \theta}(x, t): &= \eta_0\left(\frac{|x-q|}{\sqrt{\mu(t)}}\right) B_{\mu, \xi}(x, t)+B_{1, q}(x) + F_{B_{\mu,\xi}}\left(\theta_{\rho\sigma}(t)(x-\xi(t))_{\sigma}\frac{\partial}{\partial x_\rho},\cdot\right)
\end{aligned}
\end{equation}
with 
$\theta(t)$ being a 2-form depending on $t$.
In the sequel, we compute the contributions of each term involving $B_{\mu,\xi}$ in the error:
\begin{equation*}
\begin{aligned}
-(B_{\mu, \xi})_t&= 2Im\left(\frac{x-\xi(t)}{(\mu(t)^2+|x-\xi(t)|^2)^2}d\bar{x}\right)\mu(t)\dot\mu(t) + \sum_{i=1}^4\frac{\dot\xi_i}{\mu^2}(t)Z^i(y)|_{y=\frac{x-\xi(t)}{\mu(t)}}
\end{aligned}
\end{equation*}
And the linear error of $F_{B_{\mu,\xi}}\left(\theta_{\rho\sigma}(t)(x-\xi(t))_{\sigma}\frac{\partial}{\partial x_\rho},\cdot\right)$ can be computed as follows,
\begin{equation*}
\begin{aligned}
&(-\partial_t+\mathcal L_{B_{\mu, \xi}})\left(F_{B_{\mu,\xi}}\left(\theta_{\rho\sigma}(t)(x-\xi(t))_{\sigma}\frac{\partial}{\partial x_\rho},\cdot\right)\right)\\
&= -\partial_t\left(F_{B_{\mu,\xi}}\left(\theta_{\rho\sigma}(t)(x-\xi(t))_{\sigma}\frac{\partial}{\partial x_\rho},\cdot\right)\right)\\
&= -F_{B_{\mu,\xi}}\left(\dot{\theta}_{\rho\sigma}(t)(x-\xi(t))_{\sigma}\frac{\partial}{\partial y_\rho},\cdot\right)\\
&\quad +F_{B_{\mu,\xi}}\left(\theta_{\rho\sigma}(t)\dot\xi(t)_{\sigma}\frac{\partial}{\partial y_\rho},\cdot\right) -\frac{\partial}{\partial \mu}F_{B_{\mu,\xi}}\dot\mu\left(\theta_{\rho\sigma}(t)(x-\xi(t))_{\sigma}\frac{\partial}{\partial x_\rho},\cdot\right)\\
&\quad -\frac{\partial}{\partial \xi}F_{B_{\mu,\xi}}\dot\xi\left(\theta_{\rho\sigma}(t)(x-\xi(t))_{\sigma}\frac{\partial}{\partial x_\rho},\cdot\right).
\end{aligned}
\end{equation*}

\subsection{Improvement of the error.}
The key point of this paper is solving the linearized problem near the blow-up point with the error of approximation solution as a perturbation term. From the linear theory for the inner problem (see Proposition \ref{proposition5.1}), we need an approximate solution with error decaying faster than $\frac{1}{|x|^3}$ at infinity. 

Observe that the terms $-(B_{\mu, \xi})_t$ and $$(-\partial_t+\mathcal L_{B_{\mu, \xi}})\left(F_{B_{\mu,\xi}}\left(\theta_{\rho\sigma}(t)(x-\xi(t))_{\sigma}\frac{\partial}{\partial x_\rho},\cdot\right)\right)$$ decay like $\frac{1}{|x|^3}$ as $|x|\to +\infty$. Inspired by ideas of \cite{DDW2020}, we improve the approximation by adding nonlocal terms to cancel the main part of the error for $A_{\mu, \xi, \theta}$.

The main terms in $-(B_{\mu, \xi})_t$ and $$(-\partial_t+\mathcal L_{B_{\mu, \xi}})\left(F_{B_{\mu,\xi}}\left(\theta_{\rho\sigma}(t)(x-\xi(t))_{\sigma}\frac{\partial}{\partial x_\rho},\cdot\right)\right)$$ are
\begin{equation*}
\begin{aligned}
&2Im\left(\frac{x-\xi(t)}{(\mu(t)^2+|x-\xi(t)|^2)^2}d\bar{x}\right)\mu(t)\dot\mu(t) +F_{B_{\mu,\xi}}\left(\dot{\theta}_{\rho\sigma}(t)(x-\xi(t))_{\sigma}\frac{\partial}{\partial y_\rho},\cdot\right)\\
&=\frac{\dot\mu(t)}{\mu(t)^2}Z^0(y)|_{y=\frac{x-\xi(t)}{\mu(t)}}+F_{B_{\mu,\xi}}\left(\dot{\theta}_{\rho\sigma}(t)(x-\xi(t))_{\sigma}\frac{\partial}{\partial y_\rho},\cdot\right).
\end{aligned}
\end{equation*}
We look for a differential 1-form $\Phi(x, t)$ that satisfies the following equation
\begin{equation*}
\begin{aligned}
&-\Phi(x, t)_t + (d^*d + dd^*)\Phi(x, t) + 2Im\left(\frac{x-\xi(t)}{(\mu(t)^2+|x-\xi(t)|^2)^2}d\bar{x}\right)\mu(t)\dot\mu(t)\\
&\quad\quad\quad\quad\quad\quad\quad\quad\quad\quad\quad\quad  + F_B\left(\dot \theta_{\rho\sigma}(t)(x-\xi(t))_{\sigma}\frac{\partial}{\partial y_\rho},\cdot\right) = 0 \text{ in }\mathbb{R}^4\times (t_0, +\infty)
\end{aligned}
\end{equation*}
at main order. Set $\Phi(x, t): = \Phi_0(x, t) + \Phi_{1}(x, t)$, $$\Phi_0(x, t): = Im\left((x-\xi(t))\psi^{(0)}(z(\tilde{r}), t)\right)d\bar{x},$$
$$\Phi_{1}(x, t): = dx\wedge d\bar x\left(\psi^{(\rho\sigma)}(z, t)(x-\xi(t))_\sigma\frac{\partial}{\partial x_\rho}, \cdot\right),$$
$z(\tilde{r})=\left(\tilde{r}^2 + \mu^2\right)^{\frac{1}{2}}$, $\tilde{r} = |x-\xi|$,
where $\psi^{(0)}(z, t)$ and $\psi^{(\rho\sigma)}(z, t)$ satisfies
\begin{equation}\label{heatequationdimension4}
\psi^{(0)}_t = \psi^{(0)}_{zz} + \frac{5\psi^{(0)}_z}{z} + \frac{p^{(0)}(t)}{z^4}, \quad p^{(0)}(t) = 2\mu(t)\dot{\mu}(t) ,
\end{equation}
\begin{equation}\label{heatequationdimension40}
\psi^{(\rho\sigma)}_t = \psi^{(\rho\sigma)}_{zz} + \frac{5\psi^{(\rho\sigma)}_z}{z} + \frac{p^{(\rho\sigma)}(t)}{z^4}, \quad p^{(\rho\sigma)}(t) = \dot \theta_{\rho\sigma}(t),
\end{equation}
which are the radially symmetric forms of an inhomogeneous linear heat equation in $\mathbb{R}^6$.
By the Duhamel's principle, we know that
\begin{equation*}
\begin{aligned}
&\psi^{(0)}(z, t) = \int_{t_0}^{t}\left(2\mu(\tilde{s})\dot{\mu}(\tilde{s})\right) k_1(t-\tilde{s}, z)d\tilde{s},
\end{aligned}
\end{equation*}
\begin{equation*}
\begin{aligned}
&\psi^{(\rho\sigma)}(z, t) = \int_{t_0}^{t}\left(\dot \theta_{\rho\sigma}(\tilde{s})\right) k_1(t-\tilde{s}, z)d\tilde{s}
\end{aligned}
\end{equation*}
provide bounded solutions for (\ref{heatequationdimension4}) and (\ref{heatequationdimension40}) respectively; here $k_1(t, z) = \frac{1-e^{-\frac{z^2}{4t}}(1+\frac{z^2}{4t})}{z^4}$. Then we define an improved approximation as
$$
A_{\mu, \xi, \theta}^* = A_{\mu, \xi, \theta} + \Phi_0 + \Phi_1.
$$
Now the linear error $\mathcal{E}^*$ of $A_{\mu, \xi, \theta}^*$ becomes 
\begin{equation*}
\begin{aligned}
\mathcal{E}^* & = -\left(B_{\mu, \xi}\right)_t\\
&\quad + (-\partial_t+\mathcal L_{B_{\mu, \xi}})\left(B_{1, q}+F_{B_{\mu,\xi}}\left(\theta_{\rho\sigma}(t)(x-\xi(t))_{\sigma}\frac{\partial}{\partial x_\rho},\cdot\right)+\Phi_0+\Phi_1\right)\\
&\quad +  (-\partial_t+\mathcal L_{B_{\mu, \xi}})\left(\left(\eta_0\left(\frac{|x-q|}{\sqrt{\mu(t)}}-1\right)\right)B_{\mu, \xi}\right)\\
&= \tilde{\mathcal L}_i[B_{1, q}]+\tilde{\mathcal L}_i[\Phi_0]+\tilde{\mathcal L}_i[\Phi_1^{(1)}]+\tilde{\mathcal L}_i[\Phi_1^{(2)}]+\tilde{\mathcal L}_i[\Phi_1^{(3)}] + \sum_{j=1}^4\dot\xi_j(t)Z^j(y)|_{y=\frac{x-\xi(t)}{\mu(t)}}\\
&\quad + dx\wedge d\bar x\left(\psi^{(\rho\sigma)}(z, t)\dot\xi(t)_\sigma\frac{\partial}{\partial x_\rho}, \cdot\right)\\
&\quad + dx\wedge d\bar x\left(\partial_z\psi^{(\rho\sigma)}(z, t)(x-\xi(t))_\sigma\frac{2(x-\xi)\cdot\dot\xi-2\mu\dot\mu}{z}\frac{\partial}{\partial x_\rho}, \cdot\right)\\
&\quad + F_{B_{\mu,\xi}}\left(\theta_{\rho\sigma}(t)\dot\xi(t)_{\sigma}\frac{\partial}{\partial y_\rho},\cdot\right) -\frac{\partial}{\partial \mu}F_{B_{\mu,\xi}}\dot\mu\left(\theta_{\rho\sigma}(t)(x-\xi(t))_{\sigma}\frac{\partial}{\partial x_\rho},\cdot\right)\\
&\quad -\frac{\partial}{\partial \xi}F_{B_{\mu,\xi}}\dot\xi\left(\theta_{\rho\sigma}(t)(x-\xi(t))_{\sigma}\frac{\partial}{\partial x_\rho},\cdot\right)-\left(\eta_0\left(\frac{|x-q|}{\sqrt{\mu(t)}}\right)\right)_tB_{\mu, \xi}\\
&\quad  +  \mathcal L_{B_{\mu, \xi}}\left(\left(1-\eta_0\left(\frac{|x-q|}{\sqrt{\mu(t)}}\right)\right)B_{\mu, \xi}\right)\\
&:= \sum_{i=1}^4\mathcal E^*_idx_i.
\end{aligned}
\end{equation*}
Here
$$
\Phi_1^{(1)} = dx\wedge d\bar x\left(\psi^{(12)}(z, t)(x-\xi(t))_2\frac{\partial}{\partial x_2}+\psi^{(34)}(z, t)(x-\xi(t))_4\frac{\partial}{\partial x_3}, \cdot\right),
$$
$$\Phi_1^{(2)} = dx\wedge d\bar x\left(\psi^{(13)}(z, t)(x-\xi(t))_2\frac{\partial}{\partial x_2}+\psi^{(24)}(z, t)(x-\xi(t))_4\frac{\partial}{\partial x_3}, \cdot\right),$$
$$\Phi_1^{(3)} = dx\wedge d\bar x\left(\psi^{(14)}(z, t)(x-\xi(t))_2\frac{\partial}{\partial x_2}+\psi^{(23)}(z, t)(x-\xi(t))_4\frac{\partial}{\partial x_3}, \cdot\right).$$
We refer the readers to the Appendix for  the computations of the terms $\tilde{\mathcal L}_i[B_{1, q}]$, $\tilde{\mathcal L}_i[\Phi_0]$, $\tilde{\mathcal L}_i[\Phi_1^{(1)}]$, $\tilde{\mathcal L}_i[\Phi_1^{(2)}]$, $\tilde{\mathcal L}_i[\Phi_1^{(3)}]$.

\subsection{The blow-up rate.}\label{blow-up-rate}
We compute
\begin{equation*}
\int_{\mathbb{R}^4}\sum_{i=1}^4\mathcal{E}^*_i\cdot Z^0_idx = 144\frac{1}{\mu(t)}\int_{t_0}^{t}\frac{\mu(\tilde{s})\dot{\mu}(\tilde{s})}{(t-\tilde{s})^2}\Omega\left(\frac{\mu(\tilde{s})^2}{t-\tilde{s}}\right) d\tilde{s} + \frac{24\pi^2}{\mu},
\end{equation*}
where
\begin{equation*}
\Omega(\tau)=\int_{0}^{+\infty}\Gamma\left(\tau\right)\frac{\rho^2}{(1+\rho^2)^4}\rho^3d\rho= \int_{0}^{+\infty}\frac{1-e^{-\tau\frac{\rho^2+1}{4}}(1+\tau\frac{\rho^2+1}{4})}{\tau^2(\rho^2+1)^2}\frac{\rho^2}{(1+\rho^2)^4}\rho^3d\rho.
\end{equation*}
The blow-up rate is determined by the equation,
\begin{equation*}
\int_{\mathbb{R}^4}\sum_{i=1}^4\mathcal{E}^*_i\cdot Z^0_idx\approx 0,
\end{equation*}
which reduces to
\begin{equation*}
144\frac{1}{\mu(t)}\int_{t_0}^{t}\frac{\mu(\tilde{s})\dot{\mu}(\tilde{s})}{(t-\tilde{s})^2}\Omega\left(\frac{\mu(\tilde{s})^2}{t-\tilde{s}}\right) d\tilde{s} + \frac{24\pi^2}{\mu}\approx 0.
\end{equation*}
We claim that by choosing $\mu = e^{-\kappa_0t}$ for suitable $\kappa_0 > 0$, we have
\begin{equation}\label{e:201803312}
144\int_{t_0}^{t}\frac{\mu(\tilde{s})\dot{\mu}(\tilde{s})}{(t-\tilde{s})^2}\Omega\left(\frac{\mu(\tilde{s})^2}{t-\tilde{s}}\right) d\tilde{s} = -144\Xi \kappa_0(1+o(1))
\end{equation}
for a constant $\Xi > 0$. Indeed, for a small constant $\delta > 0$, we decompose the integral $$\int_{t_0}^{t}\frac{\mu(\tilde{s})\dot{\mu}(\tilde{s})}{\left(t-\tilde{s}\right)^2}
\Omega\left(\frac{\mu(\tilde{s})^2}{t-\tilde{s}}\right)d\tilde{s}$$ into
\begin{equation*}
\begin{aligned}
&\int_{t_0}^{t}\frac{\mu(\tilde{s})\dot{\mu}(\tilde{s})}{\left(t-\tilde{s}\right)^2}
\Omega\left(\frac{\mu(\tilde{s})^2}{t-\tilde{s}}\right)d\tilde{s}= \int_{t_0}^{t-\delta}\frac{\mu(\tilde{s})\dot{\mu}(\tilde{s})}{\left(t-\tilde{s}\right)^2}
\Omega\left(\frac{\mu(\tilde{s})^2}{t-\tilde{s}}\right)d\tilde{s}\\
&\quad\quad\quad\quad\quad\quad\quad\quad\quad\quad\quad\quad\quad + \int_{t-\delta}^{t}\frac{\mu(\tilde{s})\dot{\mu}(\tilde{s})}{\left(t-\tilde{s}\right)^2}
\Omega\left(\frac{\mu(\tilde{s})^2}{t-\tilde{s}}\right)d\tilde{s} := I_1 + I_2.
\end{aligned}
\end{equation*}
For the term $I_1$, $t-\tilde{s} >\delta$, we have the following estimate
\begin{equation*}
\begin{aligned}
0\leq -I_1 &\leq \kappa_0\int_{t_0}^{t-\delta}\frac{\mu(\tilde{s})^2}{\left(t-\tilde{s}\right)^2}
\Omega\left(\frac{\mu(\tilde{s})^2}{t-\tilde{s}}\right)d\tilde{s}\leq C\frac{\kappa_0}{\delta}\int_{t_0}^{t-\delta}\left|\frac{(t-\tilde{s})^{\frac{1}{2}}}{\mu(\tilde{s})}\right|^{-2}d\tilde{s}\\
&= \frac{C}{\delta^2}\left(e^{-2\kappa_0 t_0 -e^{-2\kappa_0(t-\delta)}}\right)\leq \frac{C}{\delta^2}e^{-2\kappa_0 t_0}.
\end{aligned}
\end{equation*}
For the term $I_2 =\int_{t-\delta}^{t}\frac{\mu(\tilde{s})\dot{\mu}(\tilde{s})}{\left(t-\tilde{s}\right)^2}\Omega\left(\frac{\mu(\tilde{s})^2}{t-\tilde{s}}\right)d\tilde{s}$,
we use change of variables $\frac{(t-\tilde{s})^{\frac{1}{2}}}{\mu(\tilde{s})} = \hat{s}$, then it holds that
$$
d\tilde{s} = -\frac{\mu(\tilde{s})}{\frac{1}{2}(t-\tilde{s})^{-\frac{1}{2}}+\dot{\mu}(\tilde{s})\hat{s}}d\hat{s}
$$
and
\begin{equation*}
\begin{aligned}
I_2&=\int_{t-\delta}^{t}\frac{\mu(\tilde{s})\dot{\mu}(\tilde{s})}{\left(t-\tilde{s}\right)^2}
\Omega\left(\frac{\mu(\tilde{s})^2}{t-\tilde{s}}\right)d\tilde{s}\\ &=\int^{\frac{\delta^{\frac{1}{2}}}{\mu(t-\delta)}}_{0}\frac{\mu(\tilde{s})\dot{\mu}(\tilde{s})}{\left(t-\tilde{s}\right)^2}
\Omega\left(\frac{1}{\hat{s}^2}\right)\frac{\mu(\tilde{s})}{\frac{1}{2}(t-\tilde{s})^{-\frac{1}{2}}+\dot{\mu}(\tilde{s})\hat{s}}d\hat{s}.
\end{aligned}
\end{equation*}
Note that for $\delta > 0$ small enough, $\frac{1}{2}(t-\tilde{s})^{-\frac{1}{2}}+\dot{\mu}(\tilde{s})\hat{s} =\frac{1}{2}(t-\tilde{s})^{-\frac{1}{2}}(1-2\kappa_0)
> \frac{1}{2}(t-\tilde{s})^{-\frac{1}{2}}(1-2\kappa\delta)$, $d\tilde{s} = \frac{\mu(\tilde{s})}{\frac{1}{2}(t-\tilde{s})^{-\frac{1}{2}}}(1+O(\delta))d\hat{s}$, therefore it holds that
\begin{equation*}
I_2= -2\kappa_0\left(\int^{\frac{\delta^{\frac{1}{2}}}{\mu(t-\delta)}}_{0}\frac{1}{\hat{s}^3}\Omega\left(\frac{1}{\hat{s}^2}\right)
d\hat{s} + o(1)\right)
= -\Xi\kappa_0 + o(1)
\end{equation*}
if $\frac{\delta^{\frac{1}{2}}}{\mu(t-\delta)}$ is sufficiently large. Here $\Xi = \int_{0}^\infty \Omega(s)ds \in (0, \infty)$. Therefore we have
\begin{equation*}
\begin{aligned}
144\int_{t_0}^{t}\frac{\mu(\tilde{s})\dot{\mu}(\tilde{s})}{\left(t-\tilde{s}\right)^2}
\Omega\left(\frac{\mu(\tilde{s})^2}{t-\tilde{s}}\right)d\tilde{s}  &= -144\Xi\kappa_0(1+o(1))
\end{aligned}
\end{equation*}
if $t_0$ is large enough. This proves (\ref{e:201803312}).
From (\ref{e:201803312}), we know that one can choose the main order term of $\mu(t)$ as
$$
\mu_0 = e^{-\kappa_0 t} \text{ with }\kappa_0=\frac{\pi^2}{6\Xi}.
$$
Similarly, we have
\begin{equation*}
\dot \xi \approx 0, \quad \dot \theta_{\rho\sigma}(t)\approx 0.
\end{equation*}
Therefore, we choose
we choose $\xi_0 = 0$ and $\theta^0_{\rho\sigma}(t) = 0$. Therefore, we have
\begin{lemma}
The main order terms of $\mu(t)$, $\xi$ and $\theta_{\rho\sigma}$ are
$$
\mu_0 = e^{-\kappa_0 t} \text{ with }\kappa_0=\frac{\pi^2}{6\Xi}
$$
and 
$$\xi_0 = 0,\quad \theta^0_{\rho\sigma}(t) = 0.$$
\end{lemma}

Let us fix the parameter functions $\mu_0(t)$, $\xi_0(t)$ defined above. Then we write
$$
\mu(t) = \mu_0(t) + \lambda(t).
$$
We will find a small solution $\varphi$ of
\begin{equation}\label{e:modifiedequation}
\mathcal{E}^* -\partial_t\varphi + \mathcal L_{B_{\mu, \xi}}(\varphi)+N[A_{\mu, \xi, \theta}^*-B_{\mu, \xi}+\varphi] = 0
\end{equation}
with $A_{\mu, \xi, \theta}$ defined in (\ref{e:ansatz}). In other words, let $t_0 > 0$, the connection
$$
A(x, t) = A_{\mu, \xi, \theta}^*(x, t) + \varphi(x, t)
$$
will solve the problem
\begin{equation*}\label{gaugedYM}
\left\{\begin{array}{ll}
        \frac{\partial A}{\partial t} = -D_A^*F_A + D_{A}D_{A}^*\left(A_{\mu, \xi, \theta}^*-B_{\mu, \xi}+\varphi\right)\quad\text{ in }\mathbb{R}^4\times [t_0, \infty),\\
        A(\cdot, t_0) = A_0\quad\text{ in }\mathbb{R}^4
       \end{array}
\right.
\end{equation*}
when $t_0$ is sufficiently large. Then we use the Donaldson-De Turck trick described at the end  of the introduction to obtain a solution of (\ref{e:Yangmillsheatflow}).

\subsection{The inner-outer gluing system.}
Let $\eta_0(s)$ be a smooth cut-off function satisfying $\eta_0(s) = 1$ for $s < 1$ and $\eta_0(s) = 0$ for $s > 2$. We define a sufficiently large constant of form
\begin{equation}\label{definitionofR}
R = e^{\rho t_0}
\end{equation}
for a sufficiently small positive real number $\rho$. Set
$$
\eta_R(x, t):=\eta_0\left(\frac{\left|x -\xi(t)\right|}{R\mu_{0}(t)}\right).
$$
We consider $\varphi(x, t)$ with following form
\begin{equation}\label{e:errorterm}
\varphi(x, t) = \eta_R \tilde{\phi}(x, t) + \psi(x, t)
\end{equation}
for a  1-form $\tilde{\phi}(x, t) = \phi\left(\frac{x-\xi(t)}{\mu_0(t)}, t\right)$ and $\phi(\cdot, t_0) = 0$.
Let us recall that (\ref{e:modifiedequation}) can be expressed explicitly by
\begin{equation*}
\begin{aligned}
\partial_t\varphi_j & =\Delta\varphi_j  + \sum_{i = 1}^4[\partial_i B_i, \varphi_j]\\
&\quad + \sum_{i = 1}^4[\varphi_i, \partial_i B_j] + \sum_{i = 1}^4[B_i, \partial_i\varphi_j]  + \sum_{i = 1}^4[\varphi_i, \partial_i B_j - \partial_j B_i + [B_i, B_j]]\\
&\quad + \sum_{i = 1}^4[B_i, \partial_i\varphi_j - \partial_j\varphi_i + [\varphi_i, B_j]+[B_i, \varphi_j]]+\sum_{i = 1}^4[B_j, [B_i, \varphi_i]]\\
&\quad + \sum_{i = 1}^4 \partial_j[B_i, \varphi_i] + N_j[A_{\mu, \xi, \theta}^*-B_{\mu, \xi}+\varphi] + \mathcal{E}^*_j
\end{aligned}
\end{equation*}
with
\begin{equation*}
\begin{aligned}
\mathcal{E}^* & = -\left(B_{\mu, \xi}\right)_t\\
&\quad + (-\partial_t+\mathcal L_{B_{\mu, \xi}})\left(B_{1, q}+F_{B_{\mu,\xi}}\left(\theta_{\rho\sigma}(t)(x-\xi(t))_{\sigma}\frac{\partial}{\partial x_\rho},\cdot\right)+\Phi_0+\Phi_1\right)\\
&\quad +  (-\partial_t+\mathcal L_{B_{\mu, \xi}})\left(\left(\eta_0\left(\frac{|x-q|}{\sqrt{\mu(t)}}-1\right)\right)B_{\mu, \xi}\right)\\
&: = \sum_{i=1}^4\mathcal{E}^*_i dx_i.
\end{aligned}
\end{equation*}
and
\begin{align}
&N_j[A_{\mu, \xi, \theta}^*-B_{\mu, \xi}+\varphi]\notag\\
& : = \sum_{i = 1}^4\left([\phi_i, \partial_i\phi_j]+[\partial_i\phi_i, \phi_j]\right)\notag\\
&\quad + \sum_{i = 1}^4[\varphi_i, \partial_i\varphi_j - \partial_j\varphi_i + [\varphi_i, A_{\mu, \xi, \theta, j}-B_{\mu, \xi, j}+\Phi_{0, j}+\Phi_{1, j}]]\notag\\
&\quad + \sum_{i = 1}^4[\varphi_i, \partial_i\varphi_j - \partial_j\varphi_i + [A_{\mu, \xi, \theta, i}-B_{\mu, \xi, i}+\Phi_{0, i}+\Phi_{1, i}, \varphi_j]]\label{nonlineatterm}\\
\end{align}
\begin{align*}
&\quad + \sum_{i = 1}^4[A_{\mu, \xi, \theta, i}-B_{\mu, \xi, i}+\Phi_{0, i}+\Phi_{1, i}, [\varphi_i, \varphi_j]] + \sum_{i = 1}^4[\varphi_i, [\varphi_i, \varphi_j]].\notag
\end{align*}

Then $\varphi$ defined in (\ref{e:errorterm}) solves (\ref{gaugedYM}) if the pair $(\phi, \psi)$ satisfies the following system of parabolic equations
\begin{equation}\label{e:innerproblembeforechangeofvariables2}
\begin{aligned}
&\mu_{0}^2(t)\partial_t \phi_j\\ 
& = \Delta \phi_j + \sum_{i = 1}^4[\phi_i, \partial_i B_j] + \sum_{i = 1}^4[B_i, \partial_i\phi_j]  + \sum_{i = 1}^4[\phi_i, \partial_i B_j - \partial_j B_i + [B_i, B_j]]\\
&\quad + \sum_{i = 1}^4[B_i, \partial_i\phi_j - \partial_j\phi_i + [\phi_i, B_j]+[B_i, \phi_j]] + \partial_j\sum_{i = 1}^4[B_i, \phi_i] + \sum_{i = 1}^4[B_j, [B_i, \phi_i]]\\
&\quad + \mu_{0}^3\sum_{i = 1}^4[\psi_i, \partial_i B_j] + \mu_{0}^3\sum_{i = 1}^4[B_i, \partial_i\psi_j] + \mu_{0}^3\sum_{i = 1}^4[\psi_i, \partial_i B_j - \partial_j B_i + [B_i, B_j]]\\
&\quad + \mu_{0}^3\sum_{i = 1}^4[B_i, \partial_i\psi_j - \partial_j\psi_i + [\psi_i, B_j]+[B_i, \psi_j]] + \mu_{0}^3\partial_j\sum_{i = 1}^4[B_i, \psi_i]\\
&\quad + \mu_{0}^3\sum_{i = 1}^4[B_j, [B_i, \psi_i]] + \mu_{0}^3\mathcal E^*_j(\xi+\mu_{0}y, t)\text{ in }B_{2R}\times [t_0, +\infty),\\
\end{aligned}
\end{equation}
and
\begin{equation}\label{e:outerproblem}
\begin{aligned}
&\partial_t \psi_j =  \Delta \psi_j + (1-\eta_R)\sum_{i = 1}^4[\psi_i, \partial_i B_{\mu, \xi, j}] + (1-\eta_R)\sum_{i = 1}^4[B_{\mu, \xi, i}, \partial_i\psi_j]\\
&\quad\quad\quad + (1-\eta_R)\sum_{i = 1}^4[\psi_i, \partial_i B_{\mu, \xi, j} - \partial_j B_{\mu, \xi, i} + [B_{\mu, \xi, i}, B_{\mu, \xi, j}]]\\
&\quad\quad\quad + (1-\eta_R)\sum_{i = 1}^4[B_{\mu, \xi, i}, \partial_i\psi_j - \partial_j\psi_i + [\psi_i, B_{\mu, \xi, j}]+[B_{\mu, \xi, i}, \psi_j]]\\
&\quad\quad\quad + (1-\eta_R)\partial_j\sum_{i = 1}^4[B_{\mu, \xi, i}, \psi_i] + (1-\eta_R)\sum_{i = 1}^4[B_{\mu, \xi, j}, [B_{\mu, \xi, i}, \psi_i]]\\
&\quad\quad\quad + N_j[A_{\mu, \xi, \theta}^*-B_{\mu, \xi}+\varphi] + (1-\eta_R)\mathcal E^*_j(x, t)\\
&\quad\quad\quad + \nabla\eta_{R}\nabla\tilde{\phi}_j+ \tilde{\phi}_j\big(\Delta -\partial_t\big)\eta_R\text{ in }\mathbb{R}^4\times [t_0, +\infty).
\end{aligned}
\end{equation}
for $j = 1,\cdots, 4$.

\section{Proof of the main theorem}
\subsection{The inner problem.}
To find a pair of solutions $(\phi, \psi)$ satisfying the inner problem (\ref{e:innerproblembeforechangeofvariables2}) and the outer problem (\ref{e:outerproblem}), we rewrite the inner problem (\ref{e:innerproblembeforechangeofvariables2}) as
\begin{eqnarray}\label{e5:1}
&& \mu_0^2(t)\partial_t \phi_j = \mathcal L_j[\phi] + H_j[\lambda,\xi, \theta, \dot{\lambda},\dot{\xi}, \dot{\theta}, \phi, \psi](y,t),\,\, y\in B_{2R}(0)
\end{eqnarray}
for $j = 1, 2, 3, 4$ and $ t\geq t_0$, where $H_j[\lambda,\xi, \theta, \dot{\lambda},\dot{\xi}, \dot{\theta}, \phi, \psi](y,t)$ is defined by
\begin{equation}\label{e5:2}
\begin{aligned}
H_j[\lambda,\xi, \theta, \dot{\lambda},\dot{\xi}, \dot{\theta}, \phi, \psi]&:=\mu_{0}^3\sum_{i = 1}^4[\psi_i, \partial_i B_j] + \mu_{0}^3\sum_{i = 1}^4[B_i, \partial_i\psi_j]\\
&\quad + \mu_{0}^3\sum_{i = 1}^4[\psi_i, \partial_i B_j - \partial_j B_i + [B_i, B_j]]\\
&\quad + \mu_{0}^3\sum_{i = 1}^4[B_i, \partial_i\psi_j - \partial_j\psi_i + [\psi_i, B_j]+[B_i, \psi_j]]\\
&\quad + \mu_{0}^3\partial_j\sum_{i = 1}^4[B_i, \psi_i] + \mu_{0}^3\sum_{i = 1}^4[B_j, [B_i, \psi_i]]\\
&\quad + \mu_{0}^3\mathcal E^*_j(\xi+\mu_{0}y, t).
\end{aligned}
\end{equation}
We use change of variables
\begin{equation*}\label{e5:3}
t = t(\tau),\quad \frac{dt}{d\tau} = \mu_{0}^{2}(t);
\end{equation*}
it is easy to see that the inner problem (\ref{e5:1}) becomes
\begin{eqnarray}\label{e5:4}
\partial_\tau \phi_j = \mathcal L_j[\phi] + H_j[\lambda,\xi, \theta, \dot{\lambda},\dot{\xi}, \dot{\theta}, \phi, \psi](y,t(\tau))
\end{eqnarray}
for $y\in B_{2R}(0)$, $\tau\geq \tau_0$. Here $\tau_0$ the unique positive number satisfying $t(\tau_0) = t_0$.
We will find a solution $\phi$ to the following problem
\begin{equation}\label{e5:600}
\left\{
\begin{aligned}
&\partial_\tau \phi_j = \mathcal L_j[\phi] + H_j[\lambda,\xi, \theta, \dot{\lambda},\dot{\xi}, \dot{\theta}, \phi, \psi](y,t(\tau)),\,\, y\in B_{2R}(0),\,\,\tau\geq\tau_0,\\
&\phi_j(y,\tau_0) = 0,\,\, y\in B_{2R}(0),\,\,j = 1, 2, 3, 4.
\end{aligned}
\right.
\end{equation}
We will prove that problem (\ref{e5:600}) is solvable for $\phi$ when $\psi$ is in some weighted spaces and the parameters $\lambda$, $\xi$, $\theta$ are chosen so that the right hand side $$H_j[\lambda,\xi, \theta, \dot{\lambda},\dot{\xi}, \dot{\theta}, \phi, \psi](y,t(\tau))$$  of (\ref{e5:600}) satisfies the following $L^2$-orthogonality conditions
\begin{equation}\label{e5:7}
\int_{B_{2R}}\sum_{i=1}^4H_i[\lambda,\xi, \theta, \dot{\lambda},\dot{\xi}, \dot{\theta}, \phi, \psi](y,t(\tau))Z^l_i(y)dy = 0,
\end{equation}
for all $\tau\geq \tau_0$, $l = 0,1,2,\cdots,7$. To obtain a solution $\phi$, we  apply the Schauder fixed-point theorem. First, we need a linear theory for problem (\ref{e5:600}).

For $R > 0$, let us consider the following initial value problem
\begin{equation}\label{e5:31}
\left\{
\begin{aligned}
&\partial_\tau \phi_j = \mathcal L_j[\phi] + h_j(y,\tau),\,\,y\in B_{2R}(0),\,\,\tau\geq \tau_0,\\
&\phi(y,\tau_0) = 0.
\end{aligned}
\right.
\end{equation}
We define the weighted norm for a differential form $h = \sum_{j=1}^4h_jdx_j$ as follows,
\begin{equation*}
\|h\|_{\alpha, \nu} = \sum_{j=1}^4\|h_j\|_{\alpha, \nu}\text{ with }\|h_j\|_{\alpha, \nu}: = \sup_{\tau > \tau_0}\sup_{y\in B_{2R}}\tau^\nu(1 + |y|^\alpha)|h_j(y,\tau)|.
\end{equation*}
Then we have the following estimates for problem (\ref{e5:31}).
\begin{prop}\label{proposition5.1}
Suppose $\alpha>0$, $\nu > 0$, $\|h\|_{3+\alpha,\nu} < +\infty$ and
\begin{equation*}\label{orthogalcondition}
\int_{B_{2R}}\sum_{i=1}^4h_i(y)Z^l_i(y)dy = 0~\text{ for all }~\tau\in (\tau_0,\infty), ~l = 0, 1, \cdots, 7.
\end{equation*}
Then there exist a differential 1-form $\phi = \phi[h](y, \tau)$ satisfying problem (\ref{e5:31}). For $\tau\in (\tau_0,+\infty)$, $y\in B_{2R}(0)$, it holds that
\begin{equation}\label{e5:100}
\begin{aligned}
&(1+|y|)|\nabla_y \phi_j(y, \tau)|+ |\phi_j(y,\tau)|\lesssim \tau^{-\nu}(1+|y|)^{-1-\alpha}\|h\|_{3+\alpha,\nu},\,\,j=1, 2, 3, 4.
\end{aligned}
\end{equation}
\end{prop}
The proof of Proposition \ref{proposition5.1} will be given in Section 4. Assuming that
\begin{equation*}
\|\psi\|_{**,1+\sigma,1+\alpha}\leq ce^{-\epsilon t_0}
\end{equation*}
for some small $\epsilon > 0$. Here $\|\psi\|_{**,1+\sigma,1+\alpha}$ is the least $M > 0$ such that
\begin{equation}\label{e4:400}
|\psi(x, t)|\leq M
\left\{
\begin{aligned}
&\frac{\mu_{0}^{1+\sigma}}{1+|y|^{1+\alpha}}, \quad |y|=\left|\frac{x-\xi}{\mu_{0}}\right|\leq \mu^{-1}_{0},\\
&\mu_{0}^{1+\sigma+1+\alpha}, \quad |y|=\left|\frac{x-\xi}{\mu_{0}}\right| > \mu^{-1}_{0}\\
\end{aligned}
\right.
\end{equation}
holds. Then we have the following estimates for $H_j[\lambda,\xi, \theta, \dot{\lambda},\dot{\xi}, \dot{\theta}, \phi, \psi](y,t)$ in the inner problem (\ref{e5:600}).
\begin{enumerate}
\item[(1)]
\begin{equation}
\begin{aligned}
\left|\mu_{0}^3\sum_{i = 1}^4[\psi_i, \partial_i B_j - \partial_j B_i + [B_i, B_j]]\right|\lesssim \|\psi\|_{**,1+\sigma,1+\alpha}\frac{\mu_{0}^{2+\sigma}}{1+|y|^{3+\alpha}}.
\end{aligned}
\end{equation}
\item[(2)]
\begin{equation}
\begin{aligned}
\left|\mu_{0}^3\mathcal E^*_j(\xi+\mu_{0}y, t)\right|\lesssim \frac{\mu_{0}^{3}}{1+|y|^{4}}\lesssim e^{-\epsilon t_0}\frac{\mu_{0}^{2+\sigma}}{1+|y|^{3+\alpha}}.
\end{aligned}
\end{equation}
\item[(3)]
\begin{equation}
\begin{aligned}
\left|\mu_{0}^3\sum_{i = 1}^4[\psi_i, \partial_i B_j]\right|&\lesssim \|\psi\|_{**,1+\sigma,1+\alpha}\frac{1}{1+|y|^2}\frac{\mu_{0}^{2+\sigma}}{1+|y|^{1+\alpha}}\\
&\lesssim \|\psi\|_{**,1+\sigma,1+\alpha}\frac{1}{1+|y|^{3+\alpha}}\mu_{0}^{2+\sigma}.
\end{aligned}
\end{equation}
\item[(4)]
\begin{equation}
\begin{aligned}
\left|\mu_{0}^3\sum_{i = 1}^4[B_i, \partial_i\psi_j]\right|&\lesssim \|\psi\|_{**,1+\sigma,1+\alpha}\frac{1}{1+|y|}\frac{\mu_{0}^{2+\sigma}}{1+|y|^{2+\alpha}}\\
&\lesssim \|\psi\|_{**,1+\sigma,1+\alpha}\frac{1}{1+|y|^{3+\alpha}}\mu_{0}^{2+\sigma}.
\end{aligned}
\end{equation}
\item[(5)]
\begin{equation}
\begin{aligned}
&\left|\mu_{0}^3\sum_{i = 1}^4[B_i, \partial_i\psi_j - \partial_j\psi_i + [\psi_i, B_j]+[B_i, \psi_j]]\right|\\
&\lesssim \|\psi\|_{**,1+\sigma,1+\alpha}\frac{1}{1+|y|^2}\frac{\mu_{0}^{2+\sigma}}{1+|y|^{1+\alpha}} + \|\psi\|_{**,1+\sigma,1+\alpha}\frac{1}{1+|y|}\frac{\mu_{0}^{2+\sigma}}{1+|y|^{2+\alpha}}\\
&\lesssim \|\psi\|_{**,1+\sigma,1+\alpha}\frac{1}{1+|y|^{3+\alpha}}\mu_{0}^{2+\sigma}.
\end{aligned}
\end{equation}
\item[(6)]
\begin{equation}
\begin{aligned}
&\left|\mu_{0}^3\partial_j\sum_{i = 1}^4[B_i, \psi_i]\right|\\
&\lesssim \|\psi\|_{**,1+\sigma,1+\alpha}\frac{1}{1+|y|^2}\frac{\mu_{0}^{2+\sigma}}{1+|y|^{1+\alpha}} + \|\psi\|_{**,1+\sigma,1+\alpha}\frac{1}{1+|y|}\frac{\mu_{0}^{2+\sigma}}{1+|y|^{2+\alpha}}\\
&\lesssim \|\psi\|_{**,1+\sigma,1+\alpha}\frac{1}{1+|y|^{3+\alpha}}\mu_{0}^{2+\sigma}.
\end{aligned}
\end{equation}
\item[(7)]
\begin{equation}
\begin{aligned}
\left|\mu_{0}^3\sum_{i = 1}^4[B_j, [B_i, \psi_i]]\right|&\lesssim \|\psi\|_{**,1+\sigma,1+\alpha}\frac{1}{1+|y|^2}\frac{\mu_{0}^{2+\sigma}}{1+|y|^{1+\alpha}}\\
&\lesssim \|\psi\|_{**,1+\sigma,1+\alpha}\frac{1}{1+|y|^{3+\alpha}}\mu_{0}^{2+\sigma}.
\end{aligned}
\end{equation}
\end{enumerate}

\subsection{The orthogonality conditions.}
To apply Proposition \ref{proposition5.1}, we choose the parameters $\lambda$, $\xi$ and $r$ satisfying the orthogonality conditions (\ref{e5:7}).
Fix a $\sigma\in (0, 1)$, for a functional $h(t):(t_0, \infty)\to\mathbb{R}^k$ and positive number $\delta > 0$, the weighted $L^\infty-$norm is defined as follows,
\begin{equation*}\label{e4:30}
\|h\|_\delta:=\|\mu_0(t)^{-\delta}h(t)\|_{L^\infty(t_0, \infty)}.
\end{equation*}
In the following, $\alpha > 0$ will always be a small constant.
We also assume the parameters $\lambda$, $\xi$, $\theta$, $\dot{\lambda}$, $\dot{\xi}$ and $\dot{\theta}$ belong to the following sets,
\begin{equation}\label{e4:21}
\|\dot{\lambda}(t)\|_{1+\sigma} + \|\dot{\xi}(t)\|_{1+\sigma} + \|\dot{\theta}(t)\|_{1+\sigma}\leq c,
\end{equation}
\begin{equation}\label{e4:22}
\|\lambda(t)\|_{1+\sigma} + \|\xi(t)-q\|_{1+\sigma} + \|\theta(t)\|_{1+\sigma}\leq c,
\end{equation}
here $c > 0$ is a constant independent of $R$, $t$ and $t_0$. We define the norm $\|\phi\|_{1+\alpha,1+\sigma}$ of $\phi$ as the least number $M>0$ such that the following estimate
\begin{equation}\label{e4:24}
(1+|y|)|\nabla_{y} \phi_j(y, t)| + |\phi_j(y, t)|\leq M\frac{\mu_{0}^{2+\sigma}}{1+|y|^{1+\alpha}}\text{ for }j=1, 2, 3, 4
\end{equation}
holds. For some small $\epsilon > 0$, we also suppose $\phi$ and $\psi$ satisfy the constraints
\begin{equation}\label{e4:25}
\|\phi\|_{1+\alpha,2+\sigma}\leq ce^{-\epsilon t_0}
\end{equation}
and
\begin{equation*}
\|\psi\|_{**,1+\sigma,1+\alpha}\leq ce^{-\epsilon t_0},
\end{equation*}
respectively. Then we have the following result.
\begin{prop}\label{l5:1}
The orthogonality conditions (\ref{e5:7}) are equivalent to the system
\begin{equation}\label{e5:9}
\left\{
\begin{aligned}
&\dot{\lambda} + 2\kappa_0\lambda = \Pi_0[\lambda,\xi, \theta, \dot{\lambda},\dot{\xi}, \dot{\theta}, \phi, \psi](t),\\
&\dot{\xi}_l = \Pi_l[\lambda,\xi, \theta, \dot{\lambda},\dot{\xi}, \dot{\theta}, \phi, \psi](t), \quad l = 1,\cdots, 4,\\
&\dot{\theta}_{12} = \mu_0^{-1}\Pi_{5}[\lambda,\xi, \theta, \dot{\lambda},\dot{\xi}, \dot{\theta}, \phi, \psi](t),\\
&\dot{\theta}_{13} = \mu_0^{-1}\Pi_{6}[\lambda,\xi, \theta, \dot{\lambda},\dot{\xi}, \dot{\theta}, \phi, \psi](t),\\
&\dot{\theta}_{14} = \mu_0^{-1}\Pi_{7}[\lambda,\xi, \theta, \dot{\lambda},\dot{\xi}, \dot{\theta}, \phi, \psi](t).
\end{aligned}
\right.
\end{equation}
Here $\kappa_0 = \frac{\pi^2}{6\Xi} > 0$; the right hand side terms of (\ref{e5:9}) can be written as
\begin{equation*}
\begin{aligned}
&\Pi_l[\lambda,\xi, \theta, \dot{\lambda},\dot{\xi}, \dot{\theta}, \phi, \psi](t)\\
&= e^{-\epsilon t_0}\mu_0^{1+\sigma}(t)f_l(t) + e^{-\epsilon t_0}\Theta_l\left[\dot{\lambda},\dot{\xi}, \mu_0\dot{\theta}, \lambda, (\xi-q), \mu_0\theta, \phi, \mu_0\psi\right](t),
\end{aligned}
\end{equation*}
for $l = 0, 1, \cdots, 7$, where $f_l(t)$ and $\Theta_l[\cdots](t)$ ($l = 0, \cdots, 7$) are bounded smooth functions for $t\in [t_0,\infty)$.
\end{prop}
\begin{proof}
{\bf Step 1}. We compute the integral
\begin{eqnarray*}
\int_{B_{2R}}\sum_{i=1}^4H_i[\lambda,\xi, \theta, \dot{\lambda},\dot{\xi}, \dot{\theta}, \phi, \psi](y,t(\tau))Z^0_i(y)dy,
\end{eqnarray*}
for $H( y,t(\tau))$ defined in (\ref{e5:2}). Observe that the main contribution to this integral comes from the term $\tilde{\mathcal L}[\Phi_0 + B_{1, q}]$. As we computed in Section \ref{blow-up-rate}, we have
\begin{eqnarray*}
\int_{B_{2R}}\sum_{i=1}^4\tilde{\mathcal L}_i[\Phi_0 + B_{1, q}]Z^0_i(y)dy = 144\frac{1}{\mu(t)}\int_{t_0}^{t}\frac{\mu(\tilde{s})\dot{\mu}(\tilde{s})}{(t-\tilde{s})^2}\Omega\left(\frac{\mu(\tilde{s})^2}{t-\tilde{s}}\right) d\tilde{s} + \frac{24\pi^2}{\mu}.
\end{eqnarray*}
Then using the arguments as in Section 2.6, we have
\begin{eqnarray*}
&&\mu^2\int_{B_{2R}}\sum_{i=1}^4\tilde{\mathcal L}_i[\Phi_0 + B_{1, q}]Z^0_i(y)dy = 144\mu\int_{t_0}^{t}\frac{\mu(\tilde{s})\dot{\mu}(\tilde{s})}{(t-\tilde{s})^2}\Omega\left(\frac{\mu(\tilde{s})^2}{t-\tilde{s}}\right) d\tilde{s} + 24\pi^2\mu\\
&&= \mu 144\int_{t_0}^{t}\frac{\mu_0(\tilde{s})\dot{\mu}_0(\tilde{s})}{\left(t-\tilde{s}\right)^2}
\Omega\left(\frac{\mu(\tilde{s})^2}{t-\tilde{s}}\right)d\tilde{s}\\
&&\quad+\mu 144\int_{t_0}^{t}\frac{\lambda(\tilde{s})\dot{\mu}(\tilde{s})}{\left(t-\tilde{s}\right)^2}
\Omega\left(\frac{\mu(\tilde{s})^2}{t-\tilde{s}}\right)d\tilde{s}+\mu 144\int_{t_0}^{t}\frac{\mu(\tilde{s})\dot{\mu}(\tilde{s})}{\left(t-\tilde{s}\right)^2}
\Omega'\left(\frac{\mu(\tilde{s})^2}{t-\tilde{s}}\right)\frac{2\mu(\tilde{s})\lambda(\tilde{s})}{t-\tilde{s}}d\tilde{s}\\
&&\quad +\mu 144\int_{t_0}^{t}\frac{\mu(\tilde{s})\dot{\lambda}(\tilde{s})}{\left(t-\tilde{s}\right)^2}
\Omega\left(\frac{\mu(\tilde{s})^2}{t-\tilde{s}}\right)d\tilde{s}\\
&&\quad  + 24\pi^2(\mu_0+\lambda)\\
&&= \mu 144\int_{t_0}^{t}\frac{\lambda(\tilde{s})\dot{\mu}(\tilde{s})}{\left(t-\tilde{s}\right)^2}
\Omega\left(\frac{\mu(\tilde{s})^2}{t-\tilde{s}}\right)d\tilde{s}\\
&&\quad+\mu 144\int_{t_0}^{t}\frac{\mu(\tilde{s})\dot{\mu}(\tilde{s})}{\left(t-\tilde{s}\right)^2}
\Omega'\left(\frac{\mu(\tilde{s})^2}{t-\tilde{s}}\right)\frac{2\mu(\tilde{s})\lambda(\tilde{s})}{t-\tilde{s}}d\tilde{s} +\mu 144\int_{t_0}^{t}\frac{\mu(\tilde{s})\dot{\lambda}(\tilde{s})}{\left(t-\tilde{s}\right)^2}
\Omega\left(\frac{\mu(\tilde{s})^2}{t-\tilde{s}}\right)d\tilde{s}\\
&&\quad  + 24\pi^2\lambda\\
&&=  144\Xi\kappa_0\lambda + 144\Xi\dot\lambda + 24\pi^2\lambda + O(\lambda\dot\lambda+\lambda^2).
\end{eqnarray*}
The other terms in $\int_{B_{2R}}\sum_{i=1}^4H_i[\lambda,\xi, \theta, \dot{\lambda},\dot{\xi}, \dot{\theta}, \phi, \psi](y,t(\tau))Z^0_i(y)dy$ can be dealt with similarly. Thus we obtain the equation for $\lambda$.

{\bf Step 2}. For $j = 1, 2, 3, 4$, we compute the integral
\begin{eqnarray*}
\int_{B_{2R}}\sum_{i=1}^4H_i[\lambda,\xi, \theta, \dot{\lambda},\dot{\xi}, \dot{\theta}, \phi, \psi](y,t(\tau))Z^j_i(y)dy,
\end{eqnarray*}
for $H( y,t(\tau))$ defined in (\ref{e5:2}). Observe that the main contribution to this integral comes from the term $\frac{\dot\xi_i}{\mu^2}(t)Z^i(y)|_{y=\frac{x-\xi(t)}{\mu(t)}}$. Similarly to {\bf Step 1}, we have
\begin{eqnarray*}
\mu^2\int_{B_{2R}}\sum_{j=1}^4\frac{\dot\xi_i}{\mu^2}(t)Z^i_j(y)Z^i_j(y)dy = \Pi_l[\lambda,\xi, \theta, \dot{\lambda},\dot{\xi}, \dot{\theta}, \phi, \psi](t),
\end{eqnarray*}
Thus we have the equation for $\xi_i$.

{\bf Step 3}. For $j = 5$, we compute the integral
\begin{eqnarray*}
\int_{B_{2R}}\sum_{i=1}^4H_i[\lambda,\xi, \theta, \dot{\lambda},\dot{\xi}, \dot{\theta}, \phi, \psi](y,t(\tau))Z^5_i(y)dy,
\end{eqnarray*}
for $H( y,t(\tau))$ defined in (\ref{e5:2}). Observe that the main contribution to this integral comes from the term $\tilde{\mathcal L}[\phi]$ with $$\Phi_1^{(1)} = dx\wedge d\bar x\left(\psi^{(12)}(z, t)(x-\xi(t))_2\frac{\partial}{\partial x_2}+\psi^{(34)}(z, t)(x-\xi(t))_4\frac{\partial}{\partial x_3}, \cdot\right).$$ As we computed in Section 2.5, we have
\begin{equation*}
\begin{aligned}
\mu^2\int_{B_{2R}}\sum_{i=1}^4\tilde{\mathcal L}_i[\Phi_1^{(1)}]Z^5_i(y)dy & = 144\mu\int_{t_0}^{t}\frac{\dot \theta_{12}(\tilde{s})+\dot \theta_{34}(\tilde{s})}{(t-\tilde{s})^2}\Omega\left(\frac{\mu(\tilde{s})^2}{t-\tilde{s}}\right) d\tilde{s}\\
&= \Pi_l[\lambda,\xi, \theta, \dot{\lambda},\dot{\xi}, \dot{\theta}, \phi, \psi](t),
\end{aligned}
\end{equation*}
Since $\theta$ satisfies $\theta_{12}(t) = \theta_{34}(t)$,  we thus have the equation for $\theta_{12}$.
\end{proof}

\subsection{The outer problem.}
To apply the Schauder fixed-point theorem to the outer problem (\ref{e:outerproblem}) and obtain a solution $\psi$, we consider the following linear problem first,
\begin{equation}\label{e4:3}
\begin{cases}
\partial_t\psi =
\Delta\psi + V_{\mu, \xi}(x, t)\psi + f(x, t)\quad&\text{in}\,\,\mathbb{R}^4\times (t_0, \infty),\\
\lim_{|x|\to +\infty}\psi(x, t) = 0\quad&\text{for all}\,\, t\in (t_0,\infty),
\end{cases}
\end{equation}
where $f(x, t)$ is a smooth function. Here $V_{\mu, \xi} \sim (1-\eta_R)\mu_0^{-2}\frac{1}{1+|y|^2}$ with $y = \frac{x-\xi(t)}{\mu_0(t)}$.  
Now we assume that for $\alpha$, $\beta>0$, $f(x, t)$ satisfies the following estimate
\begin{equation}\label{gg}
\begin{aligned}
|f(x, t)|\leq M\frac{\mu_0^{-2}(t)\mu_0^{\beta}(t)}{1+| y|^{2+\alpha}}, \quad  y =\frac{x-\xi(t)}{\mu_0(t)}
\end{aligned}
\end{equation}
and the least number $M>0$ satisfying (\ref{gg}) is denoted as $\|f\|_{*,\beta,2+\alpha}$.
\begin{prop}\label{l4:lemma4.1}
Suppose $\|f\|_{*,\beta,2+\alpha} < +\infty$ for some constants $\beta > 0$, $\alpha \in (1, 2)$. Let $\psi = \psi[f]$ be the solution of (\ref{e4:3}) given by the Duhamel formula (\ref{solutiongivenbyheatkernel}), then it holds that
\begin{equation}\label{e4:40}
|\psi(x, t)|\lesssim
\left\{
\begin{aligned}
&\|f\|_{*,\beta,2+\alpha}\frac{\mu_0^{\beta}}{1+|y|^{\alpha'}}, \quad |y|=\left|\frac{x-\xi}{\mu_0}\right|\leq \mu^{-1}_0,\\
&\|f\|_{*,\beta,2+\alpha}\mu_0^{\beta+\alpha'}(t), \quad |y|=\left|\frac{x-\xi}{\mu_0}\right| > \mu^{-1}_0\\
\end{aligned}
\right.
\end{equation}
and
\begin{equation}\label{e4:5}
|\nabla\psi(x, t)|\lesssim \|f\|_{*,\beta,2+\alpha}\frac{\mu_0^{\beta-1}}{1+|y|^{\alpha'+1}}\text{ for } |y|\leq 2R.
\end{equation}
Here $\alpha' = \alpha - 2\kappa > 0$ for a small constant $\kappa \in (0, \frac{\alpha}{2})$. 
\end{prop}
We will give the proof of Proposition \ref{l4:lemma4.1} in Section 5. This result will be applied to the outer problem (\ref{e:outerproblem}) with
\begin{equation*}
\begin{aligned}
V_j(x, t)\psi& =  (1-\eta_R)\sum_{i = 1}^4[\psi_i, \partial_i B_{\mu, \xi, j}] + (1-\eta_R)\sum_{i = 1}^4[B_{\mu, \xi, i}, \partial_i\psi_j]\\
&\quad + (1-\eta_R)\sum_{i = 1}^4[\psi_i, \partial_i B_{\mu, \xi, j} - \partial_j B_{\mu, \xi, i} + [B_{\mu, \xi, i}, B_{\mu, \xi, j}]]\\
&\quad + (1-\eta_R)\sum_{i = 1}^4[B_{\mu, \xi, i}, \partial_i\psi_j - \partial_j\psi_i + [\psi_i, B_{\mu, \xi, j}]+[B_{\mu, \xi, i}, \psi_j]]\\
&\quad + (1-\eta_R)\partial_j\sum_{i = 1}^4[B_{\mu, \xi, i}, \psi_i] + (1-\eta_R)\sum_{i = 1}^4[B_{\mu, \xi, j}, [B_{\mu, \xi, i}, \psi_i]]
\end{aligned}
\end{equation*}
and
\begin{equation*}
\begin{aligned}
f_j(x, t) =   N_j[A_{\mu, \xi, \theta}^*-B_{\mu, \xi}+\varphi] & + (1-\eta_R)\mathcal E^*_j(x, t)  + \nabla\eta_{R}\nabla\tilde{\phi}_j \\
&+ \tilde{\phi}_j\big(\Delta -\partial_t\big)\eta_R,\,\, j = 1, 2, 3, 4.
\end{aligned}
\end{equation*}

\begin{prop}\label{propositionestimate2}
For $j = 1, 2, 3, 4$, we have the following estimates.
\begin{itemize}
\item[(1)]
\begin{equation*}\label{e4:403a}
\begin{aligned}
|\tilde{\phi}_j\big(\Delta -\partial_t\big)\eta_R|&\lesssim \frac{\mu_{0}^{-2}\mu_{0}^{1+\sigma}}{1+|y|^{3+\alpha}}\|\phi_j\|_{1+\alpha,2+\sigma},
\end{aligned}
\end{equation*}
\item[(2)]
\begin{equation*}\label{e4:403b}
\begin{aligned}
|\nabla\eta_{R}\nabla\tilde{\phi}_j|&\lesssim \frac{\mu_{0}^{-2}\mu_{0}^{1+\sigma}}{1+|y|^{3+\alpha}}\|\phi_j\|_{1+\alpha,2+\sigma},
\end{aligned}
\end{equation*}
\item[(3)]
\begin{equation*}\label{e4:404}
\begin{aligned}
&|N_j[A_{\mu, \xi, \theta}^*-B_{\mu, \xi}+\varphi]|\lesssim \frac{\mu_{0}^{-2}\mu_{0}^{1+\sigma}}{1+|y|^{3+\alpha}}\left(\|\phi_j\|_{1+\alpha, 2+\sigma}^2+\|\psi\|_{**,1+\sigma,1+\alpha}^2\right),
\end{aligned}
\end{equation*}
\item[(4)]
\begin{equation*}\label{e4:405}
\begin{aligned}
&\left|(1-\eta_R)\mathcal E^*_j(x, t)\right| \lesssim e^{-\varepsilon t_0}\frac{\mu_{0}^{-2}\mu_{0}^{1+\sigma}}{1+|y|^{3+\alpha}}.
\end{aligned}
\end{equation*}
\end{itemize}
\end{prop}
\begin{proof}

{\bf Proof of (1)}: We have
\begin{equation*}
\begin{aligned}
|\tilde{\phi}_j\Delta\eta_R|\lesssim \frac{1}{R^2}\frac{\mu_{0}^{-2}\mu_{0}^{1+\sigma}}{1+|y|^{1+\alpha}}\|\phi_j\|_{1+\alpha,2+\sigma}\lesssim \frac{\mu_{0}^{-2}\mu_{0}^{1+\sigma}}{1+|y|^{3+\alpha}}\|\phi_j\|_{1+\alpha,2+\sigma}
\end{aligned}
\end{equation*}
and
\begin{equation*}
\begin{aligned}
|\tilde{\phi}_j\partial_t\eta_R|&\lesssim \frac{\mu_0}{R}\frac{\mu_{0}^{-2}\mu_{0}^{1+\sigma}}{1+|y|^{1+\alpha}}\|\phi_j\|_{1+\alpha,2+\sigma}\lesssim \frac{\mu_{0}^{-2}\mu_{0}^{1+\sigma}}{1+|y|^{3+\alpha}}\|\phi_j\|_{1+\alpha,2+\sigma}.
\end{aligned}
\end{equation*}

\noindent{\bf Proof of (2)}:
We have
\begin{equation*}
|\nabla\eta_{R}\nabla\tilde{\phi}_j|\lesssim \frac{1}{R}\frac{\mu_{0}^{-2}\mu_{0}^{1+\sigma}}{1+|y|^{2+\alpha}}\|\phi_j\|_{1+\alpha, 2+\sigma}\lesssim \frac{\mu_{0}^{-2}\mu_{0}^{1+\sigma}}{1+|y|^{3+\alpha}}\|\phi_j\|_{1+\alpha, 2+\sigma}.
\end{equation*}

\noindent{\bf Proof of (3)}: From the definition in (\ref{nonlineatterm}), we have
\begin{equation*}
\begin{aligned}
|N_j[A_{\mu, \xi, \theta}^*-B_{\mu, \xi}+\varphi]|&\lesssim \mu_0^{-3}\frac{\mu_{0}^{2+2\alpha}}{1+|y|^{3+2\sigma}}\left(\|\phi_j\|_{1+\alpha, 2+\sigma}+\|\psi\|_{**,1+\sigma,1+\alpha}\right)^2\\
&\quad +\mu_0^{-3}\frac{\mu_{0}^{3+3\sigma}}{1+|y|^{3+3\alpha}}\left(\|\phi_j\|_{1+\alpha, 2+\sigma}+\|\psi\|_{**,1+\sigma,1+\alpha}\right)^3\\
&\lesssim \frac{\mu_{0}^{-2}\mu_{0}^{1+\sigma}}{1+|y|^{3+\alpha}}\left(\|\phi_j\|_{1+\alpha, 2+\sigma}^2+\|\psi\|_{**,1+\sigma,1+\alpha}^2\right).
\end{aligned}
\end{equation*}

\noindent{\bf Proof of (4)}: We have
\begin{equation*}
\begin{aligned}
|(1-\eta_R)\mathcal E^*_j(x, t)|\lesssim \frac{1}{R^{1-\alpha}}\mu_{0}^{-2}\frac{\mu_{0}^{1+\sigma}}{1+|y|^{3+\alpha}}\lesssim e^{-\varepsilon t_0}\frac{\mu_{0}^{-2}\mu_{0}^{1+\sigma}}{1+|y|^{3+\alpha}}.
\end{aligned}
\end{equation*}
\end{proof}

\subsection{Proof of Theorem \ref{t:main}: Solving the inner-outer gluing system.}
We now reformulate the existence to the inner problem (\ref{e:innerproblembeforechangeofvariables2}) and the outer problem (\ref{e:outerproblem}) as a fixed point problem; then we will use Schauder fixed point theorem to find a solution.

{\bf Step 1}. Suppose $h$ is a function satisfying the assumption $\|h\|_{1+\sigma} \lesssim e^{-\varepsilon t_0}$. Then it is well known that the solution of
\begin{equation}\label{e5:14}
\dot{\lambda} + (\kappa_0 + c_0)\lambda = h(t)
\end{equation}
is given by
\begin{equation}\label{e5:15}
\lambda(t) = e^{-(\kappa_0 + c_0)t}\left[d + \int_{t_0}^te^{(\kappa_0 + c_0)\tau} h(\tau)d\tau\right],
\end{equation}
here $d$ is an arbitrary constant. Therefore, we can estimate as follows,
\begin{equation*}
\|e^{(1+\sigma)\kappa_0t}\lambda(t)\|_{L^\infty(t_0,\infty)}\lesssim e^{-(c_0-\sigma\kappa_0)t_0}d + \|h\|_{1+\sigma}
\end{equation*}
and
\begin{equation*}
\|\dot{\lambda}(t)\|_{1+\sigma}\lesssim e^{-(c_0-\sigma\kappa_0)t_0}d + \|h\|_{1+\sigma}
\end{equation*}
when the positive constant $\sigma$ is chosen in the interval $(0, \frac{c_0}{\kappa_0})$.

Denoting $\Lambda(t) = \dot{\lambda}(t)$, we have the following relation
\begin{equation}\label{e5:1700}
\Lambda + (\kappa_0 + c_0)\int_{t}^\infty\Lambda(s)ds = h(t),
\end{equation}
from which we know that there exists a bounded linear operator $\mathcal{L}_1: h\to \Lambda$ by assigning the solution $\Lambda$ of (\ref{e5:1700}) to any function $h$ satisfying the assumption $\|h\|_{1+\sigma} < +\infty$. Furthermore, $\mathcal{L}_1$ is continuous between the linear space $L^\infty(t_0, \infty)$ endowed with $\|\cdot\|_{1+\sigma}$-topology.

For any vector function $h: (t_0,\infty)\to \mathbb{R}^n$ satisfying the condition $\|h\|_{1+\sigma} < +\infty$, the solution of the following equation
\begin{equation}\label{e5:19}
\dot{\xi} = h(t)
\end{equation}
can be expressed as follows,
\begin{equation}\label{e5:2000}
\xi(t) = \xi^0(t) + \int_{t}^\infty h(s)ds,
\end{equation}
with
\begin{equation*}
\xi^0(t) = q.
\end{equation*}
Then we have
\begin{equation*}
|\xi(t) - q|\lesssim e^{-(1+\sigma)\kappa_0t}\|h\|_{1+\sigma}
\end{equation*}
and
\begin{equation*}
\|\dot{\xi} - \dot{\xi}^0\|_{1+\sigma}\lesssim \|h\|_{1+\sigma}.
\end{equation*}
Now we define $\Xi(t) = \dot{\xi}(t) - \dot{\xi}^0$, then (\ref{e5:2000}) gives us a bounded linear operator $\mathcal{L}_2: h\to \Xi$ between the linear space $L^\infty(t_0, \infty)$ endowed with the $\|\cdot\|_{1+\sigma}$-topology. Similarly, from Proposition \ref{l5:1}, there exists a bounded linear operator $\mathcal{L}_3: h\to \Upsilon:= \dot{\theta}(t)$ between the linear space $L^\infty(t_0, \infty)$ endowed with the $\|\cdot\|_{\sigma}$-topology.
Observe that ($\lambda$, $\xi$, $r$) is a solution of (\ref{e5:9}) if ($\Lambda = \dot{\lambda}(t)$, $\Xi = \dot{\xi}(t) - \dot{\xi}^0(t)$, $\Upsilon:= \dot{\theta}(t)$) is a fixed point of the following problem
\begin{equation}\label{e5:23}
(\Lambda, \Xi, \Upsilon) = \mathcal{T}_0(\Lambda, \Xi, \Upsilon)
\end{equation}
where
\begin{equation*}
\begin{aligned}
\mathcal{T}_0& := \left(\mathcal{L}_0(\hat{\Pi}_1[\Lambda, \Xi, \Upsilon, \phi, \psi], \mathcal{L}_2(\hat{\Pi}_1[\Lambda, \Xi, \Upsilon, \phi, \psi]), \cdots, \mathcal{L}_2(\hat{\Pi}_4[\Lambda, \Xi, \Upsilon, \phi, \psi]), \right.\\
& \quad\quad\quad \left.\mathcal{L}_3(\mu_0^{-1}\hat{\Pi}_5[\Lambda, \Xi, \Upsilon, \phi, \psi], \mathcal{L}_3(\mu_0^{-1}\hat{\Pi}_6[\Lambda, \Xi, \Upsilon, \phi, \psi], \mathcal{L}_3(\mu_0^{-1}\hat{\Pi}_7[\Lambda, \Xi, \Upsilon, \phi, \psi]\right)\\
& := \left(\bar{A}_0(\Lambda, \Xi, \Upsilon, \phi, \psi), \bar{A}_2(\Lambda, \Xi, \Upsilon, \phi, \psi),\cdots,  \bar{A}_7(\Lambda, \Xi, \Upsilon, \phi, \psi)\right)
\end{aligned}
\end{equation*}
with
\begin{equation*}
\hat{\Pi}_l[\Lambda, \Xi, \Upsilon, \phi, \psi] := \Pi_l\left[\int_{t}^\infty \Lambda, q + \int_{t}^\infty\Xi, \int_{t}^\infty \Upsilon,  \Lambda, \Xi, \Upsilon, \phi, \psi\right]
\end{equation*}
for $l = 0, 1, \cdots, 7$.

{\bf Step 2}. From Proposition \ref{proposition5.1} we know that there is a bounded linear operator $\mathcal{T}_1$ assigning to any function $h(y,\tau)$ with $\|h\|_{3 + \alpha, \sigma}$-bounded the solution of (\ref{e5:31}). Therefore the solution of problem (\ref{e5:4}) is a fixed point of the following problem
\begin{equation}
\phi = \mathcal{T}_1(H[\lambda,\xi, \theta, \dot{\lambda},\dot{\xi}, \dot{\theta}, \phi, \psi](y,t(\tau))).
\end{equation}

{\bf Step 3}. From Proposition \ref{l4:lemma4.1} we know that there is a bounded linear operator $\mathcal{T}_2$ assigning to any given differential 1-form $f(x, t)$ the solution $\psi = \mathcal{T}_2(f)$ for problem (\ref{e4:3}). Therefore, $\psi$ is a solution of the outer problem (\ref{e:outerproblem}) if $\psi$ is a fixed point of the operator
\begin{equation*}\label{e4:34}
\mathcal{A}(\psi):=\mathcal{T}_2(\mathbf f),
\end{equation*}
with
\begin{equation}\label{e4:35}
f_j(x, t) =  N_j[A_{\mu, \xi, \theta}^*-B_{\mu, \xi}+\varphi] + (1-\eta_R)\mathcal E^*_j(x, t)  + \nabla\eta_{R}\nabla\tilde{\phi}_j + \tilde{\phi}_j\big(\Delta -\partial_t\big)\eta_R,
\end{equation}
$j=1, 2, 3, 4$ and $\mathbf f = \sum_{j=1}^4f_jdx_j$.
Equivalently, we need to solve the following fixed point problem
\begin{equation}
\psi = \mathcal{T}_2(\mathbf f).
\end{equation}

From {\bf Step 1-3}, to obtain a solution, we need to solve the following fixed point problem with unknown functions $(\phi, \psi, \lambda, \xi, \theta)$,
\begin{equation}\label{inner_outer_gluing_system00}
\left\{
\begin{aligned}
&(\Lambda, \Xi, \Gamma) = \mathcal{T}_0(\Lambda, \Xi, \Gamma),\\
&\phi = \mathcal{T}_1(H[\lambda,\xi, \theta, \dot{\lambda},\dot{\xi}, \dot{\theta}, \phi, \psi](y,t(\tau))),\\
&\psi = \mathcal{T}_2(\mathbf f).
\end{aligned}
\right.
\end{equation}
To this aim, we use the Schauder fixed-point theorem in the following set
\begin{equation*}\label{convexsetforschauder}
\begin{aligned}
\mathcal{B} &= \Bigg\{(\phi, \psi, \lambda, \xi, \theta, \dot{\lambda}, \dot{\xi}, \dot{\theta}): \|\dot{\lambda}(t)\|_{1+\sigma} + \|\dot{\xi}(t)\|_{1+\sigma}\\
&\quad\quad\quad  + \|\dot{\theta}(t)\|_{\sigma} + \|\lambda(t)\|_{1+\sigma}  + \|\xi(t)-q\|_{1+\sigma} + \|\theta\|_{\sigma} + e^{\varepsilon t_0}\|\psi\|_{**,1+\sigma,1+\alpha}\\
&\quad\quad\quad  + e^{\varepsilon t_0}\|\phi\|_{2+\sigma, 1+\alpha} \leq c_{\mathcal B}\Bigg\}
\end{aligned}
\end{equation*}
for a fixed but large constant $c_{\mathcal B}>0$ which will be made precise as follows.

Let
\begin{equation*}
K := \max\{\|f_0\|_{1+\sigma}, \|f_1\|_{1+\sigma},\cdots, \|f_{7}\|_{1+\sigma}\}
\end{equation*}
where $f_0$, $f_1$, $\cdots$, $f_{7}$ are the functions defined in Proposition \ref{l5:1}. Then we have the following estimate
\begin{equation*}
\begin{aligned}
&\left|e^{(1+\sigma)\kappa_0t}\bar{A}_i(\Lambda, \Xi, \Upsilon, \phi, \psi)\right|\\
&\leq \tilde C\left(  e^{-(c_0-\sigma\kappa_0)t_0}d + \|\phi\|_{1+\alpha, 2+\sigma}+\|\psi\|_{**,1+\sigma,1+\alpha} + K\right) \\ &\quad +\tilde C\left(e^{-\varepsilon t_0}\left(\|\Lambda\|_{1+\sigma} + \|\Xi\|_{1+\sigma} + \|\Upsilon\|_{\sigma}\right)\right) \\
&\leq \tilde C\left(2K + e^{-\varepsilon t_0}c_{\mathcal{B}}\right) < c_{\mathcal B}.
\end{aligned}
\end{equation*}
In the above, we choose the constant $d$ satisfying the condition $e^{-(c_0-\sigma\kappa_0)t_0}d < K$ and $c_{\mathcal B} > \frac{2\tilde C K}{1-e^{-\varepsilon t_0}}$, therefore we have $\mathcal{T}_0(\mathcal{B})\subset \mathcal{B}$ (the constant $\rho$ in (\ref{definitionofR}) is chosen sufficiently small).

On the set $\mathcal{B}$, from the estimates at the end of Section 3.1, we know
\begin{equation*}\label{e6:1}
\begin{aligned}
\left|H[\lambda,\xi, \theta, \dot{\lambda},\dot{\xi}, \dot{\theta}, \phi, \psi](y,t(\tau))\right|\leq C_1 e^{-\varepsilon t_0}\frac{\mu_0^{2+\sigma}}{1+|y|^{3+\alpha}}
\end{aligned}
\end{equation*}
Using Proposition \ref{proposition5.1}, it holds that $e^{\varepsilon t_0}\|\phi\|_{2+\sigma, 1+\alpha}\leq \|\mathcal T_1\|C_1$, hence  $\mathcal{T}_1(\mathcal{B})\subset \mathcal{B}$ when $c_{\mathcal B} > \|\mathcal T_1\|C_1$. 
Similarly, Proposition \ref{l4:lemma4.1} and Proposition \ref{propositionestimate2} ensure $\mathcal{T}_2(\mathcal{B})\subset \mathcal{B}$. From these we know that the operator $\mathcal{T}$ defined in the inner-outer gluing system (\ref{inner_outer_gluing_system00}) maps the set $\mathcal{B}$ into itself. Since $\lambda$, $\xi$, $\theta$, $\dot{\lambda}$, $\dot{\xi}$, $\dot{\theta}$, $\phi$ and $\psi$ decay uniformly as $t\to +\infty$, standard parabolic estimates ensure that $\mathcal{T}$ is compact. Therefore by the Schauder fixed-point theorem, the inner-outer gluing system (\ref{inner_outer_gluing_system00}) has a fixed point in $\mathcal{B}$. Thus we find a solution to the system of (\ref{e:innerproblembeforechangeofvariables2}) and (\ref{e:outerproblem}), which gives us a solution of
\begin{equation*}
\left\{\begin{array}{ll}
        \frac{\partial A}{\partial t} = -D_A^*F_A + D_{A}D_{A}^*\left(A_{\mu, \xi, \theta}^*-B_{\mu, \xi} + \varphi\right)\text{ in }\mathbb{R}^4\times [t_0, \infty),\\
        A(\cdot, t_0) = A_0\text{ in }\mathbb{R}^4
       \end{array}
\right.
\end{equation*}
when $t_0 > 0$ is large enough. By the Donaldson-De Turck trick as discussed at the end of the introduction, there exists a unique solution $S\in C^\infty(\mathbb{R}^4\times [t_0, \infty))$ of the following problem
$$
S^{-1}\frac{\partial S}{\partial t} = -D_A^*\left(A_{\mu, \xi, \theta}^*-B_{\mu, \xi} + \varphi\right)\text{ on }\mathbb{R}^4\times [t_0, \infty), S(0) = id_{\mathbb{R}^4\times SU(2)}.
$$
Now we define $\tilde A: = (S^{-1})^*A$, then $\tilde A$ is a strong solution to the Yang-Mills gradient flow
\begin{equation*}
\left\{\begin{array}{ll}
        \frac{\partial \tilde A}{\partial t} = -D_{\tilde A}^*F_{\tilde A}\text{ in }\mathbb{R}^4\times [t_0, \infty),\\
        {\tilde A}(\cdot, t_0) = A_0\text{ in }\mathbb{R}^4.
       \end{array}
\right.
\end{equation*}
Therefore $\bar A(\cdot, t) = \tilde A(\cdot, t+t_0)$ is a solution of (\ref{e:Yangmillsheatflow}). This completes the proof of Theorem \ref{t:main}.

\section{Proof of Proposition \ref{proposition5.1}}
In this section, we prove Proposition \ref{proposition5.1}. Consider the solvability of the following linear problem
\begin{equation}\label{e:linearproblem}
\left\{
\begin{aligned}
&\partial_\tau \phi_j = \mathcal L_j[\phi] + h_j\\
&\quad\quad = \Delta \phi_j + \sum_{i = 1}^4[\phi_i, \partial_i B_j] + \sum_{i = 1}^4[B_i, \partial_i\phi_j]  + \sum_{i = 1}^4[\phi_i, \partial_i B_j - \partial_j B_i + [B_i, B_j]]\\
&\quad\quad\quad + \sum_{i = 1}^4[B_i, \partial_i\phi_j - \partial_j\phi_i + [\phi_i, B_j]+[B_i, \phi_j]]\\
&\quad\quad\quad + \partial_j\sum_{i = 1}^4[B_i, \phi_i] + \sum_{i = 1}^4[B_j, \partial_i\phi_i] + \sum_{i = 1}^4[B_j, [B_i, \phi_i]] + h_j(y, \tau)\\
&\quad\quad\quad\quad\quad\quad\quad\quad\quad\quad\quad\quad\quad\quad\quad\quad\quad\quad\quad\quad\quad\quad\quad\quad\quad\quad\quad\text{ in }\mathbb{R}^4 \times [\tau_0, +\infty)\\
&\phi_j(\cdot, \tau_0)  = 0\text{ in }\mathbb{R}^4, \,\, j = 1,\cdots, 4.
\end{aligned}
\right.
\end{equation}
Here $h(y, \tau)  = \sum_{j = 1}^4 h_j(y, \tau)dx_j: \mathbb{R}^4\times [\tau_0, +\infty)\to \Im \mathbb H\,\, d\bar x$ with support in the ball $B_{2R}(0)$.
First, we have the following property.
\begin{lemma}\label{l5555}
Suppose $\|h\|_{3+\alpha, \nu} < +\infty$ and
\begin{equation}
\int_{\mathbb{R}^4}\sum_{i=1}^4h_i(y, \tau)Z_i^j(y)dy = 0\text{ for } j=0, 1, 2, \cdots, 7.
\end{equation}
Then we have \begin{equation}\label{orthogalforphi}
\int_{B_{2R}}\sum_{i=1}^4\phi_i(y, \tau)Z_i^j(y)dy = 0\text{ for } j=0, 1, 2, \cdots, 7, \,\, \tau\in [\tau_0, +\infty).
\end{equation}
\end{lemma}
\begin{proof}
Let us test equation (\ref{e:linearproblem}) with the functions
\begin{equation*}
Z^{j}_i\eta,\quad \eta(y) = \eta_0(|y|/R)
\end{equation*}
and sum for index $i = 1, 2, 3, 4$,
where $\eta_0$ is a smooth cut-off function satisfying $\eta_0(r) = 1$ for $r < 1$, $\eta_0(r) = 0$ for $r > 2$ and $R > 0$ is a large constant. Using the fact that $\mathcal L$ is self-adjoint, we have the following relation
\begin{equation*}
\int_{\mathbb{R}^4}\sum_{i=1}^4\phi_i(\cdot, \tau) Z^j_i\eta dy = \int_{0}^\tau ds\int_{\mathbb{R}^4}\left(\sum_{i=1}^4\phi_i(\cdot, s)\cdot \mathcal L_i[\eta Z^j] + \sum_{i=1}^4h_i\cdot Z^j_i\eta\right)dy.
\end{equation*}
On the other hand,
\begin{equation*}
\begin{aligned}
&\int_{\mathbb{R}^4}\left(\sum_{i=1}^4\phi_i(\cdot, s)\cdot \mathcal L_i[\eta Z^j] + \sum_{i=1}^4h_i\cdot Z^j_i\eta\right)\\
&\quad = \int_{\mathbb{R}^4}\sum_{i=1}^4\phi_i\cdot (Z^j_i\Delta\eta + \nabla\eta\cdot \nabla Z^j_i) - \sum_{i=1}^4h_i\cdot Z^j_i(1-\eta)\\
&\quad = O(R^{-\varsigma})
\end{aligned}
\end{equation*}
uniformly in $\tau\in (\tau_0,\tau_1)$,  where $\tau_1 > 0$ is an arbitrary large constant and $\varsigma > 0$ is a small constant.
Now let $R\to +\infty$ to  obtain the relation (\ref{orthogalforphi}).
\end{proof}
\begin{lemma}\label{l5:3}
Suppose $\alpha\in (0, 1)$, $\nu > 0$, $\|h\|_{3+\alpha, \nu} < +\infty$ and
\begin{equation}\label{e:20200301}
\int_{\mathbb{R}^4}\sum_{i=1}^4h_i(y, \tau)Z_i^j(y)dy = 0\text{ for } j=0, 1, 2, \cdots, 7.
\end{equation}
Then, for $\tau_1 \in (\tau_0, +\infty)$ large enough, any solution of (\ref{e:linearproblem}) satisfies the estimate
\begin{equation*}\label{e5:35}
\|\phi(y,\tau)\|_{1+\alpha,\tau_1}\lesssim \|h\|_{3+\alpha,\tau_1}.
\end{equation*}
Here, $\|g\|_{b, \tau_1}:=\sup_{\tau\in (\tau_0,\tau_1)}\tau^\nu\|(1+| y|^b)g\|_{L^\infty(\mathbb{R}^4)}$.
\end{lemma}
\begin{proof}
Suppose $\|h\|_{3+\alpha, \nu} < +\infty $ and $\phi$ is a solution of problem (\ref{e:linearproblem}).
Given $\tau_1>\tau_0$, we then have $\|\phi\|_{1+\alpha, \tau_1} < +\infty$ and from Lemma \ref{l5555}, there holds
\begin{equation}\label{orthognal111}
\int_{B_{2R}}\sum_{i=1}^4\phi_i(y, \tau)Z_i^j(y)dy = 0  \text{ for all }\tau\in (\tau_0,\tau_1),\,\, j=0, 1, 2,\cdots, 7.
\end{equation}
Therefore we need to prove that there is a constant $C > 0$ such that, if $\tau_1>\tau_0$ is large enough,  then any solution $\phi$ of (\ref{e:linearproblem}) with properties $\|\phi\|_{1+\alpha,\tau_1}<+\infty$ and (\ref{orthognal111}) satisfies the estimate
\begin{equation*}\label{esti01}
\|\phi\|_{1+\alpha,\tau_1}\leq C\|h\|_{3+\alpha,\tau_1}.
\end{equation*}

By contradiction, we assume that there exist sequences $\tau_1^n\to +\infty$, $\phi^n$ and $h^n$ satisfying the following
\begin{equation*}
\begin{aligned}
\partial_\tau \phi^n &= \mathcal L[\phi^n] + h^n \text{ in }\mathbb{R}^4 \times [\tau_0, \tau_1^n),\\
\int_{\mathbb{R}^4}\sum_{i=1}^4\phi^n_i(y, \tau)\cdot Z^{j}_idy &=0\text{ for all }\tau\in (\tau_0,\tau_1^n),\,\,j = 0, 1, 2, \cdots, 7,\\
\phi^n(y,\tau_0)& = 0\text{ in }\mathbb{R}^4.
\end{aligned}
\end{equation*}
and
\begin{equation}\label{e6666}
\|\phi^n\|_{1+\alpha,\tau_1^n}=1,\,\,\|h^n\|_{3+\alpha,\tau_1^n}\to 0.
\end{equation}
First, we claim that there holds
\begin{equation}\label{e5:37}
\sup_{\tau_0 <\tau <\tau_1^n}\tau^\nu |\phi^n(y,\tau)|\to 0
\end{equation}
uniformly on compact subsets of $\mathbb{R}^4$.

Indeed, if (\ref{e5:37}) is not true, then there is a sequence of points $\{y_n\}$ on $\mathbb{R}^4$ satisfying $|y_n|\leq M$ and a sequence $\{\tau_2^n\}$ satisfying $\tau_0 < \tau_2^n < \tau_1^n$, such that
\begin{equation*}
(\tau_2^n)^\nu|y_n|^{\alpha+1}|\phi^n(y_n,\tau_2^n)|\geq \frac{1}{2}.
\end{equation*}
Then we have $\tau_2^n\to +\infty$. Now we define
\begin{equation*}
\bar{\phi}^n(y,t) = (\tau_2^n)^\nu \phi^n(y,\tau_2^n+t).
\end{equation*}
Then $\bar\phi^n$ satisfies the following equation
\begin{equation*}
\partial_t\bar{\phi}^n = \mathcal L[\bar{\phi}^n] + \bar{h}^n\text{ in }\mathbb{R}^4\times (\tau_0-\tau_2^n,0].
\end{equation*}
Here $\bar{h}^n(y, t): = (\tau_2^n)^\nu h^n(y,\tau_2^n+t)\to 0$ uniformly on compact subsets of $\mathbb{R}^4\times (-\infty, 0]$, furthermore, there holds
\begin{equation*}
|\bar{\phi}^n(y,t)|\leq \frac{1}{1+|y|^{\alpha+1}}\text{ in }\mathbb{R}^4\times (\tau_0-\tau_2^n,0].
\end{equation*}
Using the dominated convergence theorem, we know that there exists $\bar\phi$ such that $\bar{\phi}^n\to\bar{\phi}\neq 0$ uniformly on compact subsets of $\mathbb{R}^4\times (-\infty, 0]$ and satisfies the relations,
\begin{equation}\label{e5:103}
\begin{aligned}
\partial_t\bar{\phi} & = \mathcal L[\bar{\phi}]\text{ in }\mathbb{R}^4\times (-\infty, 0],\\
\int_{\mathbb{R}^4}\sum_{i=1}^4\bar{\phi}_i(y, t)\cdot Z^j_i(y)dy & = 0\text{ for all } t\in (-\infty, 0],\,\, j= 0, 1, 2,\cdots, 7,\\
|\bar{\phi}(y,t)|& \leq \frac{1}{1+|y|^{\alpha+1}}\text{ in }\mathbb{R}^4\times (-\infty, 0],\\
\bar{\phi}(y,t_0) & = 0,\,\, y\in \mathbb{R}^4.
\end{aligned}
\end{equation}
Now we prove that $\bar{\phi} = 0$, which is a contradiction. Indeed, there holds
\begin{equation*}
\frac{1}{2}\partial_t\int_{\mathbb{R}^4}|\bar{\phi}_t|^2 + B(\bar{\phi}_t, \bar{\phi}_t) = 0,
\end{equation*}
with
\begin{equation*}
B(\bar{\phi}, \bar{\phi}) = \int_{\mathbb{R}^4}\mathcal L[\bar{\phi}]\cdot \bar{\phi}dy.
\end{equation*}
Since $\int_{\mathbb{R}^4}\sum_{i=1}^4\bar{\phi}_i(y, t)\cdot Z^j_i(y)dy = 0$ for all $t\in (-\infty, 0]$, $j= 0,\cdots, 7$, and the quadratic form is nonnegative $B(\bar{\phi}, \bar{\phi})\geq 0$, we know $\partial_t\int_{\mathbb{R}^4}|\bar{\phi}_t|^2dy \leq 0$. Also, there holds that
\begin{equation*}
\int_{\mathbb{R}^4}|\bar{\phi}_t|^2dy = -\frac{1}{2}\partial_t B(\bar{\phi}, \bar{\phi}).
\end{equation*}
Hence we have
\begin{equation*}
\int_{-\infty}^0dt\int_{\mathbb{R}^4}|\bar{\phi}_t|^2dy < +\infty.
\end{equation*}
Therefore $\bar{\phi}_t = 0$, $\bar{\phi}$ is independent of $t$ and it holds that $\mathcal L[\bar{\phi}] = 0$. Since $\bar\phi$ satisfies the estimate $|\bar{\phi}(y,t)|\leq |y|^{\alpha+1}$, using the non-degeneracy result of Atiyah-Hitchin-Singer \cite{AtiyahHitchinSinger1978} (see also \cite{Brendle2003} and \cite{IsobeMariniJMP}), from Lemma \ref{nondegeneracy}, $\bar{\phi}$ is a linear combination of the 1-forms $Z^j$ defined in Section 2, $j = 0,\cdots, 7$. But since we also have  $\int_{\mathbb{R}^4}\sum_{i=1}^4\bar{\phi}_i(y, t)\cdot Z^j_i(y)dy = 0$, $j = 0,\cdots, 7$, we get $\bar{\phi} = 0$. This is a contradiction, therefore (\ref{e5:37}) holds.

From the assumption (\ref{e6666}), there exists a sequence $\{y_n\}$ satisfying $|y_n|\to\infty$ and
$$
(\tau_2^n)^\nu|y_n|^{1+\alpha}|\phi^n(y_n, \tau_2^n)|\geq\frac{1}{2}.
$$
Now we define
$$
\tilde \phi^n(z,\tau):=(\tau_2^n)^\nu|y_n|^{1+\alpha}\phi^n(|y_n|z, |y_n|^{-2}\tau+\tau_2^n)
$$
then we have
$$
\partial_\tau\tilde\phi^n =\Delta_z\tilde\phi^n+b_n\cdot\nabla\tilde\phi^n+c_n\tilde\phi^n+\tilde h^n(z,\tau)
$$
with
$$
\tilde h^n(z,\tau)=(\tau_2^n)^\nu|y_n|^{3+\alpha}h^n(|y_n|z, |y_n|^{-2}\tau+\tau_2^n).
$$
By assumption (\ref{e6666}), there holds
$$
|\tilde h^n(z,\tau)|\leq o(1)|z|^{-3-\alpha}((\tau_2^n)^{-1}|y_n|^{-2}\tau+1)^{-\nu}
$$
Hence $\tilde h^n(z,\tau)\to 0$ uniformly on compact subsets of $\mathbb{R}^4 \setminus\{0\}\times (-\infty ,0]$. Similarly, we have $b_n\to 0$ and $c_n\to 0$ uniformly on compact subsets of $\mathbb{R}^4 \setminus\{0\}\times (-\infty ,0]$. Furthermore, there holds $|\tilde\phi^n(\frac{y_n}{|y_n|},0)|\ge \frac{1}{2}$ and
$$
|\tilde\phi^n(z,\tau)|\leq|z|^{-1-\alpha}((\tau_2^n)^{-1}|y_n|^{-2}\tau+1)^{-\nu} .
$$
Therefore we have $\tilde\phi^n\to\tilde\phi\neq 0$ uniformly on compact subsets of $\mathbb{R}^4\setminus\{0\}\times (-\infty, 0]$ and $\tilde\phi$ satisfies the following
$$
\tilde\phi_\tau=(d^*d + dd^*)\tilde\phi\text{ in }\mathbb{R}^4\setminus\{0\}\times (-\infty ,0],
$$
$$
|\tilde\phi(z, \tau)|\leq |z|^{-1-\alpha}\text{ in }\mathbb{R}^4\setminus\{0\}\times (-\infty ,0].
$$
Let us set $\tilde{\phi} = \sum_{i=1}^4\tilde{\phi}_i(z, t)dx_i$, then for $i = 1, 2, 3, 4$, $\tilde\phi_i(z, t)$ satisfies
$$
(\tilde{\phi}_i)_\tau=\Delta \tilde{\phi}_i \text{ in }\mathbb{R}^4\setminus\{0\}\times (-\infty ,0]
$$
and
$$
|\tilde{\phi}_i(z, \tau)|\leq |z-\hat e|^{-\alpha-1}\text{ in }\mathbb{R}^4\setminus\{0\}\times (-\infty ,0].
$$
Then from Theorem 5.1 of \cite{LiDongSire}, we know that $\tilde{\phi}_i(z, t) = 0$ hence $\tilde\phi = 0$, which a contradiction. This completes the proof.
\end{proof}

{\bf Proof of Proposition \ref{proposition5.1}}
Let $\phi(y,\tau)$ be the unique solution of the following Cauchy problem
\begin{equation*}
\left\{
\begin{aligned}
&\partial_\tau\phi = \mathcal L[\phi] + h(y,\tau),\,\,y\in \mathbb{R}^4,\,\,\tau\geq \tau_0,\\
&\phi(y,\tau_0) = 0,\,\,y\in \mathbb{R}^4,\\
&\lim_{|y|\to +\infty}|\phi(y, \tau)|\to 0 \text{ for all }\tau\geq \tau_0.
\end{aligned}
\right.
\end{equation*}
For any $\tau_1 > \tau_0$, by Lemma \ref{l5:3}, there holds
\begin{equation*}
|\phi(y,\tau)|\lesssim\tau^{-\nu}(1+|y|)^{-a-1}\|h\|_{3+a, \tau_1}\text{ for all }\tau\in (\tau_0, \tau_1), \,\,y\in \mathbb{R}^4.
\end{equation*}

From the assumption that $\|h\|_{3+a,\nu} < +\infty$, we have $\|h\|_{3+a, \tau_1}\leq \|h\|_{3+a,\nu}$ for any $\tau_1 > \tau_0$. Therefore
\begin{equation*}
|\phi(y,\tau)|\lesssim\tau^{-\nu}(1+|y|)^{-a-1}\|h\|_{3+a,\nu}\text{ for all }\tau\in (\tau_0, \tau_1),\,\, y\in \mathbb{R}^4.
\end{equation*}
Since $\tau_1$ is arbitrary, we have
\begin{equation*}
|\phi(y,\tau)|\lesssim\tau^{-\nu}(1+|y|)^{-a-1}\|h\|_{3+a,\nu}\text{ for all }\tau\in (\tau_0, +\infty),\,\, y\in \mathbb{R}^4.
\end{equation*}
The estimate for the gradient of $\phi$ follows from standard parabolic regularity results and a scaling argument, thus we get (\ref{e5:100}).\qed

\section{Proof of Proposition \ref{l4:lemma4.1}}
 
Set $\psi(x, t) = \bar\psi\left(\frac{x-\xi(t)}{\mu_0(t)}, \tau(t)\right)$ with $\frac{dt}{d\tau} = \mu_{0}^{2}(t)$. Then $\bar\psi$ satisfies the following heat equation
\begin{equation}\label{heatequationselfsimilar}
\begin{cases}
\partial_\tau\bar\psi =
\Delta_y\bar\psi + \bar V(y, \tau)\bar\psi + \bar f(y, \tau)\quad&\text{in}\,\,\mathbb{R}^4\times (\tau_0, \infty),\\
\lim_{|y|\to +\infty}\bar\psi(y, \tau) = 0\quad&\text{for all}\,\, \tau\in (\tau_0,\infty).
\end{cases}
\end{equation}
Here $\bar V(y, \tau) \sim (1-\eta_R(y))\frac{1}{1+|y|^2}$ and $|\bar f(y, \tau) |\leq M\frac{\tau^{-\nu}}{1+| y|^{2+\alpha}}$ with $\tau = \frac{1}{2\kappa_0}e^{2\kappa_0 t}$ and $\nu = \frac{\beta}{2}$. Using the heat kernel estimates (see e.g. Corollary 2 of \cite{CoulhonZhang} and Theorem 16 of \cite{DaviesSimon}), we know the solution of (\ref{heatequationselfsimilar}) can be estimate as 
\begin{equation}\label{solutiongivenbyheatkernel}
|\bar\psi(y, \tau)| \lesssim \int_{\tau_0}^{\tau}\int_{\mathbb R^4}\frac{1}{(\tau-s)^{2-\kappa}}e^{-\frac{|y-u|^2}{\tau-s}}\,\bar f(u, s)duds,
\end{equation}
$\kappa > 0$ is a small constant. 
Then we have
\begin{equation*}
\begin{split}
|\bar \psi(y, \tau)|
\lesssim&\,\|f\|_{*,\beta,2+\alpha }\int_{\tau_0}^{\tau}\int_{\mathbb R^4}\frac{1}{(\tau-s)^{2-\kappa}}e^{-\frac{|y-u|^2}{\tau-s}}\frac{s^{-\nu}} { 1+ \left|{u}\right|^{2+\alpha} }duds
\\
=&\,\|f\|_{*,\beta,2+\alpha }\int_{\tau/2}^{\tau}\int_{\mathbb R^4}\frac{1}{(\tau-s)^{2-\kappa}}e^{-\frac{|y-u|^2}{\tau-s}}\frac{s^{-\nu}} { 1+ \left|{u}\right|^{2+\alpha} }duds
\\
&\,+\|f\|_{*,\beta,2+\alpha }\int_{\tau_0}^{\tau/2}\int_{\mathbb R^4}\frac{1}{(\tau-s)^{2-\kappa}}e^{-\frac{|y-u|^2}{\tau-s}}\frac{s^{-\nu}} { 1+ \left|{u}\right|^{2+\alpha} }duds
\\
:=&\,I_1+I_2.
\end{split}
\end{equation*}
To estimate the term $I_1$, we use the variable transformation $p=\frac{|y-u|}{\sqrt{\tau-s}}$, $\frac{ds}{2(\tau-s)}=\frac{1}{p}dp$, then we have
\begin{equation*}
\begin{split}
I_1&=\|f\|_{*,\beta,2+\alpha }\int_{\tau/2}^{\tau}\int_{\mathbb R^4}\frac{1}{(\tau-s)^{2-\kappa}}e^{-\frac{|y-u|^2}{\tau-s}}\frac{s^{-\nu}} { 1+ \left|{u}\right|^{2+\alpha} }duds\\
&\leq\|f\|_{*,\beta,2+\alpha }\tau^{-\nu}\int_{\tau/2}^{\tau}\int_{\mathbb R^4}\frac{1}{(\tau-s)^{2-\kappa}}e^{-\frac{|y-u|^2}{\tau-s}}\frac{1} { 1+ \left|{u}\right|^{2+\alpha} }duds
\\
&=2\|f\|_{*,\beta,2+\alpha }\tau^{-\nu}\int_{\mathbb R^4}\frac {1} { |y-u|^{2-2\kappa}(1+|{u}|^{2+\alpha})}du\int_{\frac{|y-u|}{\sqrt{\tau/2}}}^{+\infty}p^{1-2\kappa}e^{-p^2}   \,dp
\\
& \lesssim \|f\|_{*,\beta,2+\alpha }\tau^{-\nu}\int_{\mathbb R^4}\frac {1} { |y-u|^{2-2\kappa}(1+|{u}|^{2+\alpha})}du.
\end{split}
\end{equation*}
Now we estimate
\begin{equation*}
\begin{split}
\int_{\R^4}\frac{1}{|y-u|^{2-2\kappa}}\frac {1} {1+ |u|^{2+\alpha} }du.
\end{split}
\end{equation*}
Fix point $y$ (we assume $|y|\geq 2$) and separate $\mathbb R^4$ as $$B(y,\frac{|y|}{2})\cup B(0,\frac{|y|}{2})\cup \left(\mathbb{R}^4\setminus (B(y,\frac{|y|}{2})\cup B(0,\frac{|y|}{2}))\right):=B_1\cup B_2 \cup B_3.$$ Then we have
\begin{equation*}
\begin{split}
&\int_{\R^4}\frac{1}{|y-u|^{2-2\kappa}}\frac {1} {1+ |u|^{2+\alpha} }du
=\int_{B_1}\frac{1}{|y-u|^{2-2\kappa}}\frac {1} {1+ |u|^{2+\alpha} }du\\
&\quad +\int_{B_2}\frac{1}{|y-u|^{2-2\kappa}}\frac {1} {1+ |u|^{2+\alpha} }du +\int_{B_3}\frac{1}{|y-u|^{2-2\kappa}}\frac {1} {1+ |u|^{2+\alpha} }du\\
&:=K_1+K_2+K_3.
\end{split}
\end{equation*}
If $u\in B_1$, then $\frac{|y|}{2}\leq |u|\leq \frac{3|y|}{2}$, therefore
\begin{equation*}
\begin{split}
K_1&=\int_{B_1}\frac{1}{|y-u|^{2-2\kappa}}\frac {1} {1+ |u|^{2+\alpha} }du
\leq \int_{B_1}\frac{1}{|y-u|^{2-2\kappa}}\frac {1} {1+ (\frac{|y|}{2})^{2+\alpha} }du\\
&\lesssim \frac{1}{|y|^{2+\alpha}}\int_{B_1}\frac{1}{|y-u|^{2}}du
=\frac{1}{|y|^{2+\alpha}}\int_{0}^{\frac{|y|}{2}}r\,dr\lesssim\frac{1}{|y|^{\alpha}}.
\end{split}
\end{equation*}
If $u\in B_2$, then $\frac{|y|}{2}\leq |y-u|\leq \frac{3|y|}{2}$,
\begin{equation*}
\begin{split}
K_2&=\int_{B_2}\frac{1}{|y-u|^{2-2\kappa}}\frac {1} {1+ |u|^{2+\alpha} }du
\lesssim\ \int_{B_2}\frac{1}{|y|^{2-2\kappa}}\frac {1} {1+ |u|^{2+\alpha} }du
\\
&\lesssim\
\frac{1}{|y|^{2-2\kappa}}\int_{0}^{\frac{|y|}{2}}\frac {1} {1 + r^{2+\alpha} }r^3\,dr
\lesssim\frac{1}{|y|^{2-2\kappa}}\left(1+|y|^{2-\alpha}\right)\lesssim\frac{1}{|y|^{\alpha-2\kappa}}=\frac{1}{|y|^{\alpha'}}.
\end{split}
\end{equation*}
Here $\alpha' = \alpha-2\kappa\in (0, \alpha)$ since $\kappa > 0$ can be chosen small enough.
We estimate the term $K_3$ as follows,
\begin{equation*}
\begin{split}
K_{3}&=\int_{B_{3}}\frac{1}{|y-u|^{2-2\kappa}}\frac {1} {1+ |u|^{2+\alpha} }du\lesssim\ \int_{B_{3}}\frac{1}{|u|^{2-2\kappa}}\frac {1} {1+ |u|^{2+\alpha} }du
\\
&
\lesssim\ \int_{\frac{|y|}{2}}^{\infty}\frac{r^{1+2\kappa}}{r^{2+\alpha}}dr\lesssim\frac{1}{|y|^{\alpha-2\kappa}}=\frac{1}{|y|^{\alpha'}}.
\end{split}
\end{equation*}
From the above estimates, we get
\begin{equation*}
\begin{split}
\int_{\R^4}\frac{1}{|y-u|^{2}}\frac {1} {1+ |u|^{2+\alpha} }du\lesssim \frac{1}{|y|^{\alpha'}}
\end{split}
\end{equation*}
for $\alpha'\in (0, \alpha)$.
Hence we have
\begin{equation*}
\begin{split}
I_1&\lesssim\|f\|_{*,\beta,2+\alpha }\tau^{-\nu}\int_{\mathbb R^4}\frac {1} { |y-u|^{2-2\kappa}(1+|{u}|^{2+\alpha})}du
\\
&\lesssim \|f\|_{*,\beta,2+\alpha }\tau^{-\nu}\frac{1}{|y|^{\alpha'}}.
\end{split}
\end{equation*}
Now we estimate the second term $$I_2=\|f\|_{*,\beta,2+\alpha }\int_{\tau_0}^{\frac{\tau}{2}}\int_{\mathbb R^4}\frac{1}{(\tau-s)^{2-\kappa}}e^{-\frac{|y-u|^2}{\tau-s}}\frac{s^{-\nu}} { 1+ \left|{u}\right|^{2+\alpha} }duds.$$ We have
\begin{equation*}
\begin{split}
I_2&=\|f\|_{*,\beta,2+\alpha }\int_{\tau_0}^{\frac{\tau}{2}}\int_{\mathbb R^4}\frac{1}{(\tau-s)^{2-\kappa}}e^{-\frac{|y-u|^2}{\tau-s}}\frac{s^{-\nu}} { 1+ \left|{u}\right|^{2+\alpha} }duds
\\
&\lesssim \tau^{\kappa-1-\frac{\alpha}{2}}\|f\|_{*,\beta,2+\alpha }\int_{\tau_0}^{\frac{\tau}{2}}\int_{\mathbb R^4}e^{-|\frac{y}{\tau-s}-u|^2}\frac{s^{-\nu}} {\left|{u}\right|^{2+\alpha} }duds
\\
&=\tau^{\kappa-1-\frac{\alpha}{2}}\|f\|_{*,\beta,2+\alpha }\int_{\tau_0}^{\frac{\tau}{2}}\int_{|u|> 2|\frac{y}{\tau-s}|}e^{-|\frac{y}{\tau-s}-u|^2}\frac{s^{-\nu}} {\left|{u}\right|^{2+\alpha} }duds
\\
&\quad + \tau^{\kappa-1-\frac{\alpha}{2}}\|f\|_{*,\beta,2+\alpha }\int_{\tau_0}^{\frac{\tau}{2}}\int_{|u|\leq 2|\frac{y}{\tau-s}|}e^{-|\frac{y}{\tau-s}-u|^2}\frac{s^{-\nu}} {\left|{u}\right|^{2+\alpha} }duds\\
& := I_{21}+I_{22}
\end{split}
\end{equation*}
and
\begin{equation*}
\begin{split}
I_{21} & =\tau^{\kappa-1-\frac{\alpha}{2}}\|f\|_{*,\beta,2+\alpha }\int_{\tau_0}^{\frac{\tau}{2}}\int_{|u|> 2|\frac{y}{\tau-s}|}e^{-|\frac{y}{\tau-s}-u|^2}\frac{s^{-\nu}} {\left|{u}\right|^{2+\alpha} }duds
\\
&\lesssim \tau^{\kappa-1-\frac{\alpha}{2}}\|f\|_{*,\beta,2+\alpha }\int_{\tau_0}^{\frac{\tau}{2}}\int_{|u|> 2|\frac{y}{\tau-s}|}e^{-|u|^2/4}\frac{s^{-\nu}} {\left|{u}\right|^{2+\alpha} }duds\\
&\lesssim
\tau^{-\nu-\frac{\alpha}{2}+\kappa}\|f\|_{*,\beta,2+\alpha }\\
&\lesssim
\left\{
\begin{aligned}
&\|f\|_{*,\beta,2+\alpha}\frac{\tau^{-\nu}}{1+|y|^{\alpha'}}, \quad |y|\leq \sqrt{\tau}\\
&\|f\|_{*,\beta,2+\alpha}\tau^{-\nu-\frac{\alpha}{2}}, \quad |y| > \sqrt{\tau}\\
\end{aligned}
\right.,
\end{split}
\end{equation*}
\begin{equation*}
\begin{split}
I_{22} & \lesssim  \tau^{\kappa-1-\frac{\alpha}{2}}\|f\|_{*,\beta,2+\alpha }\int_{\tau_0}^{\frac{\tau}{2}}\int_{|u|\leq 2|\frac{y}{\tau-s}|}e^{-|\frac{y}{\tau-s}-u|^2}\frac{s^{-\nu}} {\left|{u}\right|^{2+\alpha} }duds\\
&\lesssim\tau^{-(2-\kappa)+1-\nu}\|f\|_{*,\beta,2+\alpha }\\
&\lesssim
\left\{
\begin{aligned}
&\|f\|_{*,\beta,2+\alpha}\frac{\tau^{-\nu}}{1+|y|^{\alpha'}}, \quad |y|\leq \sqrt{\tau}\\
&\|f\|_{*,\beta,2+\alpha}\tau^{-\nu-\frac{\alpha}{2}}, \quad |y| > \sqrt{\tau}\\
\end{aligned}
\right..
\end{split}
\end{equation*}
Here we have used the facts that $\kappa > 0$ is small and $\alpha\in (1, 2)$.

\section{Existence of bubble tower solutions}
In the $SO(4)$-equivariant case, which means that we assume that $$A(x, t) = Im(\frac{x}{2r^2}\psi(r, t)d\bar{x}),$$ 
with $r = |x|$, the Yang-Mills heat flow takes the following form,
\begin{equation}\label{bubbletower}
\frac{\partial }{\partial t}\psi = \psi_{rr} + \frac{1}{r}\psi_r -\frac{2}{r^2}(\psi-1)(\psi-2)\psi.
\end{equation}
This equation has a one-parameter family of finite energy stationary solutions, namely
\begin{equation}
\frac{2r^2}{r^2+\lambda^2},\quad \lambda\in\mathbb{R}^+.
\end{equation}
In this section, we construct a bubble tower solution for (\ref{bubbletower}).

\subsection{Construction of approximate solutions.}
If we use the transformation $\bar\psi = r^{-2}\psi$, then (\ref{bubbletower}) becomes the following heat equation
\begin{equation}\label{bubbletower111}
\frac{\partial }{\partial t}\bar\psi = \bar\psi_{rr} + \frac{5}{r}\bar\psi_r +(6-2r^2\bar\psi)\bar\psi^2
\end{equation}
with steady solution $\bar\psi_0(r) = \frac{2}{r^2+\lambda^2}$.
We write the first approximation of $\bar{\psi}(r, t)$ as
$$
\bar U(r, t) = U_*(r, t)+\eta_0\left(\frac{r}{\sqrt{\mu_2(t)}}\right)\frac{1}{\mu_2(t)^2}U\left(\frac{r}{\mu_2(t)}\right)
$$
with
$$
U(r) = \frac{2}{r^2+1}
$$
and $U_*(r, t)$ is the one bubble solution of (\ref{bubbletower}) constructed in Theorem \ref{t:main}. Then $U_*(r, t)$ has the following form
$$ U_*(r, t) = \mu_1^{-2}U\left(\frac{r}{\mu_1(t)}\right)+\varphi(r, t),$$
with $\mu_1 = e^{-c_1t}+o(e^{-c_1t}): = \mu_{01}+o(e^{-c_1t})$ for a positive number $c_1 > 0$ and $\varphi(r, t)$ is a perturbation term.

In the following, we also set $$U_2(r, t):=\eta_0\left(\frac{r}{\sqrt{\mu_2(t)}}\right)\frac{1}{\mu_2(t)^2}U\left(\frac{r}{\mu_2(t)}\right).$$
Let us define $\bar\mu_{02}(t):= \sqrt{\mu_{01}\mu_{02}}$ (we also define $\bar\mu_{01} = t^{\delta}$ with $\delta > 0$ be a small constant.) and the following cut off functions
$$
\chi(r, t) = \eta_0\left(\frac{2r}{\bar\mu_{02}(t)}\right).
$$
Note that $\chi(r) = 1$ for $r\leq \frac{1}{2}\bar\mu_{02}(t)$ and $\chi(r, t) = 0$ for $r\geq \bar\mu_{02}(t)$. We are looking for a correction term
of the form
$$
\varphi_0(r, t) = \mu_2^{-2}\phi_0\left(\frac{r}{\mu_2(t)}, t\right)\chi(r, t) := \varphi_{02}(r, t)\chi(r, t),
$$
which will be suitably determined later. Let us write
$$
S(\bar U + \varphi_0) = S(\bar U) + L_{\bar U}[\varphi_0] + N_{\bar U}[\varphi_0]
$$
where
$$
L_{\bar U}[\varphi_0] = -\frac{\partial }{\partial t}\varphi_0 + (\varphi_0)_{rr} + \frac{5}{r}(\varphi_0)_r + g'(\bar U)\varphi_0,
$$
$$
N_{\bar U}[\varphi_0] = g(\bar U + \varphi_0) - g'(\bar U)\varphi_0 - g(\bar U)
$$
and
\begin{equation*}
\begin{aligned}
S(\bar U) & = -\partial_tU_2 + g(\bar U) - g(U_2) - g(U_*),
\end{aligned}
\end{equation*}
\begin{equation*}
g(\bar U):=(6-2r^2\bar U)\bar U^2.
\end{equation*}
Observe that
\begin{equation*}
\begin{aligned}
L_{\bar U}[\varphi_0] & = (\varphi_{02})_{rr}\chi + \frac{5}{r}(\varphi_{02})_r\chi + g'(\bar U)\varphi_{02}\chi\\
& \quad  + 2(\varphi_{02})_{r}\chi_r + \frac{5}{r}\varphi_{02}\chi_r-\frac{\partial }{\partial t}(\varphi_{02}\chi).
\end{aligned}
\end{equation*}
Therefore, we have
\begin{equation*}
\begin{aligned}
&S(\bar U + \varphi_0)\\
& = \left((\varphi_{02})_{rr} + \frac{5}{r}(\varphi_{02})_r + g'(U_2)\varphi_{02} - \partial_tU_2 + g'(U_2)U_*(0)\right)\chi\\
&\quad  +\bar E_{11} + \left(g'(\bar U)-g'(U_2)\right)\varphi_{02}\chi + N_{\bar U}[\varphi_0] + 2(\varphi_{02})_{r}\chi_r + \frac{5}{r}\varphi_{02}\chi_r-\frac{\partial }{\partial t}(\varphi_{02}\chi).
\end{aligned}
\end{equation*}
Here the term $\bar E_{11}$ is defined by
\begin{equation*}
\begin{aligned}
\bar E_{11}& = -(1-\chi)\partial_tU_2  + g(\bar U) - g(U_2) - g(U_*) - g'(U_2)U_*(0)\chi.
\end{aligned}
\end{equation*}
Let us define
\begin{equation*}
\begin{aligned}
& E[\phi_0, \mu_2]: = (\varphi_{02})_{rr} + \frac{5}{r}(\varphi_{02})_r + g'(U_2)\varphi_{02} - \partial_tU_2 + g'(U_2)U_*(0)\\
& = \frac{1}{\mu^4_2}[(\phi_0)_{yy} + \frac{5}{y}(\phi_0)_y + (12W_0(y)-6y^2W^2_0(y))\phi_0+\mu_2\dot\mu_2 \eta_{0}(\sqrt{\mu_2}y)Z(y)\\
&\quad\quad\quad\quad-\frac{1}{2}(\eta_0)_ryW(y)(\mu_2)^{3/2}\dot\mu_2+(12W_0(y)-6y^2W^2_0(y))\left(\frac{\mu_2}{\mu_1}\right)^2U(0)]|_{y = \frac{r}{\mu_2(t)}}.
\end{aligned}
\end{equation*}
Here $W_0(y) = \eta_{0}(\sqrt{\mu_2}y)W(y)$ and $Z(y) : = \frac{4}{(1+|y|^2)^2}$.
To solve this equation, we need the following orthogonality condition
\begin{equation*}
\int_{0}^{+\infty}\left(\mu_2\dot\mu_2 Z(y)+(12W(y)-6y^2W^2(y))\left(\frac{\mu_2}{\mu_1}\right)^2U(0)\right)Z(y)y^5dy = 0,
\end{equation*}
which is equivalent to
\begin{equation*}
\mu_2\dot\mu_2 + c_* \left(\frac{\mu_2}{\mu_1}\right)^2 = 0,\quad c_* = \frac{\int_{0}^{+\infty}(12W(y)-6y^2W^2(y))U(0)Z(y)y^5dy}{\int_{0}^{+\infty}Z(y)Z(y)y^5dy} > 0.
\end{equation*}
The solution to this equation is
$$
\mu_{02}(t) = e^{-\frac{c_*e^{2c_1 t}}{2c_1}}.
$$
Finally, we set $\mu_2 = \mu_{02} + \mu_{12}$. Here the parameters $\mu_{12}$ is to be determined for some small but fixed $\sigma > 0$,
$$
\mu_{02}|\dot\mu_{12}|\leq \lambda_{02}^2(t)\mu_{02}^{\sigma}(t), \quad  \lambda_{02}:= \frac{\mu_{02}}{\mu_{01}},
$$
which implies $\lim_{t\to+\infty}\frac{\mu_{12}}{\mu_{02}} = 0$. It is convenient to write $\lambda_2(t): = \frac{\mu_{2}(t)}{\mu_{1}(t)} = \lambda_{02}(t) + \lambda_{12}(t)$. Thus we obtain the following result.
\begin{lemma}
\begin{equation*}
\begin{aligned}
&\frac{1}{\mu^4_2}\left[(\mu_2\dot\mu_2-\mu_{02}\dot\mu_{02}) Z(y)+(12W(y)-6y^2W^2(y))\left(\left(\lambda_2\right)^2-\left(\lambda_{02}\right)^2\right)U(0)\right]\\
& =  \frac{1}{\mu^4_2}D_2 +  \frac{1}{\mu^4_2}\Theta_2
\end{aligned}
\end{equation*}
with
\begin{equation*}
\begin{aligned}
D_2: = (\dot \mu_{02}\mu_{12} + \mu_{02}\dot \mu_{12})Z(y)+2(12W(y)-6y^2W^2(y))\left(\lambda_{02}\right)^2\frac{\mu_{12}}{\mu_{02}}U(0)
\end{aligned}
\end{equation*}
\begin{equation*}
\begin{aligned}
\Theta_2: & = \partial_t(\mu_{12}^2)Z(y)+(12W(y)-6y^2W^2(y))\left(\lambda_{02}\right)^2\left(\frac{\mu_{12}}{\mu_{02}}\right)^2U(0).
\end{aligned}
\end{equation*}
\end{lemma}  
We also define
$$
u^* = \bar U + \varphi_0,
$$
with $\varphi_0(r, t) = \mu_{02}^{-2}\phi_0\left(\frac{r}{\mu_{02}(t)}, t\right)\chi(r, t)$ and $\phi_0(y, t)$ satisfies the following equation
\begin{equation*}
\begin{aligned}
(\phi_0)_{yy} + \frac{5}{y}(\phi_0)_y & + (12W(y)-6y^2W^2(y))\phi_0+\mu_{02}\dot\mu_{02} Z(y)\\
& +(12W(y)-6y^2W^2(y))\left(\frac{\mu_{02}}{\mu_{01}}\right)^2U(0) = 0.
\end{aligned}
\end{equation*}

\subsection{The inner-outer gluing system.}
Let $R$ be a $t$-independent, slowly growing function compared with $\mu_{01}$, say $R = e^{\varrho t_0}, \quad t\geq t_0$
where $\varrho > 0$ is a sufficiently small constant. Consider the cut-off functions $\eta_i$ defined by
\begin{equation*}
\begin{aligned}
\eta_i(r, t) = \eta_0\left(\frac{r}{2R\mu_{0i}(t)}\right),
\end{aligned}
\end{equation*}
\begin{equation*}
\begin{aligned}
\zeta_1(r, t) = \eta_0\left(\frac{r}{R\mu_{01}(t)}\right)-\eta_0\left(\frac{Rr}{\mu_{01}(t)}\right),
\end{aligned}
\end{equation*}
\begin{equation*}
\begin{aligned}
\zeta_2(r, t) = \eta_0\left(\frac{r}{R\mu_{02}(t)}\right),
\end{aligned}
\end{equation*}
We set
$$
\varphi = \varphi_1\eta_1 + \varphi_2\eta_2 + \Psi,
$$
where
$$
\varphi_j(r, t) = \frac{1}{\mu_j^2}\phi_j\left(\frac{r}{\mu_j}, t\right).
$$
With these notations, we have
\begin{equation*}
\begin{aligned}
& S[u^*+\varphi]\\
& = \eta_1\mu_1^{-4}\left[-\mu_1^2\partial_t\phi_1 + (\phi_1)_{rr}+\frac{5}{r}(\phi_1)_r + (12W(y_1)-6y_1^2W^2(y_1))\phi_1\right. \\
&\quad\quad\quad\quad\quad\quad\quad\quad\quad\quad\quad\quad\quad\quad\quad\quad\quad\quad\quad\quad\left.+\zeta_1(12W(y_1)-6y_1^2W^2(y_1))\mu_1^2\Psi\right]\\
&\quad + \eta_2\mu_2^{-4}\left[-\mu_2^2\partial_t\phi_2 + (\phi_2)_{rr}+\frac{5}{r}(\phi_2)_r + (12W(y_2)-6y_2^2W^2(y_2))\phi_2\right.\\
&\quad\quad\quad\quad\quad\quad\quad\quad\quad\quad\quad\quad\quad\quad\quad \quad\quad\quad\left. +\zeta_2(12W(y_2)-6y_2^2W^2(y_2))\mu_2^2\Psi + D_2\right]\\
&\quad -\partial_t\Psi + (\Psi)_{rr}+\frac{5}{r}(\Psi)_r + V\Psi + B[\vec{\phi}] + N(\phi, \Psi, \vec{\mu}) + E^{out}.
\end{aligned}
\end{equation*}
Here
\begin{equation*}
\begin{aligned}
B[\vec{\phi}]& =\sum_{j=1}^2(-\partial_t\eta+\partial_{rr}\eta)\varphi_j + 2(\eta)_r(\varphi_j)_r + \frac{5}{r}(\eta_j)_r\varphi_j - \dot\mu_j\frac{\partial }{\partial\mu_j}\varphi_j\eta_j \\
&\quad\quad\quad  + \left((12u^*-6r^2(u^*)^2)-(12U_2-6r^2U_2^2)\right)\varphi_2\eta_2\\
&\quad\quad\quad  + \left((12u^*-6r^2(u^*)^2)-(12U_*-6r^2U_*^2)\right)\varphi_1\eta_1,
\end{aligned}
\end{equation*}
\begin{equation*}
N(\phi, \Psi, \mu) = N_{u^*}\left(\varphi_1\eta_1 + \varphi_2\eta_2 + \Psi\right),
\end{equation*}
\begin{equation*}
V = (12u^*-6r^2(u^*)^2) - \zeta_2(12U_2-6r^2U^2_2)- \zeta_1(12U_*-6r^2U_*^2),
\end{equation*}
\begin{equation*}
E^{out} = S[u^*] - \eta_2\mu_2^{-4}D_2.
\end{equation*}
Then $S[u^*+\varphi] = 0$ if the following system is satisfied,
\begin{equation}\label{innerproblem1}
-\mu_1^2\partial_t\phi_1 + (\phi_1)_{rr}+\frac{5}{r}(\phi_1)_r + (12W-6y_1^2W^2)\phi_1+\zeta_1(12W-6y_1^2W^2)\mu_1^2\Psi + c(t)Z(y_1) = 0,
\end{equation}
\begin{equation}\label{innerproblem2}
-\mu_2^2\partial_t\phi_2 + (\phi_2)_{rr}+\frac{5}{r}(\phi_2)_r + (12W-6y_2^2W^2)\phi_2+\zeta_2(12W-6y_2^2W^2)\mu_2^2\Psi + D_2 = 0,
\end{equation}
and
\begin{equation}\label{outerproblem}
-\partial_t\Psi + (\Psi)_{rr}+\frac{5}{r}\Psi_r + V\Psi + B[\vec{\phi}] + N(\phi, \Psi, \mu) + E^{out} -\eta_1\mu_1^{-4}c(t)Z(y_1) = 0.
\end{equation}
Observe that $U_*(r, t)$ is the one bubble solution of (\ref{bubbletower}) constructed in the one-bubble case,
Problem (\ref{innerproblem1}) is the linearized problem around $U_*(r, t)$. We adjust the small parameter $c(t)$ to obtain orthogonality.
\subsection{Linear theory for the inner and outer problem.}
For the inner problem, we consider the following problem
\begin{equation}\label{innerproblem1model}
-\partial_\tau\phi + (\phi)_{rr}+\frac{5}{r}(\phi)_r + (12W-6r^2W^2)\phi+f(r, \tau) = 0, \quad (r, \tau)\in (0, 2R)\times [\tau_0, +\infty).
\end{equation}
Then we have
\begin{lemma}\label{bubbletowerinnerlineartheory}
Suppose $\alpha\in (0, 1)$, $\nu > 0$, $\|f\|_{2+\alpha, \nu} < +\infty$ and
\begin{equation}\label{bubbletower:20200301}
\int_{0}^{2R}f(r,\tau)\cdot Z(r)r^5dr = 0\quad\text{for all}\quad\tau\in [\tau_0,+\infty).
\end{equation}
Then, there exists a solution of (\ref{innerproblem1model}) satisfying
\begin{equation}\label{bubbletowerestimate100}
\left|\phi(r,\tau)\right|\lesssim \|f\|_{2+\alpha, \nu}\tau^{-\nu}\frac{R^{6-\alpha}}{1+|r|^6}.
\end{equation}
\end{lemma}

For the outer problem, we consider the following problem
\begin{equation}\label{outerproblem1model}
-\partial_t\Psi + (\Psi)_{rr}+\frac{5}{r}\Psi_r + g(r, t) = 0,\quad  (r, t)\in (0, +\infty) \times [t_0,\infty).
\end{equation}
Assume that for $\alpha$, $\sigma>0$, the function $g(r, t)$ satisfies the following estimate
\begin{equation}\label{normfortheouterproblem}
\begin{aligned}
|g(r, t)|\leq M\frac{\mu_2^{-2}(t)\mu_2^{\sigma}(t)}{1+|y|^{2+\alpha}}, \quad  y =\frac{r}{\mu_2(t)}
\end{aligned}
\end{equation}
and we denote the least number $M>0$ such that (\ref{normfortheouterproblem}) holds as $\|g\|_{*,\sigma,2+\alpha}$.
Then we have the following lemma.
\begin{lemma}\label{bubbletowerouterlineartheory}
Suppose $\|g\|_{*,\sigma,2+\alpha}<+\infty$ for some $\alpha > 0$ and $\sigma > 0$. Let $\Psi(r, t) $ be the solution of (\ref{outerproblem1model}) given be the Duhamel formula
$$
\Psi (r, t) = \int_{t_0}^{t}\int_{\mathbb R^6}\frac{1}{4\pi (t-s)^3}e^{-\frac{|x-u|^2}{t-s}}g(|u|, s)duds.
$$
Then, for all $(r, t)\in (0, +\infty) \times [t_0,\infty)$, there holds
\begin{equation}\label{potoo0}
|\Psi(r, t)|\lesssim \|g\|_{*,\sigma,2+\alpha}\frac{\mu_2^{\sigma}(t)}{1+|y|^{\alpha}},
\end{equation}
for $y =\frac{r}{\mu_2(t)}$.
\end{lemma}
The proof of Lemma \ref{bubbletowerinnerlineartheory} is a minor modification of Proposition 7.1 in \cite{delPinoMussoJEMS} and Proposition 7.1 in \cite{DDW2020} (the radially symmetric case).  Lemma \ref{bubbletowerouterlineartheory} can be proved as Proposition \ref{l4:lemma4.1}.

\subsection{Estimates for outer problem.}
We have the following estimates for the outer problem.
Suppose
\begin{equation*}
\langle y_2\rangle|\nabla_{y_2}\phi_2(y_2, t)|+|\phi_2(y_2, t)|\lesssim e^{-\varepsilon t_0}\mu_2^{\sigma}(t)\left(\frac{\mu_2(t)}{\mu_1(t)}\right)^{2}\langle y_2\rangle^{-a},
\end{equation*}
then we have \\
{\bf Step 1. Estimates for $B[\vec{\phi}]$}:

$\bullet$ We have \begin{equation*}
\begin{aligned}
\left|\dot\mu_2\frac{\partial }{\partial\mu_2}\varphi_2\eta_2 \right| &\lesssim \left|\dot\mu_2\mu_2^{-3}\left(2\phi_2(y_2, t)+y_2\cdot \nabla_{y_2}\phi_2(y_2, t)\right)\eta_2\right|\\
&\lesssim e^{-\varepsilon t_0}\left(\frac{1}{\mu_1(t)\mu_2(t)}\right)^{2}\mu_2^{\sigma}(t)\left(\frac{\mu_2(t)}{\mu_1(t)}\right)^{2}\langle y_2\rangle^{-a}\\
&\lesssim e^{-\varepsilon t_0} R^2\mu_2(t)^{2}\mu_1(t)^{-2}\left(\frac{1}{\mu_1(t)}\right)^{2}\left(\frac{1}{\mu_2(t)}\right)^{2}\mu_2^{\sigma}(t)\langle y_2\rangle^{-2-a}.
\end{aligned}
\end{equation*}

$\bullet$ We have
\begin{equation*}
\begin{aligned}
\left|\varphi_2\right| &\lesssim e^{-\varepsilon t_0}\mu_2^{\sigma}(t)\left(\frac{1}{\mu_1(t)}\right)^{2}R^{-a}\quad \text{ for } 2R\leq |y_2| \leq 4R.
\end{aligned}
\end{equation*}

\begin{equation*}
\begin{aligned}
\left|\partial_{rr}\eta_2\varphi_2\right| &\lesssim e^{-\varepsilon t_0}(R\mu_2)^{-2}\mu_2^{\sigma}(t)\left(\frac{1}{\mu_1(t)}\right)^{2}R^{-a}\\
&\lesssim e^{-\varepsilon t_0}\left(\frac{1}{\mu_1(t)}\right)^{2}\left(\frac{1}{\mu_2(t)}\right)^{2}\mu_2^{\sigma}(t)R^{-2-a}\text{ for } 2R\leq |y_2| \leq 4R.
\end{aligned}
\end{equation*}

\begin{equation*}
\begin{aligned}
\left|(\eta_2)_r(\varphi_2)_r\right|  &\lesssim e^{-\varepsilon t_0}(R\mu_2)^{-1}\mu_2^{\sigma}(t)\left(\frac{1}{\mu_1(t)}\right)^{2}\left(\frac{1}{\mu_2(t)}\right)R^{-a-1}\\
&\lesssim e^{-\varepsilon t_0}\left(\frac{1}{\mu_1(t)}\right)^{2}\left(\frac{1}{\mu_2(t)}\right)^{2}\mu_2^{\sigma}(t)R^{-2-a}\text{ for } 2R\leq |y_2| \leq 4R.
\end{aligned}
\end{equation*}

\begin{equation*}
\begin{aligned}
\left|\partial_t \eta_2\varphi_2\right| &\lesssim e^{-\varepsilon t_0}R^{-1}\mu_2^{-2}|\dot \mu_2|\mu_2^{\sigma}(t)\left(\frac{1}{\mu_1(t)}\right)^{2}R^{-a}\\
& \lesssim e^{-\varepsilon t_0}R|\dot \mu_2|\mu_2^{-2}\mu_2^{\sigma}(t)\left(\frac{1}{\mu_1(t)}\right)^{2}R^{-2-a}\\
& \lesssim e^{-\varepsilon t_0}R\mu_2(t)\left(\frac{1}{\mu_1(t)^2}\right)^{2}\mu_2^{\sigma}(t)\left(\frac{1}{\mu_2(t)}\right)^{2}R^{-2-a}\text{ for } 2R\leq |y_2| \leq 4R.
\end{aligned}
\end{equation*}

$\bullet$ Since
\begin{equation*}
\begin{aligned}
\left|\left((12u^*-6r^2(u^*)^2)-(12U_2-6r^2U_2^2)\right)\right| &\lesssim \mu_2^{-2}\langle y_2\rangle^{-4},
\end{aligned}
\end{equation*}
We have
\begin{equation*}
\begin{aligned}
&\left|\left((12u^*-6r^2(u^*)^2)-(12U_2-6r^2U_2^2)\right)\varphi_2\eta_2\right|\\
&\quad \lesssim e^{-\varepsilon t_0}\mu_2^{-2}\langle y_2\rangle^{-4}\mu_2^{\sigma}(t)\left(\frac{1}{\mu_1(t)}\right)^{2}R^{-a}\\
&\quad \lesssim e^{-\varepsilon t_0}\left(\frac{1}{\mu_1(t)}\right)^{2}\mu_2^{\sigma}(t)\left(\frac{1}{\mu_2(t)}\right)^{2}R^{-2-a}.
\end{aligned}
\end{equation*}

From the linear theory for the outer problem, we have
\begin{equation*}
|\Psi|\lesssim e^{-\varepsilon t_0}\frac{\mu_2^{\sigma}(t)}{\mu_1(t)^2}\langle y_2\rangle^{-a}.
\end{equation*}
And from the linear theory for the inner problem, we have
\begin{equation*}
|\phi_1|\lesssim e^{-\varepsilon t_0}\mu_2^{\sigma}(t)\langle y_2\rangle^{-a}.
\end{equation*}
Then we have the following estimates:

$\bullet$\begin{equation*}
\begin{aligned}
\left|\partial_{rr}\eta_1\varphi_1\right| &\lesssim e^{-\varepsilon t_0}(R\mu_1)^{-2}\frac{\mu_2^{\sigma}(t)}{\mu_1(t)^2}\langle y_2\rangle^{-a}\\
&\lesssim e^{-\varepsilon t_0}\left(\frac{1}{\mu_2(t)}\right)^{2}\frac{\mu_2^{\sigma}(t)}{\mu_1(t)^2}\langle y_2\rangle^{-2-a}\text{ for } 2R\leq |y_1| \leq 4R.
\end{aligned}
\end{equation*}

$\bullet$\begin{equation*}
\begin{aligned}
\left|(\eta_1)_r(\varphi_1)_r\right|&\lesssim e^{-\varepsilon t_0}\mu_2^{-1}(R\mu_1)^{-1}\frac{\mu_2^{\sigma}(t)}{\mu_1(t)^2}\langle y_2\rangle^{-a-1}\\
&\lesssim e^{-\varepsilon t_0}\left(\frac{1}{\mu_2(t)}\right)^{2}\frac{\mu_2^{\sigma}(t)}{\mu_1(t)^2}\langle y_2\rangle^{-2-a}\text{ for } 2R\leq |y_1| \leq 4R.
\end{aligned}
\end{equation*}

$\bullet$\begin{equation*}
\begin{aligned}
\left|\partial_t \eta_1\varphi_1\right| &\lesssim e^{-\varepsilon t_0}R^{-1}\mu_1^{-2}|\dot \mu_1|\frac{\mu_2^{\sigma}(t)}{\mu_1(t)^2}\langle y_2\rangle^{-a}\\
& \lesssim e^{-\varepsilon t_0}R^{-1}\mu_1^{-1}\frac{\mu_2^{\sigma}(t)}{\mu_1(t)^2}\langle y_2\rangle^{-a}\\
&\lesssim e^{-\varepsilon t_0}R\mu_1(t)\left(\frac{1}{\mu_2(t)}\right)^{2}\frac{\mu_2^{\sigma}(t)}{\mu_1(t)^2}\langle y_2\rangle^{-2-a}\text{ for }\quad 2R\leq |y_1| \leq 4R
\end{aligned}
\end{equation*}

$\bullet$\begin{equation*}
\begin{aligned}
\left|\dot\mu_1\frac{\partial }{\partial\mu_1}\varphi_1\eta_1 \right| &\lesssim \left|\dot\mu_1\mu_1^{-3}\left(2\phi_1+y_1\cdot \nabla_{y_1}\phi_1\right)\eta_1\right|\\
&\lesssim e^{-\varepsilon t_0}\left|\dot\mu_1\right|\mu_1^{-1} \frac{\mu_2^{\sigma}(t)}{\mu_1(t)^2}\langle y_2\rangle^{-a}\\
&\lesssim e^{-\varepsilon t_0}\left(R\mu_1(t)\right)^2\left(\frac{1}{\mu_2(t)}\right)^{2}\frac{\mu_2^{\sigma}(t)}{\mu_1(t)^2}\langle y_2\rangle^{-2-a}.
\end{aligned}
\end{equation*}

$\bullet$ Since
\begin{equation*}
\begin{aligned}
&\left|\left((12u^*-6r^2(u^*)^2)-(12U_*-6r^2U_*^2)\right)\right|\lesssim \mu_1^{-2}\langle y_1\rangle^{-4}\left(\lambda_2^{-2}\langle y_2 \rangle^{-4} + \lambda_2\right),
\end{aligned}
\end{equation*}
We have
\begin{equation*}
\begin{aligned}
&\left|\left((12u^*-6r^2(u^*)^2)-(12U_*-6r^2U_*^2)\right)\varphi_1\eta_1\right|\\
&\quad \lesssim e^{-\varepsilon t_0}\mu_1^{-4}\langle y_1\rangle^{-4}\left(\lambda_2^{-2}\langle y_2 \rangle^{-4} + \lambda_2\right)\mu_2^{\sigma}(t)\langle y_2\rangle^{-a}\\
&\quad \lesssim e^{-\varepsilon t_0}\left(\frac{1}{\mu_2(t)}\right)^{2}\frac{\mu_2^{\sigma}(t)}{\mu_1(t)^2}\langle y_2\rangle^{-2-a}\\
&\quad \quad + e^{-\varepsilon t_0}\mu_1(t)^{-1}\mu_2(t)R^2\left(\frac{1}{\mu_2(t)}\right)^{2}\frac{\mu_2^{\sigma}(t)}{\mu_1(t)^2}\langle y_2\rangle^{-2-a}\\
&\quad \lesssim e^{-\varepsilon t_0}\left(1+\mu_1(t)^{-1}\mu_2(t)R^2\right)\left(\frac{1}{\mu_2(t)}\right)^{2}\frac{\mu_2^{\sigma}(t)}{\mu_1(t)^2}\langle y_2\rangle^{-2-a}.
\end{aligned}
\end{equation*}
We continue the estimates of the outer problem as follows.

{\bf Step 2. Estimates for $\eta_1\mu_1^{-4}c(t)Z(y_1)$}: We have
\begin{equation*}
\begin{aligned}
&\left|\eta_1\mu_1^{-4}c(t)Z(y_1)\right|\\
&\lesssim e^{-\varepsilon t_0}(\mu_1(t))^{-4}\mu_2^{\sigma}(t)\left(\frac{\mu_2}{\mu_1}\right)^a\frac{|x|^{2+a}}{\mu_2^{2+a}}\langle y_2\rangle^{-2-a}Z(y_1)\\
&\lesssim e^{-\varepsilon t_0} R^{2+a}Z(y_1)\frac{\mu_2^{\sigma}(t)}{\mu_1(t)^2}\left(\frac{1}{\mu_2(t)}\right)^{2}\langle y_2\rangle^{-2-a}\\
&\lesssim e^{-\varepsilon t_0}\frac{\mu_2^{\sigma}(t)}{\mu_1(t)^2}\left(\frac{1}{\mu_2(t)}\right)^{2}\langle y_2\rangle^{-2-a}.
\end{aligned}
\end{equation*}

{\bf Step 3. Estimates for $E^{out}$}: We decompose the error term $E^{out}$ as follows.

$\bullet$\begin{equation*}
\begin{aligned}
&\left|(12U_2-6r^2U_2^2)\left(U_*-U_*(0)\right)\chi\right|\\
&\quad \lesssim \mu_2^{-2}\langle y_2\rangle^{-4}\mu_1^{-2}\lambda_2\chi\lesssim \lambda_2\mu_2^{-\sigma}(t)\left(\frac{1}{\mu_2(t)}\right)^{2}\frac{\mu_2^{\sigma}(t)}{\mu_1(t)^2}\langle y_2\rangle^{-2-a}.
\end{aligned}
\end{equation*}

$\bullet$\begin{equation*}
\begin{aligned}
&\left|\left((6-2r^2\bar U)\bar U^2-(6-2r^2U_*)U_*^2-(6-2r^2U_2)U^2_2-(12U_2-6r^2U_2^2)U_*\right)\chi\right|\\
&\quad \lesssim \left|(6-2r^2U_*)U_*^2\chi\right|\lesssim \mu_1^{-4}\chi\\
&\quad \lesssim \mu_2^{-\sigma}(t)\left(\mu_2(t)\right)^{\frac{2-a}{2}}\left(\mu_1(t)\right)^{\frac{a-2}{2}}\left(\frac{1}{\mu_2(t)}\right)^{2}\frac{\mu_2^{\sigma}(t)}{\mu_1(t)^2}\langle y_2\rangle^{-2-a}.
\end{aligned}
\end{equation*}

$\bullet$ Since \begin{equation*}
\begin{aligned}
&\left|\partial_tU_2\right| \lesssim \left|\dot \mu_2\mu_2^{-3}Z(y_2)\right|\lesssim \left(\frac{1}{\mu_1(t)\mu_2(t)}\right)^{2}\langle y_2\rangle^{-4},
\end{aligned}
\end{equation*}
we have in the region $\{|x|\geq 2\sqrt{\mu_1\mu_2}\}$,
\begin{equation*}
\begin{aligned}
&\left|(1-\chi)\partial_tU_2\right|\lesssim \mu_2^{-\sigma}(t)\left(\frac{\mu_2(t)}{\mu_1(t)}\right)^{\frac{2-a}{2}}\left(\frac{1}{\mu_1(t)\mu_2(t)}\right)^{2}\mu_2^{\sigma}(t)\langle y_2\rangle^{-2-a}.
\end{aligned}
\end{equation*}

$\bullet$
\begin{equation*}
\begin{aligned}
&\left|\left((6-2r^2\bar U)\bar U^2-(6-2r^2U_*)U_*^2-(6-2r^2U_2)U^2_2\right)(1-\chi)\right|\\
&\quad \lesssim \left|U_2U_*(1-\chi)\right|\lesssim \mu_2^{-\sigma}(t)\frac{\mu_2(t)^{1-\frac{a}{2}}}{\mu_1(t)^{1-\frac{a}{2}}}\left(\frac{1}{\mu_1(t)\mu_2(t)}\right)^{2}\mu_2^{\sigma}(t)\langle y_2\rangle^{-2-a}.
\end{aligned}
\end{equation*}

$\bullet$
\begin{equation*}
\begin{aligned}
&\left|N_{\bar U}[\varphi_0]\right|\lesssim |\varphi_0|^2 \lesssim \mu_2^{-4}\left(\frac{\mu_2}{\mu_1}\right)^4\langle y_2\rangle^{-4}\chi^2\\
&\quad \lesssim \mu_2^{-\sigma}(t)\mu_2(t)^2\mu_1(t)^{-2}\left(\frac{1}{\mu_1(t)\mu_2(t)}\right)^{2}\mu_2^{\sigma}(t)\langle y_2\rangle^{-2-a}.
\end{aligned}
\end{equation*}

$\bullet$
\begin{equation*}
\begin{aligned}
&\left|\partial_r\varphi_0\partial_r\chi_r\right|\lesssim \sqrt{\mu_1\mu_2}^{-1}\left(\frac{\mu_2}{\mu_1}\right)^2\mu_2^{-3}\langle y_2\rangle^{-3}\\
&\quad \lesssim \mu_2^{-\sigma}(t)\left(\frac{\mu_2(t)}{\mu_1(t)}\right)^{\frac{1}{2}}\left(\frac{1}{\mu_1(t)\mu_2(t)}\right)^{2}\mu_2^{\sigma}(t)\langle y_2\rangle^{-2-a}.
\end{aligned}
\end{equation*}

$\bullet$
\begin{equation*}
\begin{aligned}
&\left|\varphi_0\cdot\partial_{rr}\chi\right|\lesssim \sqrt{\mu_1\mu_2}^{-2}\left(\frac{\mu_2}{\mu_1}\right)^2\mu_2^{-2}\langle y_2\rangle^{-2}\\
&\quad \lesssim \mu_2^{-\sigma}(t)\left(\frac{\mu_2(t)}{\mu_1(t)}\right)^{\frac{2-a}{2}}\left(\frac{1}{\mu_1(t)\mu_2(t)}\right)^{2}\mu_2^{\sigma}(t)\langle y_2\rangle^{-2-a}.
\end{aligned}
\end{equation*}

$\bullet$
\begin{equation*}
\begin{aligned}
&\left|\partial_t\varphi_0\chi\right|+\left|\varphi_0\partial_t\chi\right|\lesssim \left(\frac{1}{\mu_1(t)}\right)^{2}\frac{1}{\sqrt{\mu_1(t)\mu_2(t)}}\left(\frac{\mu_2}{\mu_1}\right)^2\mu_2^{-2}\langle y_2\rangle^{-2}\\
&\quad \lesssim \mu_2^{-\sigma}(t)\left(\frac{\mu_2(t)}{\mu_1(t)}\right)^{\frac{4-a}{2}}\frac{1}{\sqrt{\mu_1(t)\mu_2(t)}}\left(\frac{1}{\mu_1(t)\mu_2(t)}\right)^{2}\mu_2^{\sigma}(t)\langle y_2\rangle^{-2-a}.
\end{aligned}
\end{equation*}

{\bf Step 4. Estimates for the term $V$}:

$\bullet$ Let us recall that
\begin{equation*}
V = (12u^*-6r^2(u^*)^2) - \zeta_2(12U_2-6r^2U^2_2)- \zeta_1(12U_*-6r^2U_*^2),
\end{equation*}

In the region $\{r\geq R\mu_{01}\}$, we have
\begin{equation*}
\begin{aligned}
\left|V\Psi\right|\lesssim e^{-\varepsilon t_0}\frac{1}{r^4}\left|\Psi\right| &\lesssim \frac{\mu_1^2}{r^4}\frac{\mu_2^{\sigma}(t)}{\mu_1(t)^2}\langle y_2\rangle^{-a}\\
& \lesssim e^{-\varepsilon t_0}\frac{\mu_1(t)^2}{r^2}\left(\frac{1}{\mu_1(t)\mu_2(t)}\right)^{2}\mu_2^{\sigma}(t)\langle y_2\rangle^{-2-a}\\
& \lesssim e^{-\varepsilon t_0}\frac{1}{R^2}\left(\frac{1}{\mu_1(t)\mu_2(t)}\right)^{2}\mu_2^{\sigma}(t)\langle y_2\rangle^{-2-a}.
\end{aligned}
\end{equation*}

In the region $\{\sqrt{\mu_{01}\mu_{02}}\leq r\leq 2R^{-1}\mu_{01}\}$, we have
\begin{equation*}
\begin{aligned}
\left|V\Psi\right|\lesssim e^{-\varepsilon t_0}\frac{1}{\mu_{1}^2}\left|\Psi\right| &\lesssim  \frac{r^2}{\mu_{1}^2}\left(\frac{1}{\mu_1(t)\mu_2(t)}\right)^{2}\mu_2^{\sigma}(t)\langle y_2\rangle^{-2-a}\\
& \lesssim e^{-\varepsilon t_0}\frac{1}{R^2}\left(\frac{1}{\mu_1(t)\mu_2(t)}\right)^{2}\mu_2^{\sigma}(t)\langle y_2\rangle^{-2-a}.
\end{aligned}
\end{equation*}

In the region $\{R\mu_{02}\leq r\leq \sqrt{\mu_{01}\mu_{02}}\}$, we have
\begin{equation*}
\begin{aligned}
\left|V\Psi\right|\lesssim e^{-\varepsilon t_0}\frac{\mu_{2}^2}{r^4}\left|\Psi\right| &\lesssim  \frac{\mu_{2}^2}{r^2}\left(\frac{1}{\mu_1(t)\mu_2(t)}\right)^{2}\mu_2^{\sigma}(t)\langle y_2\rangle^{-2-a}\\
& \lesssim e^{-\varepsilon t_0}\frac{1}{R^2}\left(\frac{1}{\mu_1(t)\mu_2(t)}\right)^{2}\mu_2^{\sigma}(t)\langle y_2\rangle^{-2-a}
\end{aligned}
\end{equation*}

We also have
\begin{equation*}
\begin{aligned}
&\left|\zeta_1(u_*^{p-1}-(12U_*-6r^2U_*^2))\Psi\right|\\
& \lesssim e^{-\varepsilon t_0}\frac{\mu_{2}^2}{r^4}\left|\Psi\right| \lesssim  e^{-\varepsilon t_0}\frac{\mu_{2}^2}{r^2}\left(\frac{1}{\mu_1(t)\mu_2(t)}\right)^{2}\mu_2^{\sigma}(t)\langle y_2\rangle^{-2-a}\\
& \lesssim e^{-\varepsilon t_0}\frac{1}{R^2}\left(\frac{1}{\mu_1(t)\mu_2(t)}\right)^{2}\mu_2^{\sigma}(t)\langle y_2\rangle^{-2-a}
\end{aligned}
\end{equation*}
and
\begin{equation*}
\begin{aligned}
&\left|\zeta_2(u_*^{p-1}-(12U_2-6r^2U^2_2))\Psi\right|\\
& \lesssim \frac{1}{\mu_{1}^2}\left|\Psi\right| \lesssim  e^{-\varepsilon t_0}\frac{\mu_{2}^2R^2}{\mu_{1}^2}\left(\frac{1}{\mu_1(t)\mu_2(t)}\right)^{2}\mu_2^{\sigma}(t)\langle y_2\rangle^{-2-a}\\
& \lesssim e^{-\varepsilon t_0}\frac{1}{R^2}\left(\frac{1}{\mu_1(t)\mu_2(t)}\right)^{2}\mu_2^{\sigma}(t)\langle y_2\rangle^{-2-a}.
\end{aligned}
\end{equation*}

{\bf Step 5. Estimates for the term $N(\phi, \Psi, \mu)$}:

$\bullet$ Observe that
\begin{equation*}
\begin{aligned}
&\left|N_{u^*}\left(\varphi_1\eta_1 + \varphi_2\eta_2 + \Psi\right)\right|\lesssim \left|\varphi_1\right|^2\eta_1 + \left|\varphi_2\right|^2\eta_2 + \left|\Psi\right|^2,
\end{aligned}
\end{equation*}
we have
\begin{equation*}
\begin{aligned}
\left|\varphi_1\right|^2\eta_1 & \lesssim \mu_{1}^{-4}\mu_2^{2\sigma}(t)\langle y_2\rangle^{-2a}(e^{-\varepsilon t_0})^2\\
& \lesssim \mu_{1}^{-a}\mu_{2}^{a}R^{2-a}\left(\frac{1}{\mu_1(t)\mu_2(t)}\right)^{2}\mu_2^{2\sigma}(t)\langle y_2\rangle^{-2-a}(e^{-\varepsilon t_0})^2\\
& \lesssim e^{-\varepsilon t_0}\frac{1}{R^2}\left(\frac{1}{\mu_1(t)\mu_2(t)}\right)^{2}\mu_2^{\sigma}(t)\langle y_2\rangle^{-2-a},
\end{aligned}
\end{equation*}
\begin{equation*}
\begin{aligned}
\left|\varphi_2\right|^2\eta_2 & \lesssim \mu_{2}^{-4}\left(\mu_2^{\sigma}(t)\left(\frac{\mu_2(t)}{\mu_1(t)}\right)^{2}\langle y_2\rangle^{-a}e^{-\varepsilon t_0}\right)^2\eta_2\\
& \lesssim \mu_{1}^{-2}\mu_{2}^{2}R^{2-a}\left(\frac{1}{\mu_1(t)\mu_2(t)}\right)^{2}\mu_2^{2\sigma}(t)\langle y_2\rangle^{-2-a}(e^{-\varepsilon t_0})^2\\
& \lesssim e^{-\varepsilon t_0}\frac{1}{R^2}\left(\frac{1}{\mu_1(t)\mu_2(t)}\right)^{2}\mu_2^{\sigma}(t)\langle y_2\rangle^{-2-a}
\end{aligned}
\end{equation*}
and
\begin{equation*}
\begin{aligned}
\left|\Psi\right|^2&\lesssim \left(\frac{\mu_2^{\sigma}(t)}{\mu_1(t)^2}\langle y_2\rangle^{-a}e^{-\varepsilon t_0}\right)^2\\
& \lesssim \mu_{1}^{-2}\mu_{2}^{a}\min\{r^{2-a}, 1\}\mu_2^{\sigma}(t)\left(\frac{1}{\mu_1(t)\mu_2(t)}\right)^{2}\mu_2^{\sigma}(t)\langle y_2\rangle^{-2-a}(e^{-\varepsilon t_0})^2\\
& \lesssim e^{-\varepsilon t_0}\frac{1}{R^2}\left(\frac{1}{\mu_1(t)\mu_2(t)}\right)^{2}\mu_2^{\sigma}(t)\langle y_2\rangle^{-2-a}.
\end{aligned}
\end{equation*}
Combining all the estimates above, we obtain the following result.
\begin{lemma}
For the term
$$
F[\phi_1, \phi_2, \Psi, \mu_{12}](x,t):=V\Psi + B[\vec{\phi}] + N(\phi, \Psi, \mu) + E^{out} -\eta_1\mu_1^{-4}c(t)Z(y_1),
$$
we have 
$$
\left|F[\phi_1, \phi_2, \Psi, \mu_{12}]\right|\lesssim e^{-\varepsilon t_0}\frac{\mu_2^{\sigma}(t)}{\mu_1(t)^2\mu_2(t)^2}\langle y_2\rangle^{-a-2}.
$$
\end{lemma}

\subsection{Proof of Theorem \ref{t:main-bubble-tower}.}
We apply Lemma \ref{bubbletowerinnerlineartheory} to the inner problem (\ref{innerproblem2}). To this aim, we need to choose the parameter $\mu_{12}$ such that the orthogonality condition
\begin{equation}\label{e:orthogonalitytower}
\int_{0}^{2R}\left(\zeta_2(12W-6y_2^2W^2)\mu_2^2\Psi + D_2\right)\cdot Z(y_2)y^5_2d y_2 = 0
\end{equation}
holds. Recall that
\begin{equation*}
\begin{aligned}
D_2: = (\dot \mu_{02}\mu_{12} + \mu_{02}\dot \mu_{12})Z(y)+2(12W(y)-6y^2W^2(y))\left(\lambda_{02}\right)^2\frac{\mu_{12}}{\mu_{02}}U(0),
\end{aligned}
\end{equation*}
$$
\mu_{02}(t) = e^{-\frac{c_*e^{2c_1 t}}{2c_1}},\quad \frac{\dot\mu_{02}}{\mu_{02}} = -c_*e^{2c_1 t} = -c_*\frac{\lambda_{02}^2}{\mu_{02}^2},
$$
$$
\mu_{02}\dot\mu_{02} + c_* \left(\frac{\mu_{02}}{\mu_{01}}\right)^2 = 0,
$$
and the fact that
$$
\frac{\int_{0}^{2R}(12W(y)-6y^2W^2(y))U(0)Z(y)y^5dy}{\int_{0}^{2R}Z(y)Z(y)y^5dy} = c_* + O(R^{-2}),
$$
we know, the orthogonality condition (\ref{e:orthogonalitytower}) can be reduce to the following ODE of $\mu_{12}$,
\begin{equation}\label{aoeiue}
\begin{aligned}
\dot \mu_{12} + c_*e^{2c_1 t} \mu_{12} & = -\frac{\mu_2^2}{\mu_{02}}\frac{\int_{0}^{2R}\left(\zeta_2(12W-6y_2^2W^2)\Psi\right)\cdot Z(y_2)y^5_2dy_2}{\int_{0}^{2R}Z(y_2)^2y^5_2dy_2} + \dot\mu_{02}\frac{\mu_{12}}{\mu_{02}}O(R^{-2}).
\end{aligned}
\end{equation}

Let $h$ be a function of $t$ satisfying condition $\|h\|_{\sigma}^\sharp:=\sup_{t\geq t_0}|\mu_{02}^{-1-\sigma}(t)\mu_{01}^2(t) h(t)| \lesssim e^{-\varepsilon t_0}$. Observe that the solution of
\begin{equation}\label{e5:14}
\dot \mu_{12} + c_*e^{2c_1 t} \mu_{12} = h(t)
\end{equation}
can be expressed by the formula,
\begin{equation}\label{e5:15}
\mu_{12}(t) = e^{-\frac{c_*e^{2c_1 t}}{2c_1}}\left[d + \int_{t_0}^te^{\frac{c_*e^{2c_1 \tau}}{2c_1}} h(\tau)d\tau\right],
\end{equation}
with the constant $d$ be chosen as $d = -\int_{t_0}^{+\infty}e^{\frac{c_*e^{2c_1 \tau}}{2c_1}} h(\tau)d\tau$. Therefore, we have
\begin{equation}\label{mu12}
\begin{aligned}
&\|\mu_{12}\|_{\sigma}^\sharp  = \|e^{(1+\sigma)\frac{c_*e^{2c_1 t}}{2c_1}}\mu_{01}^2(t)\mu_{12}(t)\|_{L^\infty(t_0,\infty)}\\
&\lesssim \sup_{t\geq t_0}\left|e^{\sigma\frac{c_*e^{2c_1 t}}{2c_1}}\mu_{01}^2(t)\int_{t}^{+\infty}e^{\frac{c_*e^{2c_1 \tau}}{2c_1}}| h(\tau)|d\tau\right| \\
&\lesssim e^{-\varepsilon t_0}\sup_{t\geq t_0}\left|e^{\sigma\frac{c_*e^{2c_1 t}}{2c_1}}\int_{t}^{+\infty}e^{-\sigma\frac{c_*e^{2c_1 \tau}}{2c_1}}e^{2c_1\tau}d\tau\right|\lesssim e^{-\varepsilon t_0}.
\end{aligned}
\end{equation}
This gives us a bounded linear operator $\mathcal{T}_1: h\to \mu_{12}$ by assigning the solution $\mu_{12}$ of (\ref{e5:14}) given by (\ref{e5:15}) to any given $h$ satisfying $\|h\|_{\sigma}^\sharp < +\infty$. Thus $\mu_{12}$ is a solution of (\ref{aoeiue}) if it is fixed point of the problem
\begin{equation*}
\mu_{12} = \mathcal T_1\left(G[\phi_1, \phi_2, \Psi, \mu_{12}](t)\right)
\end{equation*}
with
$$
G[\phi_1, \phi_2, \Psi, \mu_{12}](t):=-\frac{\mu_2^2}{\mu_{02}}\frac{\int_{0}^{2R}\left(\zeta_2(12W-6y_2^2W^2)\Psi\right)\cdot Z(y_2)y^5_2dy_2}{\int_{0}^{2R}Z(y_2)^2y^5_2dy_2} + \dot\mu_{02}\frac{\mu_{12}}{\mu_{02}}O(R^{-2}).
$$

Once the orthogonality condition is satisfied, from Lemma \ref{bubbletowerinnerlineartheory} we know that there is a bounded linear operator $\mathcal{T}_2$ mapping from a function $h(y,\tau)$ satisfying $\|h\|_{2 + \alpha, \sigma} < +\infty$ to a solution $\phi$ of (\ref{innerproblem1model}) satisfying the estimate (\ref{bubbletowerestimate100}). Therefore the solution of (\ref{innerproblem1}) is a fixed point of the problem
\begin{equation*}
\phi_1 = \mathcal{T}_2\left(H_1[\phi_1, \phi_2, \Psi, \mu_{12}](y,t(\tau))\right),
\end{equation*}
while the solution of (\ref{innerproblem2}) is a fixed point of the problem
\begin{equation*}
\phi_2 = \mathcal{T}_2\left(H_2[\phi_1, \phi_2, \Psi, \mu_{12}](y,t(\tau))\right).
\end{equation*}
Here, $H_1[\phi_1, \phi_2, \Psi, \mu_{12}]$ and $H_2[\phi_1, \phi_2, \Psi, \mu_{12}]$ are defined as
\begin{equation*}
H_1[\phi_1, \phi_2, \Psi, \mu_{12}](y,t(\tau)):=\zeta_1(12W-6y_1^2W^2)\mu_1^2\Psi + c(t)Z(y_1) = 0,
\end{equation*}
\begin{equation*}
H_2[\phi_1, \phi_2, \Psi, \mu_{12}](y,t(\tau)):=\zeta_2(12W-6y_2^2W^2)\mu_2^2\Psi + D_2 = 0,
\end{equation*}
Similarly, the solution $\Psi$ of (\ref{outerproblem}) is a fixed point of the problem
\begin{equation*}\label{e4:34}
\Psi=\mathcal{T}_3(F[\phi_1, \phi_2, \Psi, \mu_{12}](x,t)),
\end{equation*}
$$
F[\phi_1, \phi_2, \Psi, \mu_{12}](x,t):=V\Psi + B[\vec{\phi}] + N(\phi, \Psi, \mu) + E^{out} -\eta_1\mu_1^{-4}c(t)Z(y_1).
$$
Here $\mathcal T_3$ is the solution operator given by Lemma \ref{bubbletowerouterlineartheory}.

From the above argument, we know that $(\phi, \Psi, \mu_{12})$ is a fixed point of the following problem,
\begin{equation}\label{inner_outer_gluing_system000}
\left\{
\begin{aligned}
&\mu_{12} = \mathcal T_1\left(G[\phi_1, \phi_2, \Psi, \mu_{12}](t)\right),\\
&\phi_1 = \mathcal{T}_2\left(H_1[\phi_1, \phi_2, \Psi, \mu_{12}](y,t(\tau))\right),\\
&\phi_2 = \mathcal{T}_2\left(H_2[\phi_1, \phi_2, \Psi, \mu_{12}](y,t(\tau))\right),\\
&\Psi=\mathcal{T}_3(F[\phi_1, \phi_2, \Psi, \mu_{12}](x,t)).
\end{aligned}
\right.
\end{equation}
Let us apply the Schauder fixed-point theorem in the following set
\begin{equation*}\label{convexsetforschauder}
\begin{aligned}
\mathcal{B} &= \Bigg\{(\phi_1, \phi_2, \Psi, \mu_{12}): \|\phi_1\|_{a, \sigma}^{(1)} + \|\phi_2\|_{a, \sigma}^{(2)} + \|\Psi\|_{*, a, \sigma} + \|\mu_{12}\|_{\sigma}^\sharp \leq c e^{-\varepsilon t_0}\Bigg\}
\end{aligned}
\end{equation*}
for a fixed positive constant $c$ large enough. Here $ \|\Psi\|_{*, a, \sigma}$ is the least $M > 0$ satisfying the following
\begin{equation*}
|\Psi|\leq M\frac{\mu_2^{\sigma}(t)}{\mu_1(t)^2}\langle y_2\rangle^{-a}.
\end{equation*}
Similarly, $\|\phi_1\|_{a, \sigma}^{(1)}$, $\|\phi_2\|_{a, \sigma}^{(2)}$ are the least $M > 0$ satisfying
\begin{equation*}
\langle y_1\rangle|\nabla_{y_1}\phi_1(y_1, t)|+|\phi_1(y_1, t)|\lesssim M\mu_2^{\sigma}(t)\langle y_2\rangle^{-a},
\end{equation*}
and
\begin{equation*}
\langle y_2\rangle|\nabla_{y_2}\phi_2(y_2, t)|+|\phi_2(y_2, t)|\lesssim M\mu_2^{\sigma}(t)\left(\frac{\mu_2(t)}{\mu_1(t)}\right)^{2}\langle y_2\rangle^{-a},
\end{equation*}
respectively. On the set $\mathcal{B}$, from the estimates in Section 6.4, we have
\begin{equation*}\label{e6:1}
\begin{aligned}
\left|G[\phi_1, \phi_2, \Psi, \mu_{12}]\right|\lesssim e^{-\varepsilon t_0}\frac{\mu_{02}^{1+\sigma}(t)}{\mu_{01}^2(t)},
\end{aligned}
\end{equation*}
\begin{equation*}\label{e6:1}
\begin{aligned}
\left|H_1[\phi_1, \phi_2, \Psi, \mu_{12}]\right|\lesssim e^{-\varepsilon t_0}\mu_2^{\sigma+a}(t)\mu_1^{-a}(t)\langle y_1\rangle^{-a-2},
\end{aligned}
\end{equation*}
\begin{equation*}\label{e6:1}
\begin{aligned}
\left|H_2[\phi_1, \phi_2, \Psi, \mu_{12}]\right|\lesssim e^{-\varepsilon t_0}\mu_2^{\sigma}(t)\left(\frac{\mu_2(t)}{\mu_1(t)}\right)^{2}\langle y_2\rangle^{-a-2},
\end{aligned}
\end{equation*}
\begin{equation*}\label{e6:1}
\begin{aligned}
\left|F[\phi_1, \phi_2, \Psi, \mu_{12}]\right|\lesssim e^{-\varepsilon t_0}\frac{\mu_2^{\sigma}(t)}{\mu_1(t)^2\mu_2(t)^2}\langle y_2\rangle^{-a-2}.
\end{aligned}
\end{equation*}
From Lemma Lemma \ref{bubbletowerinnerlineartheory}, Lemma \ref{bubbletowerouterlineartheory} and estimate (\ref{mu12}), the operator $\mathcal T$ defined in (\ref{inner_outer_gluing_system000}) maps $\mathcal{B}$ into $\mathcal{B}$ (the choice of the constant $c$ in the definition of $\mathcal{B}$ can be achieved by the same reason as in the proof of Theorem \ref{t:main}). Since $\phi_1$, $\phi_2$, $\Psi$, $\mu_{12}$ decay uniformly as $t\to +\infty$, standard parabolic estimate ensures that $\mathcal T$ defined in (\ref{inner_outer_gluing_system000}) is a compact operator. From the Schauder fixed-point theorem, there is a fixed point of $\mathcal T$ in the set $\mathcal{B}$. This provides a solution of (\ref{bubbletower111}) with form
$$
\bar \psi(r, t) = U_*(r, t)+\left(\frac{2}{r^2}-\frac{1}{\mu_2(t)^2}U\left(\frac{r}{\mu_2(t)}\right)\right)+\varphi_2(r, t).
$$
Observe that $\frac{2}{r^2}-\bar\psi(r, t)$ is also a solution of (\ref{bubbletower111}) and this solution has form
$$
\frac{2}{r^2}-\bar\psi(r, t) = -U_*(r, t)+\frac{1}{\mu_2(t)^2}U\left(\frac{r}{\mu_2(t)}\right)-\varphi_2(r, t),
$$
which is our desired solution.

\section*{Appendix: Some useful computations}
Here we give the detailed expressions for the terms $\tilde{\mathcal L}_i[B_{1, q}]$, $\tilde{\mathcal L}_i[\Phi_0]$, $\tilde{\mathcal L}_i[\Phi_1^{(1)}]$, $\tilde{\mathcal L}_i[\Phi_1^{(2)}]$, $\tilde{\mathcal L}_i[\Phi_1^{(3)}]$ in Section 2.4. First, we compute $\tilde{\mathcal L}[\varphi] = \sum_{i = 1}^4\tilde{\mathcal L}_i[\varphi]dx_i$ for the differential 1-form $\varphi$ we introduced before. By direct computations, for $B_{1, q}(x) = Im\left(\frac{x-q}{1+|x-q|^2}d\bar{x}\right)$, we have
\begin{equation*}
\begin{aligned}
\tilde{\mathcal L}_1[B_{1, q}] = \frac{24\mu^2}{(|x-q|^2+1)(|x-q|^2+\mu^2)^2}\left((x_2-q_2)i+(x_3-q_3)j+(x_4-q_4)k\right),
\end{aligned}
\end{equation*}

\begin{equation*}
\begin{aligned}
\tilde{\mathcal L}_2[B_{1, q}] = \frac{24\mu^2}{(|x-q|^2+1)(|x-q|^2+\mu^2)^2}\left(-(x_1-q_1)i-(x_4-q_4)j+(x_3-q_3)k\right),
\end{aligned}
\end{equation*}

\begin{equation*}
\begin{aligned}
\tilde{\mathcal L}_3[B_{1, q}] = \frac{24\mu^2}{(|x-q|^2+1)(|x-q|^2+\mu^2)^2}\left((x_4-q_4)i-(x_1-q_1)j-(x_2-q_2)k\right),
\end{aligned}
\end{equation*}

\begin{equation*}
\begin{aligned}
\tilde{\mathcal L}_4[B_{1, q}] = \frac{24\mu^2}{(|x-q|^2+1)(|x-q|^2+\mu^2)^2}\left(-(x_3-q_3)i+(x_2-q_2)j-(x_1-q_1)k\right).
\end{aligned}
\end{equation*}
Similarly, we have the following computations.

(1) For $$\phi =\Phi_0(x, t) = Im\left((x-\xi(t))\psi^{(0)}(z(\tilde{r}), t)d\bar{x}\right),$$ we have
\begin{equation*}
\begin{aligned}
\tilde{\mathcal L}_1[\phi] = \frac{24\mu^2 f_0\left(\sqrt{|x-\xi|^2+\mu^2(t)}\right)}{\left(|x-\xi|^2+\mu^2\right)^2}\left((x_2-\xi_2)i +  (x_3-\xi_3)j +  (x_4-\xi_4)k\right),
\end{aligned}
\end{equation*}

\begin{equation*}
\begin{aligned}
\tilde{\mathcal L}_2[\phi] = \frac{24\mu^2 f_0\left(\sqrt{|x-\xi|^2+\mu^2(t)}\right)}{\left(|x-\xi|^2+\mu^2\right)^2}\left(-(x_1-\xi_1)i -(x_4-\xi_4)j +  (x_3-\xi_3)k\right),
\end{aligned}
\end{equation*}

\begin{equation*}
\begin{aligned}
\tilde{\mathcal L}_3[\phi] = \frac{24\mu^2 f_0\left(\sqrt{|x-\xi|^2+\mu^2(t)}\right)}{\left(|x-\xi|^2+\mu^2\right)^2}\left((x_4-\xi_4)i - (x_1-\xi_1)j - (x_2-\xi_2)k\right),
\end{aligned}
\end{equation*}

\begin{equation*}
\begin{aligned}
\tilde{\mathcal L}_4[\phi] =  \frac{24\mu^2 f_0\left(\sqrt{|x-\xi|^2+\mu^2(t)}\right)}{\left(|x-\xi|^2+\mu^2\right)^2}\left(-(x_3-\xi_3)i +  (x_2-\xi_2)j - (x_1-\xi_1)k\right)
\end{aligned}
\end{equation*}
with
\begin{equation*}
\begin{aligned}
f_0\left(\sqrt{|x-\xi|^2+\mu^2(t)}\right)
& = \int_{t_0}^{t}2\mu(\tilde{s})\dot{\mu}(\tilde{s})k_1(t-\tilde{s}, z)d\tilde{s} \\
& = 2\int_{t_0}^{t}\mu(\tilde{s})\dot{\mu}(\tilde{s})\frac{1-e^{-\frac{|x-\xi|^2+\mu^2(\tilde{s})}{4(t-\tilde{s})}}(1+\frac{|x-\xi|^2+\mu^2(\tilde{s})}{4(t-\tilde{s})})}{(|x-\xi|^2+\mu^2(\tilde{s}))^2}d\tilde{s}\\
& = 2\int_{t_0}^{t}\frac{\mu(\tilde{s})\dot{\mu}(\tilde{s})}{(t-\tilde{s})^2}\Gamma\left(\frac{\mu(\tilde{s})^2}{t-\tilde{s}}\right)d\tilde{s}.
\end{aligned}
\end{equation*}
Here and in the following $\Gamma(\tau) = \frac{1-e^{-\tau\frac{\rho^2+1}{4}}(1+\tau\frac{\rho^2+1}{4})}{\tau^2(\rho^2+1)^2}$.

(2) For $$\phi = dx\wedge d\bar x\left(\psi^{(12)}(z, t)(x-\xi(t))_2\frac{\partial}{\partial x_2}+\psi^{(34)}(z, t)(x-\xi(t))_4\frac{\partial}{\partial x_3}, \cdot\right),$$ we have
\begin{equation*}
\begin{aligned}
\tilde{\mathcal L}_1[\phi] &= f_{1,1}(x, t)\left(-(x_1-\xi_1)i\right) + f_{1,2}(x, t)\left(-\left(x_4-\xi_4\right)j+(x_3-\xi_3)k\right),
\end{aligned}
\end{equation*}

\begin{equation*}
\begin{aligned}
\tilde{\mathcal L}_2[\phi] &= f_{1,1}(x, t)\left(-(x_2-\xi_2)i\right) + f_{1,2}(x, t)\left(-\left(x_3-\xi_3\right)j-\left(x_4-\xi_4\right)k\right),
\end{aligned}
\end{equation*}

\begin{equation*}
\begin{aligned}
\tilde{\mathcal L}_3[\phi] &= f_{1,1}(x, t)\left(-(x_3-\xi_3)i\right) + f_{1,2}(x, t)\left(\left(x_2-\xi_2\right)j-\left(x_1-\xi_1\right)k\right),
\end{aligned}
\end{equation*}

\begin{equation*}
\begin{aligned}
\tilde{\mathcal L}_4[\phi] &= f_{1,1}(x, t)\left(-(x_4-\xi_4)i\right) + f_{1,2}\left(\left(x_1-\xi_1\right)j+\left(x_2-\xi_2\right)k\right)
\end{aligned}
\end{equation*}
with
\begin{equation*}
\begin{aligned}
f_{1,1}(x, t)
& = \left(4\frac{2\left(2|x-\xi|^2-\mu^2(t)\right)\sqrt{|x-\xi|^2+\mu^2(t)}f_1\left(\sqrt{|x-\xi|^2+\mu^2(t)}\right)}{\left(|x-\xi|^2+\mu^2\right)^{5/2}}\right.\\
&\qquad\qquad\qquad\qquad\qquad\qquad\left. +4\frac{\left(|x-\xi|^2+\mu^2\right)^{2}f_1'\left(\sqrt{|x-\xi|^2+\mu^2(t)}\right)}{\left(|x-\xi|^2+\mu^2\right)^{5/2}}\right),
\end{aligned}
\end{equation*}
\begin{equation*}
\begin{aligned}
f_{1,2}(x, t)
=24\frac{\mu^2(t)f_1\left(\sqrt{|x-\xi|^2+\mu^2(t)}\right)}{\left(|x-\xi|^2+\mu^2\right)^{2}}
\end{aligned}
\end{equation*}
and
\begin{equation*}
\begin{aligned}
f_1\left(\sqrt{|x-\xi|^2+\mu^2(t)}\right)
& = \int_{t_0}^{t}\left(\dot \theta_{12}+\dot \theta_{34}\right)(\tilde{s})k_1(t-\tilde{s}, z)d\tilde{s} \\
& = \int_{t_0}^{t}\left(\dot \theta_{12}+\dot \theta_{34}\right)(\tilde{s})\frac{1-e^{-\frac{|x-\xi|^2+\mu^2(\tilde{s})}{4(t-\tilde{s})}}(1+\frac{|x-\xi|^2+\mu^2(\tilde{s})}{4(t-\tilde{s})})}{(|x-\xi|^2+\mu^2(\tilde{s}))^2}d\tilde{s}\\
& = \int_{t_0}^{t}\frac{\left(\dot \theta_{12}+\dot \theta_{34}\right)(\tilde{s})}{(t-\tilde{s})^2}\Gamma\left(\frac{\mu(\tilde{s})^2}{t-\tilde{s}}\right)d\tilde{s}.
\end{aligned}
\end{equation*}

(3) For $$\phi = dx\wedge d\bar x\left(\psi^{(13)}(z, t)(x-\xi(t))_2\frac{\partial}{\partial x_2}+\psi^{(24)}(z, t)(x-\xi(t))_4\frac{\partial}{\partial x_3}, \cdot\right),$$ we have
\begin{equation*}
\begin{aligned}
\tilde{\mathcal L}_1[\phi] = \frac{24\mu^2 f_2\left(\sqrt{|x-\xi|^2+\mu^2(t)}\right)}{\left(|x-\xi|^2+\mu^2\right)^2}\left(-(x_4-\xi_4)i +  (x_1-\xi_1)j +  (x_2-\xi_2)k\right),
\end{aligned}
\end{equation*}

\begin{equation*}
\begin{aligned}
\tilde{\mathcal L}_2[\phi] = \frac{24\mu^2 f_2\left(\sqrt{|x-\xi|^2+\mu^2(t)}\right)}{\left(|x-\xi|^2+\mu^2\right)^2}\left((x_3-\xi_3)i -(x_2-\xi_2)j +  (x_1-\xi_1)k\right),
\end{aligned}
\end{equation*}

\begin{equation*}
\begin{aligned}
\tilde{\mathcal L}_3[\phi] = \frac{24\mu^2 f_2\left(\sqrt{|x-\xi|^2+\mu^2(t)}\right)}{\left(|x-\xi|^2+\mu^2\right)^2}\left((x_2-\xi_2)i + (x_3-\xi_3)j + (x_4-\xi_4)k\right),
\end{aligned}
\end{equation*}

\begin{equation*}
\begin{aligned}
\tilde{\mathcal L}_4[\phi] =  \frac{24\mu^2 f_2\left(\sqrt{|x-\xi|^2+\mu^2(t)}\right)}{\left(|x-\xi|^2+\mu^2\right)^2}\left(-(x_1-\xi_1)i - (x_4-\xi_4)j + (x_3-\xi_3)k\right)
\end{aligned}
\end{equation*}
with
\begin{equation*}
\begin{aligned}
f_2\left(\sqrt{|x-\xi|^2+\mu^2(t)}\right)
& = \int_{t_0}^{t}\left(\dot \theta_{13}+\dot \theta_{24}\right)(\tilde{s})k_1(t-\tilde{s}, z)d\tilde{s} \\
& = \int_{t_0}^{t}\left(\dot \theta_{13}+\dot \theta_{24}\right)(\tilde{s})\frac{1-e^{-\frac{|x-\xi|^2+\mu^2(\tilde{s})}{4(t-\tilde{s})}}(1+\frac{|x-\xi|^2+\mu^2(\tilde{s})}{4(t-\tilde{s})})}{(|x-\xi|^2+\mu^2(\tilde{s}))^2}d\tilde{s}\\
& = \int_{t_0}^{t}\frac{\left(\dot \theta_{13}+\dot \theta_{24}\right)(\tilde{s})}{(t-\tilde{s})^2}\Gamma\left(\frac{\mu(\tilde{s})^2}{t-\tilde{s}}\right)d\tilde{s}.
\end{aligned}
\end{equation*}

(4) For $$\phi = dx\wedge d\bar x\left(\psi^{(14)}(z, t)(x-\xi(t))_2\frac{\partial}{\partial x_2}+\psi^{(23)}(z, t)(x-\xi(t))_4\frac{\partial}{\partial x_3}, \cdot\right),$$ we have
\begin{equation*}
\begin{aligned}
\tilde{\mathcal L}_1[\phi] =  f_{3,1}(x, t)\left(-\left(x_3-\xi_3\right)i+(x_2-\xi_2)j\right) +f_{3,2}(x, t)\left(-(x_1-\xi_1)k\right),
\end{aligned}
\end{equation*}

\begin{equation*}
\begin{aligned}
\tilde{\mathcal L}_2[\phi] =  f_{3,1}(x, t)\left(\left(x_4-\xi_4\right)i-(x_1-\xi_1)j\right) +f_{3,2}(x, t)\left(-(x_2-\xi_2)k\right),
\end{aligned}
\end{equation*}

\begin{equation*}
\begin{aligned}
\tilde{\mathcal L}_3[\phi] =  f_{3,1}(x, t)\left(\left(x_1-\xi_1\right)i+(x_4-\xi_4)j\right) +f_{3,2}(x, t)\left(-(x_3-\xi_3)k\right),
\end{aligned}
\end{equation*}

\begin{equation*}
\begin{aligned}
\tilde{\mathcal L}_4[\phi] =  f_{3,1}(x, t)\left(-\left(x_2-\xi_2\right)i-(x_3-\xi_3)j\right) +f_{3,2}(x, t)\left(-(x_4-\xi_4)k\right)
\end{aligned}
\end{equation*}
with
\begin{equation*}
\begin{aligned}
f_{3,1}(x, t)
& = 24\frac{\mu^2(t)f_3\left(\sqrt{|x-\xi|^2+\mu^2(t)}\right)}{\left(|x-\xi|^2+\mu^2\right)^{2}},
\end{aligned}
\end{equation*}
\begin{equation*}
\begin{aligned}
f_{3,2}(x, t)
&=\left(4\frac{2\left(2|x-\xi|^2-\mu^2(t)\right)\sqrt{|x-\xi|^2+\mu^2(t)}f_3\left(\sqrt{|x-\xi|^2+\mu^2(t)}\right)}{\left(|x-\xi|^2+\mu^2\right)^{5/2}}\right.\\
&\qquad\qquad\qquad\qquad\qquad\qquad\left. +4\frac{\left(|x-\xi|^2+\mu^2\right)^{2}f_3'\left(\sqrt{|x-\xi|^2+\mu^2(t)}\right)}{\left(|x-\xi|^2+\mu^2\right)^{5/2}}\right)
\end{aligned}
\end{equation*}
and
\begin{equation*}
\begin{aligned}
f_3\left(\sqrt{|x-\xi|^2+\mu^2(t)}\right)
& = \int_{t_0}^{t}\left(\dot \theta_{14}+\dot \theta_{23}\right)(\tilde{s})k_1(t-\tilde{s}, z)d\tilde{s} \\
& = \int_{t_0}^{t}\left(\dot \theta_{14}+\dot \theta_{23}\right)(\tilde{s})\frac{1-e^{-\frac{|x-\xi|^2+\mu^2(\tilde{s})}{4(t-\tilde{s})}}(1+\frac{|x-\xi|^2+\mu^2(\tilde{s})}{4(t-\tilde{s})})}{(|x-\xi|^2+\mu^2(\tilde{s}))^2}d\tilde{s}\\
& = \int_{t_0}^{t}\frac{\left(\dot \theta_{14}+\dot \theta_{23}\right)(\tilde{s})}{(t-\tilde{s})^2}\Gamma\left(\frac{\mu(\tilde{s})^2}{t-\tilde{s}}\right)d\tilde{s}.
\end{aligned}
\end{equation*}

\section*{Acknowledgements}
Y. Sire is partially supported by the Simons foundation and the NSF grant DMS-2154219. J.C. Wei is supported by National Key R\&D Program of China 2022YFA1005602, and Hong Kong General Research Fund “New frontiers in singularity formations of nonlinear partial differential
equations”. Y. Zheng is supported by NSF of China (No. 12171355) and by GuangDong Basic and Applied Basic Research Foundation (2023A1515011800).

\bibliographystyle{plain}

\end{document}